\documentclass[10pt]{article}
 \usepackage{amsmath,amsfonts,amssymb,amsthm}
 \usepackage{hyperref}
 \usepackage{slashed}

 \theoremstyle{plain}
 \newtheorem{theorem}{Theorem}[section]
 \newtheorem{lemma}[theorem]{Lemma}
 \newtheorem{corollary}[theorem]{Corollary}
 \newtheorem{proposition}[theorem]{Proposition}
 
 \newtheorem{remark}[theorem]{Remark}

 \theoremstyle{definition}
 \newtheorem{definition}[theorem]{Definition}
 
 \newtheorem{example}[theorem]{Example}

 \title{G$_2$ manifolds with nodal singularities along circles}

 \author{Gao Chen}

\begin{document}

 \maketitle
\begin{abstract}
  The goal of this paper is the construction of a compact manifold with G$_2$ holonomy and nodal singularities along circles using twisted connected sum method. This paper finds matching building blocks by solving the Calabi conjecture on certain asymptotically cylindrical manifolds with nodal singularities. However, by comparison to the untwisted connected sum case, it turns out that the obstruction space for the singular twisted connected sum  construction is infinite dimensional. By analyzing the obstruction term, there are strong evidences that the obstruction may be resolved if a further gluing is performed in order to get a compact manifold with G$_2$ holonomy and isolated conical  singularities with link $\mathbb{S}^3\times\mathbb{S}^3$.
\end{abstract}
\tableofcontents
\section{Introduction}
The main goal of this paper is to construct a compact manifold with G$_2$ holonomy and nodal singularities along circles using twisted connected sum method. The main motivation is the study of the moduli space of manifolds with G$_2$ holonomy. It was proved by Joyce (Theorem C of \cite{Joyce}) that locally near smooth manifolds with G$_2$ holonomy, the moduli space is a smooth manifold with the third betti number of the original manifold as the dimension. In the plenary talk of the AMS Sectional Meeting in 2016 at Stony Brook, Sir Simon Donaldson listed the global behavior of the moduli space of G$_2$ manifolds as one of the most important problems in the area of manifolds with special holonomies.

This kind of problem was studied in many other dimensions. In dimension 1, it is trivial to say that any compact oriented 1-manifold is diffeomorphic to each other. Moreover, the moduli space of compact oriented Riemannian 1-manifolds is characterized by the cohomology class of the unit-length oriented 1-form. In dimension 2, the topology of a compact Riemann surface is determined by its genus. For each fixed genus, the classical theorem by Torelli \cite{Torelli} says that the non-singular projective algebraic curve is determined by its Jacobian variety, in other words, cohomology class of holomorphic 1-forms. In dimension 4, it was proved that any K3 surface is diffeomorphic to each other \cite{Kodaira} and the moduli space of K3 surfaces is characterized by its cohomology classes of three 2-forms as in Theorem \ref{K3Torelli}.

A K3 surface becomes singular if the nondegeneracy condition in Theorem \ref{K3Torelli} is not satisfied. When approaching the points corresponding to singular K3 surfaces, a typical method is to rescale the singular point and study the bubbling limit \cite{Anderson, Donaldson}. Therefore, the structure of  K3 surfaces is in some sense determined by the structure of non-compact hyperK\"ahler 4-manifolds with decaying curvatures. Such manifolds are called \textit{gravitational instantons}. Under faster than quadratic curvature decay condition, gravitational instantons were classified by the author and Xiuxiong Chen \cite{ChenChen1, ChenChen2, ChenChen3} generalizing previous works of Kronheimer \cite{Kronheimer1, Kronheimer2} and Minerbe \cite{Minerbe}. An important example among them is the ALE-A$_1$ gravitational instanton, which is also known as Eguchi-Hanson space. The starting point is the A$_1$ singularity $z_1^2+z_2^2+z_3^2=0$. There are two ways to resolve it. The first way is to blow up the singular point and the second way is to deform it to $z_1^2+z_2^2+z_3^2=\epsilon$. The key point is that the deformation and the blow up are diffeomorphic to each other. This phenomenon is closely related to the fact that all K3 surfaces are diffeomorphic to each other and the moduli space of K3 surfaces is smooth.

In dimension 6, one can do the similar thing for Calabi-Yau threefolds \cite{CandelasdelaOssa}. In this case, the starting point is the nodal singularity $z_1^2+z_2^2+z_3^2+z_4^2=0$ as in Example \ref{Nodal-singularity}. It can be birationally resolved in two different ways and can also be deformed to $\sum_{j=1}^4 z_j^2=\epsilon$. However, they are not diffeomorphic to each other. Therefore, near the points representing singular manifolds, roughly speaking, the moduli space of Calabi-Yau threefolds looks like the union of several manifolds with possibly different dimensions. The precise statement was proved by Rong and Zhang \cite{RongZhang}. In general the relationship between the birational resolution and the smoothing of a singularity is called an \textit{extremel transition}. The famous Reid's fantasy \cite{Reid} conjectures that all Calabi-Yau threefolds are connected to each other by extremal transitions. This conjecture is still far from reach. In fact, even the precise statement of the conjecture varies \cite{Gross1, Gross2, Rossi} due to the mysterious role played by the possibly non-K\"ahler Calabi-Yau threefolds.

In dimension 7, the fundamental question is whether the dimension of moduli space of G$_2$ manifolds can change or not. This problem is complicated because all the currently known examples of compact singular G$_2$ examples have codimension 4 singularities. According to the work of Joyce \cite{Joyce}, and more recent work of Joyce-Karigiannis \cite{JoyceKarigiannis}, they are related to gravitational instantons.

The main goal of this paper is to instead, construct an example with nodal singularities along circles. The main tool is the twisted connected sum method due to Kovalev \cite{Kovalev} and Corti-Haskins-Nordstr\"om-Pacini \cite{CortiHaskinsNordstromPacini2}. In Section 8.4.5 of \cite{CortiHaskinsNordstromPacini2}, they proposed three technical problems in the construction of singular G$_2$ manifolds using twisted connected sum method. The first problem is the construction of asymptotically cylindrical Calabi-Yau threefolds with nodal singularities. The second problem is finding matching data on the ends. The third problem is to control neck length in the gluing construction. This paper starts from the solution of the first problem by combing the Theorem 1.4 of Hein-Sun's work \cite{HeinSun} with Theorem D of Haskins-Hein-Nordstr\"om's work \cite{HaskinsHeinNordstrom} and proving the following theorem:

\begin{theorem}
Let $i:\mathcal{Z}\subset \mathbb{CP}^{N_1}\times\Delta$ be a flat family of projective varieties over disc $\Delta$. Denote $(\pi_1\circ i)^*(\mathcal{O}(1))$ by $\mathcal{L}$. Denote $(\pi_2\circ i)^{-1}(s)$ by $Z_s$. Denote $\mathcal{L}|_{Z_s}$ by $L_s$.  Suppose that there exists an morphism $f:\mathcal{Z}\rightarrow\mathbb{CP}^1=\mathbb{C}\cup\{\infty\}$. Denote $f^{-1}(\infty)$ by $\mathcal{S}$. Denote $\mathcal{Z}\setminus \mathcal{S}$ by $\mathcal{V}$. Denote $\mathcal{S}\cap Z_s$ by $S_s$. Denote $Z_s\setminus S_s$ by $V_s$. Suppose that $V^{\mathrm{sing}}$ is a finite subset of $V_0$. Suppose that $\mathcal{Z}$ is smooth and the induced map $(\pi_2\circ i)_*$ on the tangent space is surjective at each point on $\mathcal{Z}\setminus V^{\mathrm{sing}}$. Suppose that $df$ is not the pull back of any form on $\Delta$ at each point on $\mathcal{S}$. Suppose that the complex dimension $n$ of $Z_s$ is at least 3. Suppose that for each $x\in V^{\mathrm{sing}}$, there exists a holomorphic function $\epsilon_x(s)$ with $\epsilon_x(0)=0$ such that the germ $(\mathcal{Z},x,\pi_2\circ i)$ is isomorphic to the germ $(\mathcal{C}_x,o_x,\pi_x)$, where \[C_{x,s}=\pi_x^{-1}\mathcal{C}_x(s)=\{z_1^2+...z_n^2=\epsilon_x(s)\},\] and $o_x$ is the tip point of $C_x=C_{x,0}$. Assume that $\Omega_s$ is a meromorphic family of meromorphic $n$-forms on $Z_s$. Assume that $\Omega_s$ is holomorphic on $\mathcal{V}\setminus V^{\mathrm{sing}}$. Assume that $\frac{\Omega_s}{f}$ is holomorphic near $\mathcal{S}$. Assume that the ratio of $\Omega_s$ to $\Omega_{\epsilon_{x}(s)}$ in Example \ref{Stenzel-metric} is holomorphic near $x$. For simplicity, denote $(Z_0,S_0,L_0,\Omega_0,V_0)$ by $(Z,S,L,\Omega,V)$. Then after replacing $\Omega$ by its product with a constant, there exists an asymptotically cylindrical Calabi-Yau metric $\omega\in c_1(L)|_V$ on $V$ such that \[\frac{\omega^n}{n!}=\frac{i^{n^2}}{2^n}\Omega\wedge\bar\Omega.\] Moreover, $\omega$ has conical singularity with rate $\lambda>0$ and tangent cone $(C_x,\omega_{C_x})$ at $x$ as in Definition \ref{Definition-conical-singularity}.
\label{Non-compact-Calabi-Yau}
\end{theorem}

Then the second problem is solved for a particular example using the additional information about the quartic K3 surfaces in $\mathbb{CP}^3$ using Theorem \ref{K3Torelli} and Chapter 3 of \cite{ReidK3}.

\begin{proposition}
It is possible to find the following data with required properties:

(1) $X_-=\mathbb{CP}^3$. $\Pi$ is a 2-plane in $\mathbb{CP}^4$. $X_+$ is a quartic 3-fold in $\mathbb{CP}^4$ containing $\Pi$ with nine nodal singularities $X_+^{\mathrm{sing}}$.

(2) $|S_{0,\pm},S_{\infty,\pm}|\subset|-K_{X_{\pm}}|$ are pencils with smooth base locuses $C_{\pm}$ disjoint with $X_+^{\mathrm{sing}}$. $Z_{\pm}$ are the blow-up of $X_{\pm}$ at $C_{\pm}$.

(3) $S_{\pm}$ are smooth K3 surfaces in $|S_{0,\pm},S_{\infty,\pm}|$ disjoint with $X_+^{\mathrm{sing}}$. Their proper transforms are also denoted by $S_{\pm}\subset Z_{\pm}$. $\Omega_{\pm}$ are meromorphic 3-forms on $Z_{\pm}$ with simple poles along $S_{\pm}$.

(4) $(S_{\pm},\omega_{S_\pm},\omega^{J}_{S_\pm}+i\omega^{K}_{S_\pm})$ are Calabi-Yau surfaces.

(5) $\omega_{\pm}$ are asymptotically cylinderical Calabi-Yau metrics on $V_{\pm}=Z_{\pm}\setminus S_{\pm}$ with $\frac{\omega_{\pm}^3}{6}=\frac{i}{8}\Omega_{\pm}\wedge\bar\Omega_{\pm}$.

(6) $\omega_+$ has conical singularity in the sense of Definition \ref{Definition-conical-singularity} with the nodal singularity in Example \ref{Nodal-singularity} as the tangent cone for all $x\in V_+^{\mathrm{sing}}$.

(7) $K_{\pm}$ are compact subsets of $V_{\pm}$. $P_{\pm}:[1,\infty)\times \mathbb{S}^1\times S_{\pm}\rightarrow V_{\pm}\setminus K_{\pm}$ are diffeomorphisms on the ends. $t_\pm$ are coordinates on $[1,\infty)$, $\vartheta_\pm$ are coordinates on $\mathbb{S}^1$. Up to exponentially decaying errors $\varrho_{\pm}$ and $\varsigma_{\pm}$, \[(P_{\pm}^{*}\omega_{\pm},P_{\pm}^{*}\Omega_{\pm})=(\omega_{\infty,\pm},\Omega_{\infty,\pm}) +(d\varrho_{\pm},d\varsigma_{\pm}),\]
where \[\omega_{\infty,\pm}=dt_{\pm}\wedge d\vartheta_{\pm}+\omega_{S_{\pm}},\] and \[\Omega_\infty=(d\vartheta_{\pm}-idt_{\pm})\wedge(\omega^{J}_{S_{\pm}}+i\omega^{K}_{S_{\pm}}).\]

(8) $r$ is a diffeomorphism from $(S_{+},\omega_{S_+},\omega^{J}_{S_+},\omega^{K}_{S_+})$ to $(S_{-},\omega^{J}_{S_-},\omega_{S_-},-\omega^{K}_{S_-})$.
\label{Matching}
\end{proposition}

Extend $t_\pm$ to non-negative smooth functions on $V_\pm$ such that $t_+$ equals to 0 near $V^{\mathrm{sing}}_+=X^{\mathrm{sing}}_+$. Choose $\chi=\chi(s):\mathbb{R}\rightarrow[0,1]$ as a smooth function satisfying $\chi(s)=1$ for $s\le 1$ and $\chi(s)=0$ for $s\ge 2$. Using the data in Proposition \ref{Matching}, as in Section 3 of \cite{CortiHaskinsNordstromPacini2}, for fixed large enough $T$, define \[\omega_{T,\pm}=\omega_{\pm}-d((1-\chi(t_{\pm}-T+2))\varrho_{\pm})\]
and \[\Omega_{T,\pm}=\Omega_{\pm}-d((1-\chi(t_{\pm}-T+2))\varsigma_{\pm})\] on $V_{\pm}$. Let $M_{\pm}$ be $\mathbb{S}^1\times V_{\pm}$. Let $\theta_{\pm}$ be the coordinates on $\mathbb{S}^1$. Define
\[\varphi_{T,\pm}=d\theta_{\pm}\wedge\omega_{T,\pm}+\mathrm{Re}\Omega_{T,\pm}.\]
Remark that using the diffeomorphism \[(\theta_-,t_-,\vartheta_-,x_-)=(\vartheta_+,2T+1-t_+,\theta_+,r(x_+)),\]
$\varphi_{T,\pm}$ can be glued into a closed G$_2$ structure $\varphi_T$.

Let $M$ be the manifold obtained by this gluing map, then the construction of a G$_2$ manifold with nodal singularities along circles is reduced to finding a perturbation of $\varphi_T$ which induces a metric with G$_2$ holonomy but still preserves the singularities. However, the analysis on manifolds with conical singularities along smooth submanifolds is very complicated. A slightly simpler problem is the analysis on manifolds with isolated conical singularities. Therefore, this paper starts from solving an analogy problem instead. In this case, $(Z_-,\omega_-,\Omega_-)$ is $(Z_+,\omega_+,-\Omega_+)$ and $r$ is the identity map. So $r$ is a diffeomorphism from $(S_{+},\omega_{S_+},\omega^{J}_{S_+},\omega^{K}_{S_+})$ to $(S_{-},\omega_{S_-},-\omega^{J}_{S_-},,-\omega^{K}_{S_-})$ instead and the gluing map is given by \[(\theta_-,t_-,\vartheta_-,x_-)=(\theta_+,2T+1-t_+,-\vartheta_+,r(x_+)).\] Remark that in this case, the $\mathbb{S}^1$ factor with coordinate $\theta=\theta_-=\theta_+$ is global. So the theorem in this case is

\begin{theorem}
(Doubling construction of Calabi-Yau threefolds)

For the new choice of gluing data, for sufficiently large $T$, there is an $\mathbb{S}^1$-invariant perturbation $\varphi$ of $\varphi_T$ such that the holonomy group of $\varphi$ is contained in $\mathrm{SU}(3)\subset\mathrm{G}_2$ and for each $x\in V_{\pm}^{\mathrm{sing}}$, there exist numbers $c_{1,x}>0$, $c_{2,x}>0$, $c_{3,x}$ and a homeomorphism $P_x:O_x\rightarrow U_x$ between a neighborhood $o\in O_x\subset C_x$ and $x\in U_x\subset V_{\pm}$, such that
\[|\nabla^j_{\varphi_{\mathbb{S}^1\times C_x}}((\mathrm{Id}\times P_x)^*\varphi-\varphi_{\mathbb{S}^1\times C_x})|_{\varphi_{\mathbb{S}^1\times C_x}}=O(r^{\lambda-j})\] as $r\rightarrow 0$ for a positive number $\lambda$ and all $j\in\mathbb{N}_0$, where \[\varphi_{\mathbb{S}^1\times C_x}=c_{1,x}d\theta\wedge\omega_{C_x}+c_{2,x}\mathrm{Re}(e^{ic_{3,x}}\Omega_{C_x}).\]
\label{Doubling-construction-of-Calabi–Yau-threefolds}
\end{theorem}

To emphasize that there is a global $\mathbb{S}^1$ factor, in this case, it is better to use $\mathbb{S}^1\times M$ instead of $M$ to denote the gluing of $\mathbb{S}^1\times V_{\pm}$.

Remark that the non-singular version of Theorem \ref{Doubling-construction-of-Calabi–Yau-threefolds} was proved by Doi and Yotsutani \cite{DoiYotsutani}.

A large portion of the proof of Theorem \ref{Doubling-construction-of-Calabi–Yau-threefolds} is inspired by the work of Karigiannis-Lotay \cite{KarigiannisLotay} and has an analogy in \cite{KarigiannisLotay}. The main tool of the proof of Theorem \ref{Doubling-construction-of-Calabi–Yau-threefolds} is the weighted analysis developed by Lockhart-McOwen \cite{LockhartMcOwen} and Melrose-Mendoza \cite{MelroseMendoza} independently and further refined by many people. One of the key points is Theorem \ref{Polyhomogenous-L-totally-character} when the weight changes. Another key point is the study of harmonic forms on the nodal singularity.

Back to the singular twisted connected sum case. It involves weighted analysis for manifolds with edge singularities. It was pioneered by Mazzeo \cite{Mazzeo} and followed by many people. In this paper, the analogy of Theorem \ref{Polyhomogenous-L-totally-character} is proved. However, the obstruction space in this case is infinite dimensional.

In personal discussions with the author, Sir Simon Donaldson and Edward Witten conjectured that the nodal singularities along circles may be replaced by isolated conical singularities with the homogenous space $(\mathrm{SU}(2)\times \mathrm{SU}(2)\times \mathrm{SU}(2))/\mathrm{SU}(2)$ as the link. As pointed out by Atiyah-Witten \cite{AtiyahWitten}, there are three ways of resolving the cone over $(\mathrm{SU}(2)\times \mathrm{SU}(2)\times \mathrm{SU}(2))/\mathrm{SU}(2)$. In this paper, by analyzing the infinite dimensional obstruction space, there are strong evidences that this conjecture is correct. It is left for future studies.

The basic facts about G$_2$ structures, the nodal singularity and the weighted analysis are reviewed in Section 2. Theorem \ref{Non-compact-Calabi-Yau} is proved in Section 3. Proposition \ref{Matching} is proved in Section 4. Harmonic forms on the nodal singularity is studied in Section 5. Theorem \ref{Doubling-construction-of-Calabi–Yau-threefolds} is proved in Section 6. The analogy of the refined change of index formula for singular twisted connected sum, the fact that the obstruction space is infinite dimensional as well as the conjectural picture are discussed in Section 7.

\noindent{\bf Acknowledgement:}  The author would like to thank Edward Witten for introducing him to this problem and for many fruitful discussions. The author is also grateful to the helpful conversations with Jeff Cheeger, Xiuxiong Chen, Sir Simon Donaldson, Lorenzo Foscolo, Mark Haskins, Hans-Joachim Hein, Helmut Hofer, Fanghua Lin, Rafe Mazzeo, Johannes Nordstr\"om, Song Sun, Akshay Venkatesh, Jeff Viaclovsky and Ruobing Zhang. This material is based upon work supported by the National Science Foundation under Grant No. 1638352, as well as support from the S. S. Chern Foundation for Mathematics Research Fund.

\section{Preliminaries}
\begin{definition}
A G$_2$ structure on a 7-dimensional manifold $M$ is defined by a 3-form $\varphi$ such that at each point there exists an element in $\mathrm{GL}(7,\mathbb{R})$ which maps $\varphi$ to $e^{123} + e^{145} + e^{167} + e^{246} -e^{257} -e^{347} -e^{356}$, where $e^{ijk} = e^i \wedge e^j \wedge e^k$ and $\{e^i\}$ are the standard basis of $T^*M$. It induces a metric $g_\varphi$ by
\[g_\varphi(u,v)\mathrm{Vol}_{g_\varphi} =\frac{1}{6}(u\lrcorner\varphi)∧(v\lrcorner\varphi)\wedge\varphi.\]
\end{definition}

\begin{definition}
The G$_2$ structure provides a $g$-orthogonal decomposition of forms on $M$. In particular, for three forms,
$\Omega^3(M) = \Omega^3_1(M) \oplus \Omega^3_7(M) \oplus \Omega^3_{27}(M)$,
where $\Omega^3_1(M) = \{f\varphi\}$, $\Omega^3_7(M)=\{X\lrcorner*_{g_\varphi}\varphi\}$ and the orthogonal complement is $\Omega^3_{27}(M)$.
\end{definition}

\begin{proposition}
(Lemma 3.1.1 of \cite{Joyce})
Denote $*_\varphi\varphi$ by $\Theta(\varphi)$, then using the metric induced by $\varphi$,
\[\Theta(\varphi+\gamma)=*\varphi+\frac{4}{3}\pi_1(\gamma)+*\pi_7(\gamma)-*\pi_{27}(\gamma)-Q(\gamma),\]
where $\pi_1, \pi_7$ and $\pi_{27}$ are the orthogonal projection to $\Omega^3_1$, $\Omega^3_7$ and $\Omega^3_{27}$, and $Q$ is the higher order term satisfying the estimates in Lemma 3.1.1 of \cite{Joyce}.
\label{Theta-function}
\end{proposition}

\begin{theorem}
(\cite{FernandezGray}) The holonomy group of a metric $g$ is contained in G$_2$ if and only if $g$ is induced by a G$_2$ structure $\varphi$ satisfying $d\varphi=d\Theta(\varphi)=0$.
\end{theorem}

Therefore, it suffices to consider the moduli space of G$_2$ structures $\varphi$ satisfying $d\varphi=d\Theta(\varphi)=0$.

\begin{definition}
Suppose that $M$ is a K\"ahler manifold with complex dimension $n$. Suppose that $\omega$ is a K\"ahler form and $\Omega$ is a holomorphic $n$-form. Then $(M,g,J,\omega,\Omega)$ is called a Calabi-Yau $n$-fold if \[\frac{\omega^n}{n!}=\frac{i^{n^2}}{2^n}\Omega\wedge\bar\Omega\]
\end{definition}

\begin{example}
(\cite{Stenzel})

Let $C_\epsilon=\{z_1^2+...z_{n+1}^2=\epsilon\}\subset\mathbb{C}^{n+1}$. When $\epsilon\not=0$, up to scaling, the unique $\mathrm{SO}(n+1)$-invariant asymptotically conical Calabi-Yau metric $g$ on $C_\epsilon$ is given by \[\omega_s=\frac{i}{2}\partial\bar\partial |\epsilon|^{\frac{n-1}{n}}(f(\cosh^{-1}(\frac{|z_1|^2+...|z_{n+1}|^2}{|\epsilon|}))),\] where $(f'(w)^n)'=n(\sinh w)^{n-1}$, and $f(0)=f'(0)=0$.
When $\epsilon=0$, \[\omega_0=\frac{i}{2}\partial\bar\partial(\frac{n}{n-1})^{\frac{n+1}{n}}(|z_1|^2+...|z_{n+1}|^2)^{\frac{n-1}{n}}.\]
Define \[\Omega_\epsilon=\frac{dz_1dz_2...dz_n}{z_{n+1}},\] then by direct calculation,
\[\frac{\omega_\epsilon^n}{n!}=\frac{i^{n^2}}{2^n}\Omega_\epsilon\wedge\bar\Omega_\epsilon.\]
\label{Stenzel-metric}
\end{example}

\begin{example}
Let $C$ be the nodal singularity $\{z_1^2+z_2^2+z_3^2+z_4^2=0\}\subset\mathbb{C}^4$. Then \[\omega_C=\frac{i}{2}\partial\bar\partial(\frac{3}{2})^{\frac{4}{3}}(|z_1|^2+...|z_4|^2)^{\frac{2}{3}}\] and \[\Omega_C=\frac{dz_1dz_2dz_3}{z_4}\] satisfy $\frac{\omega_\epsilon^3}{6}=\frac{i}{8}\Omega_C\wedge\bar\Omega_C$ and therefore define a Calabi-Yau cone structure on $C$.

 The nodal singularity $\{(z_1+iz_2)(z_1-iz_2)+(z_3+iz_4)(z_3-iz_4)=0\}$ is birationally equivalent to its small resolution
 \[\{z_j\in\mathbb{C}, z\in\mathbb{C}\cup\{\infty\}:z_1^2+z_2^2+z_3^2+z_4^2=0, z=\frac{z_1+iz_2}{z_3+iz_4}=-\frac{z_3-iz_4}{z_1-iz_2}\}.\]
 It replaces the tip point by $\mathbb{CP}^1=\mathbb{C}\cup\{\infty\}$.
 The other small resolution is given by \[\{z_j\in\mathbb{C}, z\in\mathbb{C}\cup\{\infty\}:z_1^2+z_2^2+z_3^2+z_4^2=0, z=\frac{z_1+iz_2}{z_3-iz_4}=-\frac{z_3+iz_4}{z_1-iz_2}\}.\]
 Both of the small resolutions are Calabi-Yau threefolds by \cite{CandelasdelaOssa}.

 It is easy to see that $C$ is diffeomorphic to the cone over $\mathbb{S}^2\times\mathbb{S}^3$. The deformation $\{z_1^2+...z_4^2=\epsilon\}$ is diffeomorphic to $\mathbb{R}^3\times\mathbb{S}^3$. Both of the small resolutions are diffeomorphic to $\mathbb{S}^2\times\mathbb{R}^4$.
\label{Nodal-singularity}
\end{example}

The following proposition is well-known:
\begin{proposition}
For any Calabi-Yau threefold $(M,g,J,\omega,\Omega)$, define an $\mathbb{S}^1$-invariant G$_2$ structure $\varphi$ on $\mathbb{S}^1\times M$ by $\varphi=d\theta\wedge\omega+\mathrm{Re}\Omega$, where $\theta$ is the standard coordinate on $\mathbb{S}^1$. It satisfies $d\varphi=d\Theta(\varphi)=0$. On the other hands, any $\mathbb{S}^1$-invariant G$_2$ structure on $\mathbb{S}^1\times M$ satisfying $d\varphi=d\Theta(\varphi)=0$ must comes from a Calabi-Yau threefold structure on $M$.
\end{proposition}

The next part is the definition of an asymptotically cylindrical K\"ahler manifold with conical singularities.

\begin{definition}
Let $(F,g_F)$ be a compact Riemannian manifold. Then $C(F)$ is the set $((0,\infty)\times F)\cup\{o\}$. Let $r$ be the coordinate on $(0,\infty)$ and define $r(o)=0$. Then the cone metric $g_{C(F)}$ is defined by $g_{C(F)}=dr^2+r^2g_F$. Similarly, for any compact Riemannian manifold $(F_\infty,g_{F_\infty})$, define the product metric on $\mathbb{R}\times F_\infty$ by $g_\infty=dt^2+g_{F_\infty}$, where $t$ is the coordinate on $\mathbb{R}$. The K\"ahler structures $J_{C(F)}, J_\infty$, the K\"ahler forms $\omega_{C(F)},\omega_\infty$, the $(n,0)$-forms $\Omega_{C(F)}, \Omega_\infty$, and the G$_2$ structures $\varphi_{C(F)}, \varphi_\infty$ on $C(F)$ or $\mathbb{R}\times F_\infty$ are defined similarly.
\end{definition}

\begin{definition}
A Calabi-Yau cone $C$ with smooth cross-section and with Ricci-flat K\"ahler cone metric $\omega_C=\frac{i}{2}\partial\bar\partial r^2$ is regular if its Reeb field, i.e. the holomorphic Killing field $J(r\frac{\partial}{\partial r})$, generates a free $\mathbb{S}^1$-action on $C\setminus\{o\}$. This exhibits $C$ as the blow-down of the zero section of $\frac{1}{q}K_B$ for some K\"ahler-Einstein Fano manifold $B$ and $q\in\mathbb{N}$. C is called strongly regular if $-\frac{1}{q}K_B$ is very ample.
\label{Definition-strongly-regular-cone}
\end{definition}
\begin{remark}
Example \ref{Stenzel-metric} is strongly regular.
\end{remark}

\begin{definition}
Let $V$ be a manifold with K\"ahler metric $\omega$. $x\in V$ is called a conical singularity with rate $\nu_x>0$ and tangent cone $(C_x,\omega_{C_x})$ with respect to $\omega$ if there exist a K\"ahler metric cone $(C_x,J_{C_x},\omega_{C_x})$ and a biholomorphism $P_x:O_x\rightarrow U_x$ between neighborhoods $o\in O_x\subset C_x$ and $x\in U_x\subset V$ such that \[|\nabla^j_{\omega_{C_x}}(P^*\omega-\omega_{C_x})|_{\omega_{C_x}}=O(r^{\nu_x-j})\] as $r\rightarrow 0$ for all $j\in\mathbb{N}_0$. Assume that the set $\{r\le r_{0,x}\}$ is contained in $O_x$.
\label{Definition-conical-singularity}
\end{definition}

\begin{definition}
Let $V$ be a manifold with a metric $g$. $g$ is called asymptotically cylindrical with rate $\nu_\infty>0$ and cross-section $F_\infty$ if there exist a set $U_\infty$ and a diffeomorphism $P_\infty:[1,\infty)\times F_\infty \rightarrow U_\infty$ such that $V\setminus U_\infty$ is bounded and \[|\nabla^j_{g_\infty}(P^*g-g_\infty)|_{g_\infty}=O(e^{-\nu_\infty t})\] as $t\rightarrow \infty$ for all $j\in\mathbb{N}_0$. The asymptotically cylindrical almost complex structure $J$, K\"ahler form $\omega$, $(n,0)$-form $\Omega$ and G$_2$ structure $\varphi$ are defined similarly.
\end{definition}

The next goal is to describe the analysis on an asymptotically cylindrical K\"ahler manifold $V$ with conical singularities following Lockhart-McOwen \cite{LockhartMcOwen}. Remark that in this paper, $\delta$ is multiplied by $-1$ and $\lambda$ is divided by $i$ compared to \cite{LockhartMcOwen}. The same result was obtained by Melrose and Mendoza \cite{MelroseMendoza} independently.

\begin{definition}
Assume that $V$ is asymptotically cylindrical with conical singularities at $V^{\mathrm{sing}}$. Assume that $U_x$ and $U_\infty$ are disjoint. Let $N_2$ be the number of points in $V^{\mathrm{sing}}$. Assume that $\nu>0$ is smaller than the minimum of $\{\nu_1,...,\nu_{N_2},\nu_\infty\}$. For any $x\in V^{\mathrm{sing}}$, choose $r_x$ as a smooth function with range $[0,2r_{0,x}]$ such that $r_x=2r_{0,x}$ outside $U_x$ and $r_x=r$ when $r\le r_{0,x}$. Extend $t$ to a non-negative smooth function on $V$ such that $t$ equals to 0 on $U_x$ for all $x\in V^{\mathrm{sing}}$. For any $\delta=(\delta_1,...\delta_{N_2},\delta_\infty)\in \mathbb{R}^{N_2+1}$, using the metric $\omega$, define the weighted $L^2$ space $L^2_{\delta}$ by \[||\gamma||_{L^2_{\delta}}=(\int_V |\prod_{i=1}^{N_2}r_{x_i}^{-\delta_i}e^{\delta_\infty t}\gamma|^2\prod_{i=1}^{N_2}r_{x_i}^{-2n})^{\frac{1}{2}},\]
where $n$ is the complex dimension of $V$.
Assume that $k$ is a large enough interger. Define the weighted Hilbert space by
\[||\gamma||_{W^{k,2}_{\delta}} =(\sum_{j=0}^{k}||\nabla^j\gamma||^2_{L^2_{(\delta_1-j,...\delta_{N_2}-j,\delta_\infty)}})^{\frac{1}{2}}.\]
Roughly speaking, it means $\gamma$ has rate $r_{x_i}^{\delta_i}$ near each $x_i\in V^{\mathrm{sing}}$ and rate $e^{-\delta_\infty t}$ on the end. In general, one can define $W^{k,p}_{\delta}$ and $C^{k,\alpha}_{\delta}$ spaces. However, the $W^{k,2}_{\delta}$ space is enough for this paper. Define the space $C^{\infty}_{\delta}$ as the intersection of $W^{k,2}_{\delta}$ for all $k$. Choose $\chi=\chi(s):\mathbb{R}\rightarrow[0,1]$ as a smooth function satisfying $\chi(s)=1$ for $s\le 1$ and $\chi(s)=0$ for $s\ge 2$. Denote $\chi(\frac{2r_{x_i}}{r_{0,x_i}})$ by $\chi_i$. Denote $1-\chi(t)$ by $\chi_\infty$. It is also useful to consider spaces like \[(\oplus_{i=1}^{N_2}\mathbb{R}\chi_i)\oplus W^{k,2}_{\delta}\] for small positive $\delta_i$. For constants $c_1, ... c_{N_2}$, define
\[||\sum_{i=1}^{N_2}c_i\chi_i+\gamma|| _{(\oplus_{i=1}^{N_2}\mathbb{R}\chi_i)\oplus W^{k,2}_{\delta}}=\sum_{i=1}^{N_2}|c_i|+||\gamma||_{W^{k,2}_{\delta}}.\]
\label{Definition-weighted-Sobelov}
\end{definition}

Now let $D$ be the Laplacian operator $\Delta=dd^*+d^*d$ or the operator $d+d^*$ acting on the direct sum of all odd degree forms. Then the order $m$ of $D$ is 2 or 1 respectively.
\begin{definition}
$\lambda_i\in\mathbb{C}$ is called a critical rate for $D$ near $x_i$ if there exists \[\gamma=\sum_{p=0}^{p_{i,\lambda_i}} e^{\lambda_i\log r}(-\log r)^p \gamma_{i,\lambda_i,p}\] in the kernel of $D_{C_{x_i}}$. $\lambda_\infty\in\mathbb{C}$ is called a critical rate for $D$ near infinity if there exists \[\gamma=\sum_{p=0}^{p_{\infty,\lambda_\infty}} e^{-\lambda_\infty t}t^p \gamma_{\infty,\lambda_\infty,p}\] in the kernel of $D_{[0,\infty)\times F_\infty}$. Define $\mathcal{K}_i(\lambda_i)$ or $\mathcal{K}_\infty(\lambda_\infty)$ as the space of such $\gamma$. Define the multiplicity $d_i(\lambda_i)$ or $d_\infty(\lambda_\infty)$ as the dimension of $\mathcal{K}_i(\lambda_i)$ or $\mathcal{K}_\infty(\lambda_\infty)$. $\delta_i\in\mathbb{R}$ or $\delta_\infty\in\mathbb{R}$ is called critical near $x_i$ or infinity if it is the real part of a critical $\lambda_i$ or $\lambda_\infty$. $\delta$ is called critical if either at least one of the $\delta_i$ is critical near $x_i$ or $\delta_\infty$ is critical near infinity. For non-critical weights $\delta_i<\delta'_i$, define $N(\delta,\delta')$ by \[N(\delta,\delta')=\sum_{i=1}^{N_2}\sum_{\delta_i<\mathrm{Re}(\lambda''_i)<\delta_i}d_i(\lambda''_i)+
\sum_{\delta_\infty<\mathrm{Re}(\lambda''_\infty)<\delta'_\infty}d_\infty(\lambda''_\infty).\]
\label{Definition-critical}
\end{definition}

The main theorem of Lockhart-McOwen \cite{LockhartMcOwen} is the following:

\begin{theorem}
$D:W^{k,2}_{\delta}\rightarrow W^{k-m,2}_{(\delta_1-m,...\delta_{N_2}-m,\delta_\infty)}$ is Fredholm if and only if $\delta$ is non-critical. Moreover, the Fredholm index $i_{\delta}(\Delta)$ is independent of $k$. For non-critical weights $\delta_i<\delta'_i$ and $\delta_\infty<\delta'_\infty$, $i_{\delta}(\Delta)-i_{\delta'}(\Delta)=N(\delta,\delta')$.
\label{Lockhart-McOwen}
\end{theorem}

In the first paragraph of page 420 of \cite{LockhartMcOwen}, Lockhart-McOwen used the following theorem of \cite{Kondratev} and \cite{MazjaPlamenevskii}:

\begin{theorem}
The operator $D_{C_{x_i}}:W^{k,2}_{\delta_i}(C_{x_i})\rightarrow W^{k-m,2}_{\delta_i-m}(C_{x_i})$ or the operator $D_\infty:W^{k,2}_{\delta_\infty}(\mathbb{R}\times F_\infty)\rightarrow W^{k-m,2}_{\delta_\infty}(\mathbb{R}\times F_\infty)$ is an isomorphism for non-critical $\delta_i$ or non-critical $\delta_\infty$.
\label{Invertible-IL}
\end{theorem}

In Section 5 of \cite{LockhartMcOwen}, Lockhart-McOwen used the following theorem of \cite{AgmonNirenberg}, \cite{Kondratev} and \cite{MazjaPlamenevskii}:

\begin{theorem}
Suppose that $\delta_i>\delta'_i$ or $\delta_\infty>\delta'_\infty$ are non-critical, then for any $\gamma\in W^{k,2}_{\delta'_i}(C_{x_i}\cap\{r\le r_{0,x_i}\})$ or $\gamma\in W^{k,2}_{\delta'_\infty}([1,\infty)\times F_\infty)$ satisfying $D_{C_{x_i}}\gamma=0$ or $D_\infty \gamma=0$, there exists $\gamma'$ in the direct sum of $\mathcal{K}_i(\lambda''_i)$ for all $\delta_i'<\mathrm{Re}(\lambda''_i)<\delta_i$ or $\mathcal{K}_\infty(\lambda''_\infty)$ for all $\delta_\infty'<\mathrm{Re}(\lambda''_\infty)<\delta_\infty$ such that $\gamma-\gamma'\in W^{k,2}_{\delta_i}(C_{x_i}\cap\{r\le r_{0,x_i}\})$ or $W^{k,2}_{\delta_\infty}([1,\infty)\times F_\infty)$.
Moreover,
\[||\gamma||_{(\oplus_{\delta'_i<\mathrm{Re}(\lambda''_i)<\delta_i}\chi_i\mathcal{K}_i(\lambda''_i))
\oplus W^{k,2}_{\delta_i}}\le C||\gamma||_{W^{k,2}_{\delta'_i}}\] or
\[||\gamma||_{(\oplus_{\delta'_\infty<\mathrm{Re}(\lambda''_\infty)<\delta_\infty} \chi_\infty\mathcal{K}_\infty(\lambda''_\infty))
\oplus W^{k,2}_{\delta_\infty}}\le C||\gamma||_{W^{k,2}_{\delta'_\infty}}.\]
\label{Polyhomogenous-IL}
\end{theorem}

Choose a non-critical weight $\delta_i\in(\mathrm{Re}(\lambda_i),\mathrm{Re}(\lambda_i)+\nu)$ or a non-critical weight $\delta_\infty\in(\mathrm{Re}(\lambda_\infty),\mathrm{Re}(\lambda_\infty)+\nu)$. By Theorem \ref{Invertible-IL}, the maps \[D_{C_{x_i}}:W^{k,2}_{\delta_i}(C_{x_i})\rightarrow W^{k-m,2}_{\delta_i-m}(C_{x_i})\] and \[D_\infty:W^{k,2}_{\delta_\infty}(\mathbb{R}\times F_\infty)\rightarrow W^{k-m,2}_{\delta_\infty}(\mathbb{R}\times F_\infty)\] are isomorphisms. Therefore, after shrinking $r_{0,x_i}$ or replacing $t$ by $t+T$ if necessary, the maps \[\chi_i D+(1-\chi_i)D_{C_{x_i}}:W^{k,2}_{\delta_i}(C_{x_i})\rightarrow W^{k-m,2}_{\delta_i-m}(C_{x_i})\] and \[\chi_\infty D+(1-\chi_\infty)D_\infty:W^{k,2}_{\delta_\infty}(\mathbb{R}\times F_\infty)\rightarrow W^{k-m,2}_{\delta_\infty}(\mathbb{R}\times F_\infty)\] are also isomorphisms. Remark that for each $\gamma\in\mathcal{K}_i(\lambda_i)$ or $\gamma\in\mathcal{K}_\infty(\lambda_\infty)$, $-D(\gamma\chi_i)\in W^{k-m,2}_{\delta_i-m}$ or $-D(\gamma\chi_\infty)\in W^{k-m,2}_{\delta_\infty}$. One can use $\chi_iD+(1-\chi_i)D_{C_{x_i}}$ or $\chi_\infty D+(1-\chi_\infty)D_\infty$ to revert it. Then the inverse image plus $\gamma$ lies in kernel of $D$ restricted to the set where $\chi_i$ or $\chi_\infty$ is identically 1. Denote the space of such sums by $\mathcal{P}_i(\lambda_i)$ or $\mathcal{P}_\infty(\lambda_\infty)$. Remark that the definition of $\mathcal{P}_i(\lambda_i)$ or $\mathcal{P}_\infty(\lambda_\infty)$ can be changed to any set such that $D\gamma=0$ near $x_i$ or infinity for any $\gamma$ in $\mathcal{P}_i(\lambda_i)$ or $\mathcal{P}_\infty(\lambda_\infty)$ and there is a bijection between $\mathcal{K}_i(\lambda_i)$ and $\mathcal{P}_i(\lambda_i)$ or between $\mathcal{K}_\infty(\lambda_\infty)$  and $\mathcal{P}_\infty(\lambda_\infty)$ using their asymptotic behaviours.

The following theorem can be proved using Theorem \ref{Invertible-IL} and Theorem \ref{Polyhomogenous-IL}:

\begin{theorem}
Suppose that $\delta_i>\delta'_i$ or $\delta_\infty>\delta'_\infty$ are non-critical, then for any $\gamma\in W^{k,2}_{\delta'_i}(C_{x_i}\cap\{r\le r_{0,x_i}\})$ satisfying $D\gamma\in W^{k-m,2}_{\delta_i-m}$ or $\gamma\in W^{k,2}_{\delta'_\infty}([1,\infty)\times F_\infty)$ satisfying $D\gamma\in W^{k-m,2}_{\delta_\infty}$, there exists $\gamma'$ in the direct sum of $\chi_i\mathcal{P}_i(\lambda''_i)$ for all $\delta'_i<\mathrm{Re}(\lambda''_i)<\delta_i$ or $\chi_\infty\mathcal{P}_\infty(\lambda''_\infty)$ for all critical $\delta'_\infty<\mathrm{Re}(\lambda''_\infty)<\delta_\infty$ such that $\gamma-\gamma'\in W^{k,2}_{\delta_i}(C_{x_i}\cap\{r\le r_{0,x_i}\})$ or $W^{k,2}_{\delta_\infty}([1,\infty)\times F_\infty)$.
Moreover,
\[||\gamma||_{(\oplus_{\delta'_i<\mathrm{Re}(\lambda''_i)<\delta_i}\chi_i\mathcal{P}_i(\lambda''_i))
\oplus W^{k,2}_{\delta_i}}\le C(||\gamma||_{W^{k,2}_{\delta'_i}}+||D\gamma||_{W^{k-m,2}_{\delta_i-m}})\] or
\[||\gamma||_{(\oplus_{\delta'_\infty<\mathrm{Re}(\lambda''_\infty)<\delta_\infty} \chi_\infty\mathcal{P}_\infty(\lambda''_\infty))
\oplus W^{k,2}_{\delta_\infty}}\le C(||\gamma||_{W^{k,2}_{\delta'_\infty}}+||D\gamma||_{W^{k-m,2}_{\delta_\infty}}).\]
\label{Polyhomogenous-L-totally-character}
\end{theorem}
\begin{proof}
The proof of this theorem is essentially due to Lockhart-McOwen \cite{LockhartMcOwen}. There were lots of closely related theorems due to many authors, for example, Proposition 4.21 of \cite{KarigiannisLotay} and Proposition 2.9 of \cite{HeinSun}. For the reader's convenience, a proof is included here without claiming any originality.

Suppose that this theorem is not true. Choose $\delta_i$ as a non-critical value for the failure of this theorem such that it is smaller than the infimum of all such $\delta_i$ plus $\frac{\nu}{2}$. Then choose a non-critical $\delta''_i\in (\delta_i-\nu,\delta_i-\frac{\nu}{2})$. Assume that $\gamma\in W^{k,2}_{\delta'_i}(C_{x_i}\cap\{r\le r_{0,x_i}\})$ and $D\gamma\in W^{k-m,2}_{\delta_i-m}$. Then $D\gamma\in W^{k-m,2}_{\delta''_i-m}$. So there exists $\gamma''$ in the direct sum of $\chi_i\mathcal{P}_i(\lambda_i)$ for all $\delta'_i<\mathrm{Re}(\lambda_i)<\delta''_i$ such that $\gamma-\gamma''\in W^{k,2}_{\delta''_i}(C_{x_i}\cap\{r\le r_{0,x_i}\})$. Since $D\gamma''$ vanishes on a neighborhood of $x_i$, it is easy to see that $D(\gamma-\gamma'')\in W^{k-m,2}_{\delta_i-m}$. So it is also true that $D_{C_{x_i}}(\gamma-\gamma'')\in W^{k-m,2}_{\delta_i-m}$. Remark that $\gamma-\gamma''-D_{C_{x_i}}^{-1}(\chi_i(\gamma-\gamma''))$ satisfies the hypothesis of Theorem \ref{Polyhomogenous-IL} in a smaller neighborhood of $x_i$. So it can be written as the sum of elements in $\chi_i\mathcal{K}_i(\lambda_i)$ for all $\mathrm{Re}(\lambda_i) \in (\delta''_i,\delta_i)$ plus an element in $W^{k,2}_{\delta_i}$. The construction of $\mathcal{P}_i(\lambda_i)$ provides a contradiction to the definition of $\delta_i$. The argument near infinity is similar.
\end{proof}

The following corollary is a refinement of Theorem \ref{Lockhart-McOwen}:

\begin{corollary}
Suppose that $\delta_i>\delta'_i$ or $\delta_\infty>\delta'_\infty$ are non-critical, then the maps \[D:(\bigoplus_{\substack{i=1,...,N_2,\infty\\
\delta_i>\mathrm{Re}(\lambda_i)>\delta'_i}}\chi_i\mathcal{P}_i(\lambda_i))
\oplus W^{k,2}_{\delta}\rightarrow W^{k-m,2}_{(\delta_1-m,...\delta_{N_2}-m,\delta_\infty)}\]
and \[D:W^{k,2}_{\delta'}\rightarrow W^{k-m,2}_{(\delta'_1-m,...\delta'_{N_2}-m,\delta'_\infty)}\] commute with the inclusion map. Moreover, the inclusion map induces isomorphisms on the kernels and cokernels of $D$.
\label{Refined-index-change-formula}
\end{corollary}
\begin{proof}
It is trivial to see that the inclusion map commutes with $D$ and induces maps between kernels of $D$ or between cokernels of $D$. It is also trivial to see that the induces map is injective between the kernels of $D$. Suppose that $\gamma\in W^{k+m,2}_{\delta'}$ and $D\gamma=0$. By Theorem \ref{Polyhomogenous-L-totally-character}, it is in the image of the map between the kernels induced by the inclusion map. On the other hands, for any element in $W^{k-m,2}_{(\delta'_1-m,...\delta'_{N_2}-m,\delta'_\infty)}$, the product with $\chi_i$ can be inverted using the map $\chi_iD+(1-\chi_i)D_{C_{x_i}}$. The product of this inverse with $\chi_i$ lies in $W^{k,2}_{\delta'}$. Thus, any element in $W^{k-m,2}_{(\delta'_1-m,...\delta'_{N_2}-m,\delta'_\infty)}$ can be written as an element in the image of $D$ on $W^{k,2}_{\delta'}$ plus an element vanishing near $x_i$ and infinity. This proves the surjectivity of the map between cokernels induced by the inclusion map. The injectivity of this map is an immediate corollary of Theorem \ref{Polyhomogenous-L-totally-character}.
\end{proof}

For the Laplacian operator acting on functions, as in Proposition 2.9 of \cite{HeinSun}, any element $\gamma$ in $\mathcal{K}_i(\lambda_i)$ or $\mathcal{K}_\infty(\lambda_\infty)$ can be as a generalized Fourier series. Then the following proposition follows easily from the explicit solution of the ordinary differential equation as well as Theorem 2.14 of \cite{HeinSun}:

\begin{proposition}
For the Laplacian operator acting on functions, the following statements are true:

(1) Any critical $\lambda_i\in\mathbb{C}$ or $\lambda_\infty\in\mathbb{C}$ is in fact real.

(2) $\mathcal{K}_i(\lambda_i)$ has no $(-\log r)^p$ terms.

(3) $\mathcal{K}_\infty(0)=\mathrm{Span}\{1,t\}$.

(4) There is no critical rate in $(-2n+2,0)$ near $x_i$.

(5) $\mathcal{K}_i(0)=\mathrm{Span}\{1\}$.

(6) There is no critical rate in $(0,1]$ near $x_i$.

(7) If $\lambda_i\in (1,2)$, then any element in $\mathcal{K}_i(\lambda_i)$ is pluriharmonic.

(8) Any element in $\mathcal{K}_i(2)$ can be written as a pluriharmonic function in $\mathcal{K}_i(2)$ plus a $J(r\frac{\partial}{\partial r})$-invariant function in $\mathcal{K}_i(2)$.

(9) Denote the direct sum of pluriharmonic functions in $\mathcal{K}_i(\lambda_i)$ with rates $\lambda_i\in(1,2]$ by $\mathcal{P}_i$. In Corollary \ref{Refined-index-change-formula}, $\mathcal{P}_i(\lambda_i)$ can be replaced by the corresponding $\mathcal{P}_i$ because any pluriharmonic function is harmonic with respect to both $\omega_{C_{x_i}}$ and $\omega$. Denote the space of $J(r\frac{\partial}{\partial r})$-invariant functions in $\mathcal{K}_i(2)$ by $\mathcal{H}_i$. Remark that the difference between $\mathcal{H}_i$ and corresponding element in $\mathcal{P}_i(2)$ lies in $W^{k,2}_{2+\frac{\nu}{2}}$ near $x_i$.
\label{Calabi-Yau-cone}
\end{proposition}

The next goal is the analysis on manifolds with edge singularities. It is required that the smooth part of the manifold can be viewed as a manifold with boundary and the boundary is a fibration. As a special case, assume that the boundary is the trivial fibration over $\mathbb{S}^1$ with fiber $F$, in other words, is $F\times\mathbb{S}^1$. So in this special case, the singular manifold looks like $C(F)\times\mathbb{S}^1$ locally.

In the pioneer work of Mazzeo \cite{Mazzeo}, the weighted Sobelov space is defined using the same formula as in Definition \ref{Definition-weighted-Sobelov}. Elliptic differential operators like $\Delta:W^{k,2}_{\delta}\rightarrow W^{k-2,2}_{\delta-2}$ are studied in \cite{Mazzeo}. The main result of \cite{Mazzeo} discusses whether the elliptic differential operator is Fredholm or not. On the other hands, Theorem 7.14 of \cite{Mazzeo} is closely related Theorem \ref{Polyhomogenous-L-totally-character}.

Remark that there are different versions of weighted analysis by changing the definition slightly. One way is to change the domain and range of $\Delta$ as Cheeger \cite{Cheeger} and Hunsicker-Mazzeo \cite{HunsickerMazzeo} did when studying Hodge theory on manifolds with edge singularities. The other way is to change the definition of $C^{k,\alpha}_\delta$ as Chen-Donaldson-Sun \cite{ChenDonaldsonSun1,ChenDonaldsonSun2,ChenDonaldsonSun3} did when using manifolds with edge singularities to study the K\"ahler-Einstein problem.

\section{Asymptotically cylindrical Calabi-Yau manifolds with isolated conical singularities}

In this section, Theorem \ref{Non-compact-Calabi-Yau} is proved as a combination of \cite{HeinSun} and \cite{HaskinsHeinNordstrom}. There are several major technical problems in this process. Firstly, it is necessary to find a good substitute for the finiteness of diameter property in \cite{HeinSun}. Secondly, it is not clear how to get the generalization of \cite{EyssidieuxGuedjZeriahi} on the existence of weak solutions because the weak solution in the sense of current is too weak to apply the standard analysis for the asymptotically cylindrical manifolds. Thirdly, the openess part of \cite{HeinSun} uses a non-standard weighted analysis and therefore does not have a simple generalization to the non-compact case. Finally, as personally communicated to the author by Hein-Sun, the proof of Proposition 3.2 of \cite{HeinSun} requires a little more explanation. In fact, one needs to show that on the central fiber we can assume a priori the existence of $K_1>0$ such that the $K$-inequalities hold for $s=0$, $t\in[0,1]$ with this $K_1$, and one needs to take $K$ to be bigger than this $K_1$. In the setting studied in the paper of Hein-Sun, this follows directly from the results of \cite{EyssidieuxGuedjZeriahi} since one can work on the fixed variety $X_0$. In our setting, it is necessary find out solutions to the above problems.

As in \cite{HaskinsHeinNordstrom}, near $S_s$, let $(t,\vartheta)$ be the coordinates on $[T,\infty)\times\mathbb{S}^1$ such that $f$ can be written as $e^{t+i\vartheta}$ on $V_s$ near $S_s$. Using the diffeomorphisms between fibers of $f$ near $S_s$, there exists a smooth family of local diffeomorphisms between $f^{-1}(\{|\frac{1}{f}|\le e^{-T}\})\cap Z_s$ and $\{|\frac{1}{f}|<e^{-T}\}\times S_s$. Its restriction yields a smooth family of diffeomorphisms \[\Phi_{\infty,s}:[T,\infty)\times \mathbb{S}^1\times S_s\rightarrow U_{\infty,s}=f^{-1}(\{0<|\frac{1}{f}|\le e^{-T}\})\cap Z_s\subset V_s.\] Using Yau's solution to the Calabi conjecture, it is easy to choose a holomorphic family of nowhere vanishing holomorphic $(n-1)$-forms $\Omega_{S_s}$ and a family of Ricci-flat metrics $\omega_{S_s}\in[c_1(L)|_{S_s}]$ on $S_s$ such that after replacing $\Omega_s, \Omega_{\infty,s}$ by their products with constants, \[(\omega_{\infty,s},\Omega_{\infty,s})=(dt\wedge d\vartheta+\omega_{S_s},(d\vartheta-idt)\wedge\Omega_{S_s})\] satisfy $\omega_{\infty,s}^n=i^{n^2}\Omega_{\infty,s}\wedge\bar\Omega_{\infty,s}$ and $\Phi_{\infty,s}^*\Omega_s=\Omega_{\infty,s}+d\varsigma_s$ with \[|\nabla^j_{\omega_{\infty,s}}\varsigma_s|_{\omega_{\infty,s}}\le C_j e^{-\nu_\infty t}\] for all $t>T$ and $j\in\mathbb{N}_0$, where $\nu_\infty$, $C_j$, $T$ are positive constants independent of $s$. Remark that on the cylinder $[T,\infty)\times \mathbb{S}^1\times S_s$, any exponentially decaying closed form can be written as $d$ of an exponentially decaying form.

The next goal is to construct a family of background metrics $\hat\omega_s$. Using Part 1 of Section 4.2 of \cite{HaskinsHeinNordstrom}, the pull back of Fubini-Study metric can be modified to a smooth family of metrics $\hat\omega_s$ by adding a smooth family of $i\partial\bar\partial$ exact forms supported in $U_{\infty,s}$ such that $\Phi_{\infty,s}^*\hat\omega_s=\omega_{\infty,s}+d\varrho_s$ with $|\nabla^j_{\omega_{\infty,s}}\varrho_s|_{\omega_{\infty,s}}\le C_j(e^{-\nu_\infty t})$ for all $t>T$ and $j\in\mathbb{N}_0$ after modifying the positive constants $\nu_\infty$, $C_j$ and $T$ if necessary. Moreover, $\hat\omega_s$ can be chosen to satisfy \[\int_V \frac{\hat\omega_s^n}{n!}-\frac{i^{n^2}}{2^n}\Omega_s\wedge\bar\Omega_s=0.\] Remark that the assumption that $\Omega_s$ are comparable to explicit $n$-forms $\Omega_{\epsilon_x(s)}$ near $x\in V^{\mathrm{sing}}$ means that the singularities are harmless.

By assumption, for each $x_i\in V^{\mathrm{sing}}$, there exists a local biholomorphism $\Phi_{x_i}:\mathcal{U}_{x_i}\rightarrow \mathcal{O}_{x_i}$ between a neighborhood $\mathcal{U}_{x_i}$ of $x_i$ in $\mathcal{Z}$ and a neighborhood $\mathcal{O}_{x_i}$ of the vertex $o_{x_i}$ in $\mathcal{C}_{x_i}$. Assume that $\mathcal{U}_{x_i}$ and $U_{\infty,s}$ are disjoint. As in Proposition 2.4 of \cite{ArezzoSpotti}, there exists a bounded family of functions $\psi_{1,s}$ such that $\psi_{1,s}$ is smooth outside $x_i$ and $\hat\omega_s+i\partial\bar\partial\psi_{1,s}$ equals to $\Phi_{x_i}^*\omega_{\epsilon_{x_i}(s)}$ after shrinking the neighborhood $\mathcal{U}_{x_i}$ of $x_i$. Assume that $\psi_{1,s}$ is supported in $\mathcal{U}_{x_i}$ before the shrinking of $\mathcal{U}_{x_i}$. In particular, the condition \[\int_V \frac{(\hat\omega_s+i\partial\bar\partial\psi_{1,s})^n}{n!}-\frac{i^{n^2}}{2^n}\Omega_s\wedge\bar\Omega_s=0\] is still true.

Assume that $\nu$ is small enough. Let $F_s$ be the function in $(\oplus_{i=1}^{N_2}\mathbb{R}\chi_i)\oplus C^{\infty}_{\nu,...\nu,\nu}$ satisfying \[\frac{(\hat\omega_s+i\partial\bar\partial\psi_{1,s})^n}{n!}=e^{F_s}\frac{i^{n^2}}{2^n}\Omega\wedge\bar\Omega.\] $F_s$ is the real part of a holomorphic function on $\mathcal{U}_{x_i}$. Denote $\mathcal{U}_{x_i}\cap Z_s$ by $U_{x_i,s}$. Then it is easy to find a family of functions $F_{\tau,s}$ on $V_s=Z_s\setminus S_s$ for all $\tau\in[0,1]$ continuous in $(\oplus_{i=1}^{N_2}\mathbb{R}\chi_i)\oplus C^{\infty}_{\nu,...\nu,\nu}$ topology such that $F_{\tau,s}=\tau F_s$ on $U_{x_i,s}$, $F_{1,s}=F_s$ on $x\in \mathcal{Z}$, $F_{0,s}=0$ and
\[\int_{V_s} (e^{F_{\tau,s}}-1)\frac{i^{n^2}}{2^n}\Omega_s\wedge\bar\Omega_s=0.\]

As in \cite{HeinSun} and \cite{HaskinsHeinNordstrom}, define $\mathcal{T}\subset[0,1]$ as the set of $\tau$ such that
\[\frac{(\hat\omega_0+i\partial\bar\partial\psi_{\tau,0})^n}{n!}=e^{F_{\tau,0}}\frac{i^{n^2}}{2^n}\Omega\wedge\bar\Omega\]
has a bounded and smooth solution $\psi_{\tau,0}$ on $V\setminus V^{\mathrm{sing}}$ such that $\psi_{\tau,0}\in C^{\infty}_{\nu}$ near infinity and $\hat\omega_0+i\partial\bar\partial\psi_{\tau,0}$ has conical singularity at each $x_i\in V^{\mathrm{sing}}$ with tangent cone the same as Example \ref{Nodal-singularity}. It is clear that $1\in \mathcal{T}$.

The openness of $\mathcal{T}$ can be proved using the following proposition as in \cite{HeinSun}:
\begin{proposition}
The Laplacian operator is a bijective map between the set \[W^{k,2}_{(2+\nu,...,2+\nu,\nu)}
\oplus(\oplus_{i=1}^{N_2}\chi_i(\mathcal{P}_i\oplus\mathbb{R}r_{x_i}^2\oplus\mathcal{H}_i))\]
and the set \[\{\gamma\in W^{k-2,2}_{\nu,...,\nu,\nu}\oplus(\oplus_{i=1}^{N_2}\mathbb{R}\chi_i):\int_V \gamma\omega^n=0\}.\]
\end{proposition}
\begin{proof}
In the compact case, this proposition was proved in Hein-Sun's paper \cite{HeinSun} using a non-standard weighted norm. It is not clear how to do the analogy here. However, this proposition can be proved using Corollary \ref{Refined-index-change-formula} and Proposition \ref{Calabi-Yau-cone} as an analogy of Proposition 2.7 of \cite{HaskinsHeinNordstrom}:

Consider the Laplacian operator acting on functions from $W^{k,2}_{(-n+1,...,-n+1,\nu)}$ to $W^{k-2,2}_{(-n-1,...,-n-1,\nu)}$. Remark that the weight is non-critical by Proposition \ref{Calabi-Yau-cone} if $\nu$ is small enough. The dual operator is the Laplacian operator acting on functions from $W^{2-k,2}_{(-n+1,...,-n+1,-\nu)}$ to $W^{-k,2}_{(-n-1,...,-n-1,-\nu)}$. Thus \[i_{(-n+1,...,-n+1,\nu)}(\Delta)=-i_{(-n+1,...,-n+1,-\nu)}(\Delta).\] Therefore, by Theorem \ref{Lockhart-McOwen} and Proposition \ref{Calabi-Yau-cone}, \[i_{(-n+1,...,-n+1,\nu)}(\Delta)=-\frac{1}{2}
\sum_{-\nu<\mathrm{Re}(\lambda_\infty)<\nu}d(\lambda_\infty)=-1.\]

Using Corollary \ref{Refined-index-change-formula} and Proposition \ref{Calabi-Yau-cone}, the index of
\[\Delta:W^{k,2}_{(2+\nu,...,2+\nu,\nu)}
\oplus(\oplus_{i=1}^{N_2}\chi_i(\mathcal{P}_i\oplus\mathcal{H}_i))\rightarrow W^{k-2,2}_{\nu,...,\nu,\nu}\]
is also -1. Using the fact that $-\Delta (\chi_i r_{x_i}^2)-4n\chi_i\in W^{k-2,2}_{\nu,...\nu,\nu}$, it is easy to see that the index of
\[\Delta:W^{k,2}_{(2+\nu,...,2+\nu,\nu)}
\oplus(\oplus_{i=1}^{N_2}\chi_i(\mathcal{P}_i\oplus\mathbb{R}r_{x_i}^2\oplus\mathcal{H}_i))\rightarrow
W^{k-2,2}_{\nu,...,\nu,\nu}\oplus(\oplus_{i=1}^{N_2}\mathbb{R}\chi_i)\]
is also -1. Remark that for any function $\psi$ in the kernel of this operator, the decay condition of $\psi$ near $x_i$ and infinity insures that the boundary term in the integral $\int |\nabla \psi|^2-\psi\Delta \psi$ vanishes. This implies that $\nabla \psi=0$, so $\psi=0$ by the decay condition near infinity. Another integration by parts shows that the integral of any function in the image is 0. Now, this proposition is an immediate corollary of the fact that the index is -1.
\end{proof}

Now assume that $\{\tau_i\}\subset\mathcal{T}\rightarrow \tau_\infty$. It suffices to show that $\tau_\infty\in\mathcal{T}$. Using Theorem 4.1 of \cite{HaskinsHeinNordstrom}, for all $s\not=0$, there exists $\psi_{\tau,s}\in C^\infty_\nu$ such that
\[\frac{(\hat\omega_s+i\partial\bar\partial\psi_{\tau,s})^n}{n!}=e^{F_{\tau,s}}\frac{i^{n^2}}{2^n}\Omega_s\wedge\bar\Omega_s.\] The goal is to obtain uniform estimates on $\psi_{\tau_j,s_j}$ for a sequence $s_j\rightarrow 0$ so that the limit is expected to be $\psi_{\tau_\infty,0}$.

The starting point is the $C^0$-estimate as in Step 1 of \cite{HaskinsHeinNordstrom}. It requires an estimate using $\omega_{1,s}=\hat\omega_s+i\partial\bar\partial\psi_{1,s}$ as the background metric. Recall the definition of $t$ in Definition \ref{Definition-weighted-Sobelov}. Using this notation, the required estimate can be stated as the following:
\begin{proposition}
For all small enough $s\not=0$ and $\mu>0$, there exists a family of piecewise constant positive functions $\xi_{\mu,s}$ on $V_s$ with $C_{\mu}^{-1}e^{-2\mu t}\le\xi_{\mu,s}\le C_\mu e^{-2\mu t}$ and $\int_{V_s}\xi_{\mu,s}\omega_{1,s}^n=1$ such that
\[||e^{-\mu t}(u-\bar u_{\mu})||_{L^{2\sigma}(\omega_{1,s})}\le C_{\mu,\sigma}||\nabla^{\omega_{1,s}} u||_{L^2(\omega_{1,s})}\]
for all $\sigma\in[1,\frac{n}{n-2}]$ and all $u\in C^\infty_0(V_s)$, where $\bar u_\mu=\int_{V_s}u\xi_{\mu,s}\omega_{1,s}^n$.
\end{proposition}
\begin{proof}
Recall that $\hat\omega_s+i\partial\bar\partial\psi_{1,s}$ is Ricci flat in $\mathcal{U}_{x_i}$. It is clear that its diameter and volume on $U_{x_i,s}$ have two-sided bounds for small enough $s$. In particular, it has a Sobolev bound on $U_{x_i,s}$ uniform in $s$. On the region $(Z_s\setminus U_{\infty,s})\setminus \cup_{i=1}^{N_2}U_{x_i,s}$, the metrics are smooth, so its diameter and volume also have two-sided bounds and moreover the Sobolev bound is also uniform in $s$ for small enough $s$. Now the proposition follows as the proof of Proposition 4.21 of \cite{HaskinsHeinNordstrom}.
\end{proof}

Using $\omega_{1,s}$ as the background metric, Step 1 of the proof of Theorem 4.1 of \cite{HaskinsHeinNordstrom} can be applied without change. In particular, the potential $\psi_{\tau,s}-\psi_{1,s}$ has a $C^0$-bound uniform in $\tau$ and $s$ for $\tau\in[0,1]$ and small non-zero $s$. However, recall that $\psi_{1,s}$ already has a uniform $C^0$-bound. Therefore, the potential $\psi_{\tau,s}$ also has a uniform $C^0$-bound. The next goal is the proof of the following $C^2$-bound for the metric $\omega_{\tau,s}=\hat\omega_s+i\partial\bar\partial\psi_{\tau,s}$:
\begin{proposition}
For all $\tau\in[0,1]$, all small enough non-zero $s$, there exists a constant $C$ independent of $\tau$ and $s$ such that $C^{-1}\hat\omega_s\le\omega_{\tau,s}$ on $V_s$. Moreover, for any closed subset $\mathcal{K}$ of $\mathcal{V}\setminus V^{\mathrm{sing}}$, there exists a constant $C_\mathcal{K}$ only depending on $\mathcal{K}$ such that in addition $\omega_{\tau,s}\le C_\mathcal{K}\hat\omega_s$ on $K_s=\mathcal{K}\cap V_s$.
\end{proposition}
\begin{proof}
Viewing the identity map as a harmonic from $(V_s,\omega_{\tau,s})$ to $(V_s,\hat\omega_s)$, the Eells-Sampson's Bochner type formula ( Equation (16) of \cite{EellsSampson}, see also Equation (3) of \cite{Chern} and Thereom 4.1 of \cite{Lu}) implies that
\[-\Delta_{\omega_{\tau,s}}\log \mathrm{tr}_{\omega_{\tau,s}}\hat\omega_s \ge C(-|\mathrm{Ric}_{\omega_{\tau,s}}|_{\omega_{\tau,s}} -|\mathrm{Rm}_{\hat\omega_s}|_{\hat\omega_s}\mathrm{tr}_{\omega_{\tau,s}}\hat\omega_s).\]
Remark that $-\Delta|z|^2=4$ on $\mathbb{C}$ using the Laplacian operator $\Delta=d^*d+dd^*$ acting on 0-forms.
The Riemannian curvature of $\hat\omega_s$ is uniformly bounded on $U_{\infty,s}$. The Ricci curvature of $\omega_{\tau,s}$ is bounded by $C\mathrm{tr}_{\omega_{\tau,s}}\hat\omega_s$ because the Ricci form $i\partial\bar\partial\log F_{\tau,s}$ is bounded using $\hat\omega_s$-norm. Therefore, \[-\Delta_{\omega_{\tau,s}}\log \mathrm{tr}_{\omega_{\tau,s}}\hat\omega_s \ge -C\mathrm{tr}_{\omega_{\tau,s}}\hat\omega_s\] on $U_{\infty,s}$.

On the compact part $V_s\setminus U_{\infty,s}$, the curvature of $\hat\omega_s$ is no longer bounded. However, recall that $\hat\omega_s$ equals to the pull back of the Fubini-Study metric in this region. Thus, as in the proof of Lemma 3.2 of \cite{RongZhang}, embed $V_s$ into $\mathbb{CP}^{N_1}$ and view the composition of the embedding with the identity map as a harmonic map from $(V_s,\omega_{\tau,s})$ to $(\mathbb{CP}^{N_1},\omega_{FS})$. In this case, the Eells-Sampson's Bochner type formula implies that
\[-\Delta_{\omega_{\tau,s}}\log \mathrm{tr}_{\omega_{\tau,s}}\hat\omega_s \ge C(-|\mathrm{Ric}_{\omega_{\tau,s}}|_{\omega_{\tau,s}} -|\mathrm{Rm}_{\omega_{FS}}|_{\omega_{FS}}\mathrm{tr}_{\omega_{\tau,s}}\hat\omega_s).\]
The Riemannian curvature of the Fubini-Study metric is indeed bounded. The Ricci curvature of $\omega_{\tau,s}$ is also bounded by $C\mathrm{tr}_{\omega_{\tau,s}}\hat\omega_s$. Thus \[-\Delta_{\omega_{\tau,s}}\log \mathrm{tr}_{\omega_{\tau,s}}\hat\omega_s \ge -C\mathrm{tr}_{\omega_{\tau,s}}\hat\omega_s\] is true on $V_s$.

As in the proof of Lemma 3.2 of \cite{RongZhang}, using the formula \[-\Delta_{\omega_{\tau,s}}\psi_{\tau,s}=2n-2\mathrm{tr}_{\omega_{\tau,s}}\hat\omega_s,\] it is easy to see that
\[-\Delta_{\omega_{\tau,s}}(\log \mathrm{tr}_{\omega_{\tau,s}}\hat\omega_s-C\psi_{\tau,s}) \ge C\mathrm{tr}_{\omega_{\tau,s}}\hat\omega_s-2Cn.\]
Since $\omega_{\tau,s}$ and $\hat\omega_s$ are asymptotically cylindrical, for a large enough number $T_{\tau,s}$ depending on $\tau$ and $s$, for all $t\ge T_{\tau,s}$, $\log \mathrm{tr}_{\omega_{\tau,s}}\hat\omega_{\tau,s}-C\psi_{\tau,s}\le\log(2n)$. So either $\log \mathrm{tr}_{\omega_{\tau,s}}\hat\omega_s-C\psi_{\tau,s}$ is bounded above by $\log(3n)$ or it attains its maximum at a point on $V_s$. In either cases, there is a uniform upper bound of $\mathrm{tr}_{\omega_{\tau,s}}\hat\omega_s$ independent of $\tau$ and $s$ using the $C^0$-bound of $\psi_{\tau,s}$. The lower bound comes from the upper bound of $\mathrm{tr}_{\omega_{\tau,s}}\hat\omega_s$ and the bound of $F_{\tau,s}$ on $K_s$.
\end{proof}

Let $\mathcal{K}$ be the set $\{t\ge1\}$, then there a $C^2$-bound on $\mathcal{K}$. The uniform $C^{2,\alpha}$-estimate independent of $\tau$, $s$ and $T\ge1$ on $\{T\le t\le T+1\}$ for real Monge-Amp\'ere equation was done by Evans-Krylov-Trudinger. See Section 17.4 of \cite{GilbargTrudinger} for details. In complex case, the arguments in Section 17.4 of \cite{GilbargTrudinger} still work. An alternative way to achieve the $C^{2,\alpha}$-estimate on $\{T\le t\le T+1\}$ was done by Theorem 1.5 of \cite{ChenWang} using the rescaling argument. Now it is standard to get a $C^\infty$-bound of $\psi_{\tau,s}$ on $\{T\le t\le T+1\}$ independent of $\tau$, $s$ and $T\ge 1$ through Schauder estimates. Using Step 3 and Step 4 of proof of Theorem 4.1 of \cite{HaskinsHeinNordstrom}, there is a $C^\infty_\nu$-bound of $\psi_{\tau,s}$ on $\{t\ge 1\}$. The same argument implies the $C^\infty$-estimate on $\mathcal{K}$ with bound depending on $\mathcal{K}$ but independent of $\tau$ and $s$ for all compact subset $\mathcal{K}$ of $\mathcal{V}\setminus V^{\mathrm{sing}}$.

Recall that there exists a smooth family of diffeomorphisms between $S_0$ and $S_s$. Its product with $[T,\infty)\times\mathbb{S}^1$ is a smooth family of diffeomorphisms from $\{t\ge T\}\cap V_0$ to $\{t\ge T\}\cap V_s$ such that it is the identity map when $s=0$. For any closed subset $K$ of $V_0\setminus V^{\mathrm{sing}}$, there exists a smooth family of embeddings $\hat\Phi_{K,s}:K\rightarrow V_s$ such that it is the identity map when $s=0$ and is the given diffeomorphism when restricted to $t\ge T$. Using $\hat\Phi_{K,s}$, it is possible to talk about the $C^\infty_{\mathrm{loc}}$ convergence as the following:

\begin{proposition}
For all $\tau_i\rightarrow\tau_\infty$ and $s_i\rightarrow 0$, there exists a subsequence $\tau_{i_k}$ and $s_{i_k}$ and a metric $\tilde\omega=\hat\omega_0+i\partial\bar\partial\tilde\psi$ on $V_0\setminus V^{\mathrm{sing}}$ such that for any closed subset $K$ of $V_0\setminus V^{\mathrm{sing}}$, $\hat\Phi_{K,s}^*\psi_{\tau_{i_k},s_{i_k}}\rightarrow\tilde\psi$ in $C^\infty_\nu$-sense on $K$. Moreover, \[\frac{\tilde\omega^n}{n!}=e^{F_{\tau_\infty,0}}\frac{i^{n^2}}{2^n}\Omega_0\wedge\bar\Omega_0\] in weak sense.
\end{proposition}
\begin{remark}
For bounded plurisubharmonic functions on a smooth manifold, the Monge-Amp\'ere equation makes sense in the weak sense. A weak solution on a singular manifold is defined as the solution whose pull back to the resolution of the singular manifold satisfies the Monge-Amp\'ere equation in weak sense.
\end{remark}
\begin{proof}
For any $T_0\ge T$, using the diagonal argument and the uniform $C^{\infty}$-bound on $\{T_0\le t\le T_0+1\}$, it is easy to find a subsequence $C^{\infty}$-converging on $\{T_0\le t\le T_0+1\}$. The limit belongs to $C^\infty_{\nu}$ and satisfies the equation. The convergence in $C^\infty_{\nu}$-norm follows from the weighted analysis on the cylinder. The pre-compactness on $K\cap\{t\le T\}$ follows from the $C^\infty_{\mathrm{loc}}$-estimate of $\psi_{\tau,s}$ as in Theorem 1.4 of \cite{RongZhang}.

The limit $\tilde\omega$ satisfies the equation locally on $V\setminus V^{\mathrm{sing}}$. In \cite{HeinSun} and \cite{RongZhang}, they claim that the compact analogy is a weak solution without proof. In an email from Hein-Sun to the author, they provided the following explaination:

$\tilde\omega$ can be pulled back to the resolution of $V$ so that the equation is satisfied locally on the resolution except on the exceptional divisor. However, the pull back of $\tilde\omega$ can be written as $i\partial\bar\partial$ of a bounded plurisubharmonic function locally except on the exceptional divisor. By Section 5 of \cite{Demailly}, the bounded plurisubharmonic function can be extended to the resolution. By Prop 4.6.4 of \cite{Klimek}, the extension satisfies the Monge-Amp\'ere equation in the weak sense because the Monge-Amp\'ere mass on the exceptional divisor vanishes.
\end{proof}

On the other hand, it is also interesting to consider the Gromov-Hausdorff limit of a subsequence of $\omega_{\tau_{i_k},s_{i_k}}$. To get started, using the $C^\infty_{\mathrm{loc}}$-convergence of the metric outside the singularity, there exists $d>0$ such that for any $x\in V^{\mathrm{sing}}$ and any point $q\in U_{x,s_{i_k}}$ with $\mathrm{dist}_{\omega_{\tau_{i_k},s_{i_k}}}(q,\partial U_{x,s_{i_k}})=d$, the volume of the ball $B_{\omega_{\tau_i,s_i}}(q,d)$ has a positive lower bound. Therefore, for any $x\in V^{\mathrm{sing}}$ and any point $p\in U_{x,s_{i_k}}$, let $D$ be the distance of $p$ to $\partial U_{x,s_{i_k}}$ using $\omega_{\tau_{i_k},s_{i_k}}$. Then if $D>3d$, choose $q$ as the point on the minimal geodesic of length $D$ joining $p$ and $\partial U_{x,s_{i_k}}$ such that the distance from $p$ to $q$ is $D-d$. Therefore, using the fact that $\omega_{\tau_{i_k},s_{i_k}}$ is Ricci flat on $U_{x,s_{i_k}}$, the volume comparison implies that using $\omega_{\tau_{i_k},s_{i_k}}$,
\[\mathrm{Vol}(B(p,D-2d))\ge\frac{\mathrm{Vol}(B(p,D)\setminus B(p,D-2d))}{(\frac{D}{D-2d})^{2n}-1}\ge\frac{CD\mathrm{Vol}(B(q,\epsilon))}{d}.\]
Since the volume of $U_{x,s_{i_k}}$ is bounded, $D$ is also bounded. Using the standard $\epsilon$-net argument, it is easy to see that $(V_{s_{i_k}},\omega_{\tau_{i_k},s_{i_k}})$ has a Gromov-Hausdorff limit $(X,d_X)$ with Gromov-Hausdorff approximation equalling to $\hat\Phi_{K,s_{i_k}}$ when restricted to the set $K=V_0\setminus\cup_{x\in V^{\mathrm{sing}}}U_{x,0}$. Remark that the space $(X,d_X)$ is non-compact and therefore the result of Donaldson-Sun can not be applied directly. In order to solve this problem, the Gromov-Hausdorff limit of the metrics defined by the (1,1)-forms $(\pi_1\circ i)^*\omega_{FS}+i\partial\bar\partial\psi_{\tau_{i_k},s_{i_k}}$ is considered instead. This (1,1)-form may not be positive. However, by checking the difference between $(\pi_1\circ i)^*\omega_{FS}$ and $\hat\omega_{s_{i_k}}$ carefully, there exists a bump (1,1)-form $\beta$ on $\mathbb{CP}^1$ such that for large enough number $N_3$, $(\pi_1\circ i)^*\omega_{FS}+i\partial\bar\partial\psi_{\tau_{i_k},s_{i_k}}+N_3 f^*\beta$ is positive. Without loss of generality assume that $N_3\beta=N_4 [c_1(\mathcal{O}(1))]$ for an integer $N_4$. Thus, after modifying the metric $\omega_{\tau_{i_k},s_{i_k}}$ near infinity, there exists a compact metric on $Z_s$ in the cohomology class $c_1((\pi_1\circ i)^*(\mathcal{O}(1))\otimes(f^*\mathcal{O}(1))^{\otimes N_4})$ such that the diameter, the volume and the Ricci curvature have two-sided bounds. The Gromov-Hausdorff limit is isometric to $(X,d_X)$ except on the end $U_{\infty,s_{i_k}}$. Therefore, even though $(X,d_X)$ is non-compact, Cheeger-Colding theory \cite{CheegerColding1, CheegerColding2, CheegerColding3} and Donaldson-Sun theory \cite{DonaldsonSun1, DonaldsonSun2} can still be used.

As in \cite{RongZhang}, $(X,d_X)$ is isometric to the completion of the metric $\tilde\omega$. In particular, when fixing $\tau_i=\tau_\infty\in\mathcal{T}$, using the uniqueness of weak solutions, $d_X$ is isometric to $\omega_{\tau_i,0}$. Now assume that instead $\tau_i\rightarrow\tau_\infty\in\mathcal{T}$, then by choosing small enough $s_i$, $(X,d_X)$ is also the Gromov-Hausdorff limit of $\omega_{\tau_i,0}$. The rest parts of \cite{HeinSun} can be applied without change.

\section{The matching problem}

This section solves the matching problem. The starting point is the review of the matching problem of smooth manifolds in \cite{CortiHaskinsNordstromPacini2} using a particular example:
\begin{example}
(Example 7.3 of \cite{CortiHaskinsNordstromPacini1})

Fix a 2-plane $\Pi\subset\mathbb{CP}^4$. Let $X_+\subset\mathbb{CP}^4$ be a general quartic 3-fold containing $\Pi$. It has 9 nodal singularities. The blowing up of $X_+$ over $\Pi$ yields a non-singular 3-fold $Y_+\rightarrow X_+$ with nine (-1,-1)-curves resolving
the 9 ordinary double points of $X_+$ on $\Pi$. In Example 7.3 of \cite{CortiHaskinsNordstromPacini1}, Corti-Haskins-Nordstr\"om-Pacini prove that $N_+:=H^2(Y_+,\mathbb{Z})=\mathbb{Z}^2$ with basis $\tilde\Pi$ (the proper transform of $\Pi$) and $-K_{Y_+}$. The quadratic form $[.]\cup[.]\cup -K_{Y_+}$ in this basis equals to \[\begin{bmatrix}
    -2 & 1 \\
    1 & 4
\end{bmatrix}.\]
It is easy to see that $-K_{Y_+}$ also equals to the pull back of the $\mathcal{O}(1)$-bundle of $\mathbb{CP}^4$.
\end{example}

On the other hand, choose $Y_-=X_-=\mathbb{CP}^3$. Then $N_-:=H^2(Y_-,\mathbb{Z})=\mathbb{Z}$ with base $\mathcal{O}(1)$. The anti-canonical divisor $-K_{Y_-}=\mathcal{O}(4)$. The quadratic form $c_1(\mathcal{O}(1))\cup c_1(\mathcal{O}(1)) \cup -K_{Y_-}$ in this basis equals to 4.

Choose smooth anti-canonical divisors $S_{\pm}$ of $Y_{\pm}$. It is easy to find other anti-canonical divisors $S_{0,\pm}$ intersecting $S_{\pm}$ transversally. Let $C_{\pm}$ be their intersections. Then the ratios of the corresponding sections provide holomorphic functions $f_{\pm}$ from $\tilde Z_{\pm}$ to $\mathbb{CP}^1$ with $\{f_{\pm}=\infty\}$ equal to the proper transforms $S_{\infty,\pm}$ of $S_{\pm}$, where $\tilde Z_{\pm}$ is the blowing up of $Y_{\pm}$ at $C_{\pm}$. Define $Z_+$ as the blowing up of $X_+$ at $C_+$. Define $Z_-$ as $\tilde Z_-$. Let $\Omega_{\pm}$ be the meromorphic 3-forms on $\tilde Z_{\pm}$ with simple poles along $S_{\infty,\pm}$. Their residues $\omega^{J}_{S_{\pm}}+i\omega^{K}_{S_{\pm}}$ are nowhere vanishing (2,0)-forms on $S_{\infty,\pm}=S_\pm$. The main goal in the smooth case is to find K\"ahler classes on $\tilde Z_\pm$ and a diffeomorphism $r$ from $S_+$ to $S_-$ such that the unique Ricci-flat metrics $\omega_{S_\pm}$ on the restrictions of the K\"ahler classes on $S_{\pm}$ satisfy \[(\omega_{S_+},\omega^{J}_{S_+},\omega^{K}_{S_+})=(r^*\omega^{J}_{S_-},r^*\omega_{S_-},-r^*\omega^{K}_{S_-}).\]
Remark that $Y_{\pm}$ are simply-connected. By Lefchetz hyperplane theorem, $S_{\pm}$ are also simply-connected. Since they have trivial canonical line bundles using the adjoint formula, they are all K3 surfaces. Therefore, the following theorem about the moduli space of marked hyperk\"ahler structures on the K3 surface proved by \cite{Pjateckii-SapiroSafarevic, BurnsRapoport, LooijengaPeters, Todorov, Siu} can be applied:
\begin{theorem}
(\cite{Besse})

Let $S$ be the smooth 4-manifold which underlies the minimal resolution of $\mathbb{T}^4/\mathbb{Z}_2$. Let $\Omega$ be the space of three cohomology classes $[\alpha_1], [\alpha_2], [\alpha_3]$ in $H^2(S, \mathbb{R})$ which satisfy the following conditions:

(1) (Integrability) \[\int_M \alpha_i\wedge\alpha_j = 2\delta_{ij}V.\]

(2) (Nondegeneracy) For any $[\Sigma]\in H_2(S, \mathbb{Z})$ with $[\Sigma]^2=-2$, there exists $i\in \{1, 2, 3\}$ with $[\alpha_i][\Sigma]\not=0.$

$\Omega$ has two components $\Omega^+$ and $\Omega^-$. For any $([\alpha_1], [\alpha_2], [\alpha_3])\in\Omega^+$, there exists on $S$ a hyperk\"ahler structure for which the cohomology classes of the K\"ahler forms $[\omega_i]$ are the given $[\alpha_i]$. It is unique up to tri-holomorphic isometries which induce identity on $H_2(S,\mathbb{Z})$. Moreover, any hyperk\"ahler structure on K3 surface must be constructed by this way.
\label{K3Torelli}
\end{theorem}

Thus, it suffices to find out the matching cohomology classes. In general, the matching data can only be found in the deformation classes of $\tilde Z_{\pm}$.

Remark that all K3 surfaces are diffeomorphic to $S$. Denote $H^2(S,\mathbb{Z})$ by $L$. $L$ is a lattice \[L=-E_8\oplus-E_8\oplus \begin{bmatrix}
    0 & 1 \\
    1 & 0
\end{bmatrix}^{\oplus 3}\] using the intersection form. For example, see \cite{SchulzTammaro} for a concrete description of the K3 lattice. It is clear that the quadratic form in fact acts on $L\otimes\mathbb{R}$. The set of elements in $L\otimes\mathbb{R}$ for which the square using this quadratic form is positive is called the positive cone.
It is easy to get the following proposition:
\begin{proposition}
The lattice $N_+\oplus N_-$ with quadratic form $\begin{bmatrix}
    -2 & 1 & 0\\
    1 & 4 & 0\\
    0 & 0 & 4
\end{bmatrix}$ can be embedded into $L$.
\label{Perpendicular-embedding}
\end{proposition}
\begin{proof}
Let $B_1,B_2,B_3$ and $C_1,C_2,C_3$ be the basis in the last three components. Embed $\tilde\Pi$ into a simple root of the first $-E_8$ component. The adjacent simple root of $-E_8$ is an element whose square equals to $-2$ and its product with $\tilde\Pi$ is 1. Let $-K_{Y_+}$ be the sum of this element with $B_1+C_1+B_2+C_2+B_3+C_3$, then it is clear that $(-K_{Y_+})^2=4$. Now let $-\frac{1}{4}K_{Y_-}=B_1+C_1-B_2-C_2$.
\end{proof}

A key proposition in \cite{CortiHaskinsNordstromPacini1} is the following:
\begin{proposition}
(Proposition 6.9 of \cite{CortiHaskinsNordstromPacini2})

Fix the embeddings $N_{\pm}\subset L$ as in Proposition \ref{Perpendicular-embedding}. Let $D_{N_{\pm}}$ be the Griffiths domains $\{\Pi\in\mathbb{P}(N_{\pm}^{\perp}\otimes\mathbb{C}):\Pi\wedge\bar\Pi>0\}$. Let $\mathcal{Y}_{\pm}$ be the deformation types of $Y_{\pm}$ such that there exist anti-canonical K3 divisors $S_{\pm}$ on $Y_{\pm}$ with $N_{\pm}$-polarised markings $h_{\pm}:L\cong H^2(S_{\pm},\mathbb{Z})$, which means, by definition, the restriction maps $H^2(Y_{\pm},\mathbb{Z})\rightarrow H^2(S_{\pm},\mathbb{Z})$ are equivalent to the inclusions $N_{\pm}\hookrightarrow L$ for the chosen isomorphisms $N_{\pm}\cong H^2(Y_{\pm},\mathbb{Z})$ and $h_{\pm}$. Then there exist

(1) sets $U_{\mathcal{Y}_{\pm}}\subset	D_{N_\pm}$ with complement locally finite unions of complex analytic submanifolds of positive codimensions;

(2) open subcones $\mathrm{Amp}_{\mathcal{Y}_{\pm}}$ of the positive cones of $N_{\pm}\otimes\mathbb{R}$ with the following property: for any $\Pi_{\pm}\in U_{\mathcal{Y}_{\pm}}$ and $k_{\pm}\in \mathrm{Amp}_{\mathcal{Y}_{\pm}}$, there exist $Y_{\pm}\in\mathcal{Y}_{\pm}$, smooth anti-canonical divisors $S_{\pm}$, and $N_{\pm}$-polarized markings $h_{\pm}:L\rightarrow H^2(S_{\pm},\mathbb{Z})$ such that $h_{\pm}(\Pi_{\pm})=H^{2,0}(S_{\pm})$ and $h_{\pm}(k_\pm)$ are the restrictions to $S_\pm$ of K\"ahler classes on $Y_\pm$.
\label{Surjectivity-deformation}
\end{proposition}

For any $\Pi_{\pm}\in U_{\mathcal{Y}_{\pm}}$ and $k_{\pm}\in \mathrm{Amp}_{\mathcal{Y}_{\pm}}$, choose (2,0)-forms $\omega^{J}_{S_\pm}+i\omega^{K}_{S_\pm}$ in $H^{2,0}(S_{\pm})$ and denote $h_{\pm}(k_\pm)$ by $\omega_{S_\pm}$. Then there exist K\"ahler classes on $\tilde Z_\pm$ such that their restrictions to $S_{\pm}$ are also $\omega_{S_\pm}$.

The following proposition was proved in \cite{CortiHaskinsNordstromPacini2} using Proposition \ref{Surjectivity-deformation}:
\begin{proposition}
(Proposition 6.18 of \cite{CortiHaskinsNordstromPacini2})
There exist $\Pi_{\pm}$ and $k_{\pm}$ such that the corresponding $Y_{\pm}$, $S_{\pm}$, $h_{\pm}^{-1}:H^2(S_{\pm},\mathbb{Z})\rightarrow L$, $\omega_{S_\pm},\omega^{J}_{S_\pm},\omega^{K}_{S_\pm}$ satisfy \[h_+^{-1}([\omega_{S_+}],[\omega^{J}_{S_+}],[\omega^{K}_{S_+}]) =h_-^{-1}([\omega^{J}_{S_-}],[\omega_{S_-}],-[\omega^{K}_{S_-}]).\]
\end{proposition}
\begin{proof}
The proof due to \cite{CortiHaskinsNordstromPacini2} is sketched here in order to see how to adjust it to the singular situation.

Let $T(\mathbb{R})$ be the subspace $N^{\perp}_+\cap N^{\perp}_-$ of $L\otimes\mathbb{R}$. Consider the real manifold \[\mathcal{A}=\mathbb{S}(\mathrm{Amp}_\mathcal{Y_+})\times \mathbb{S}(\mathrm{Amp}_\mathcal{Y_-})\times \mathbb{S}(T(\mathbb{R})),\]
where $\mathbb{S}(\mathrm{Amp}_\mathcal{Y_+})$, $\mathbb{S}(\mathrm{Amp}_\mathcal{Y_-})$ or $\mathbb{S}(T(\mathbb{R}))$ means the set of elements in $\mathrm{Amp}_\mathcal{Y_+}$, $\mathrm{Amp}_\mathcal{Y_-}$ or $T(\mathbb{R})$ whose square equals to 1 with respect to the quadratic form. There are two projections $\mathrm{pr}_\pm: \mathcal{A}\rightarrow D_{N_{\pm}}, (k_+,k_-,k)\rightarrow <k_\mp\pm ik>$, where $D_{N_{\pm}}$ are the Griffiths period domains. The key point of the proof due to \cite{CortiHaskinsNordstromPacini2} is the fact that the real analytic embedded submanifolds $\mathbb{S}(\mathrm{Amp}_{\mathcal{Y}_\mp})\times \mathbb{S}(T(\mathbb{R}))$ of $D_{N_{\pm}}$ are totally real with maximal dimensions. Therefore, it is easy to see that the set $\mathrm{pr}_+^{-1}(U_{\mathcal{Y}_+})\cap \mathrm{pr}_-^{-1}(U_{\mathcal{Y}_-})$ is non-empty.
\end{proof}

Up to here, it has been shown how to find the matching data for smooth asymptotically cylindrical Calabi-Yau manifolds. Remark that $\mathbb{S}(\mathrm{Amp}_\mathcal{Y_-})$ has a single point. Denote it by $k_-$. In order to get the matching data for the manifolds with nodal singularities, the cohomology class $k_+$ defined as $h_+^{-1}[\omega_{S_+}]$ must comes from the restriction of $-\frac{1}{2}c_1(K_{Y_+})$. This means that even though for any $k$ in the complement of real submanifolds with smaller dimensions in $\mathbb{S}(T(\mathbb{R}))$, it is still true that $<k_-+ik>\in U_{\mathcal{Y}_+}$, in general, $<k_+-ik>$ may not be in $U_{\mathcal{Y}_-}$ due to the restriction on the value of $k_+$. As the ``handcrafted gluing" problem in \cite{CortiHaskinsNordstromPacini2}, the following well-known lemma can be used solve this problem:

\begin{lemma}
(Chapter 3 of \cite{ReidK3}, cited as Lemma 7.15 of \cite{CortiHaskinsNordstromPacini2})

Let $S$ be a K3 surface, and let $A$ be a nef line bundle on $S$ with $A^2>0$ (i.e., $A$ is nef and big). Then either

(1) $|A|$ is monogonal, that is, $A=aE+\Gamma$, where $E$ and $\Gamma$ are holomorphic curves with $E^2=0$, $E\cdot \Gamma=1$, $\Gamma^2=-2$, and $a=1,2,3,...$, or

(2) $|A|$ has no fixed point, is base point free and either:

(2.1) the morphism given by $|A|$ is birational onto its image and an isomorphism away from a finite union of -2 curves, or

(2.2) $A$ is hyperelliptic, that is, one of the following cases holds: (2.2.1) $A^2=2$ and $S$ is a double cover of $\mathbb{CP}^2$; (2.2.2) $A=2B$ with $B^2=2$ and $S$ is a double cover of the Veronese surface; or (2.2.3) $S$ has an elliptic pencil $E$ with $A\cdot E=2$.
\label{ReidK3lemma}
\end{lemma}

Motivated by Lemma \ref{ReidK3lemma}, the first thing to check is linear combinations $a\tilde\Pi-bK_{Y_+}-\frac{c}{4}K_{Y_-}$ satisfying $a\tilde\Pi-bK_{Y_+}-\frac{c}{4}K_{Y_-}\in L$, \[(a\tilde\Pi-bK_{Y_+}-\frac{c}{4}K_{Y_-})^2=-2\] and \[(a\tilde\Pi-bK_{Y_+}-\frac{c}{4}K_{Y_-})\cdot (-K_{Y_+})=0.\] The first condition implies that $a, b, c\in\mathbb{Z}$. The second condition implies that $-2a^2+2ab+4b^2+4c^2=-2$. The third condition implies that $a+4b=0$. Therefore $-36b^2+4c^2=-2$. This is impossible. Thus, for all $C\in L$ satisfying $C^2=-2$ and $C\cdot (-K_{Y_+})=0$, $C$ can not be in $(N_+\otimes\mathbb{R})\oplus (N_-\otimes\mathbb{R})$. So $C^\perp$ intersects $T(\mathbb{R})$ transversally. There is no difficulty to find $k\in \mathbb{S}(T(\mathbb{R}))$ such that $k$ does not lie in $C^\perp$ for all such $C$ and $<k_-+ik>\in U_{\mathcal{Y}_+}$ is still true. Similarly, it is possible to assume that for all $E\in L$ satisfying $E^2=0$, $E\cdot (-K_{Y_+})=0$ and $E\cdot (-\frac{1}{4}K_{Y_-})=2$, $k$ does not lie in $C^\perp$.

By Theorem \ref{K3Torelli}, there exists a K3 surface $S_-$ with a marking $h_-$ such that \[h_-([\omega^{J}_{S_-}],[\omega_{S_-}],-[\omega^{K}_{S_-}])=(k_+,k_-,k).\] It is smooth and does not contain any -2 curve because $k_+$ and $k$ can not vanish simutanously on it. $A=-\frac{1}{4}c_1(K_{Y_-})=2k_-$ lies in $H^{2}(S_-,\mathbb{Z})$ and is a K\"ahler (1,1)-class on $S_-$, so $A$ is a nef line bundle on $S_-$ with $A^2=4>0$. By Lemma \ref{ReidK3lemma}, the morphism given by $|A|$ is an isomorphic onto its image. By Kodaira vanishing theorem and Riemann-Roch theorem, the morphism given by $|A|$ is in fact an isomorphic onto a smooth quartic surface in $\mathbb{CP}^3$. This solves the matching problem.

In an email from Nordstr\"om to the author, he said that Lemma 2.4 of \cite{Fukuoka} and Lemma 5.18 of \cite{CrowleyNordstrom} may provide more examples of matching data for the singular twisted connected sum problem.

\section{Harmonic forms on the nodal cone}
This section deals with the homogenous harmonic forms on strongly regular Calabi-Yau cones $C=C(F)$ with complex dimension 3 defined in Definition \ref{Definition-strongly-regular-cone}. Some results in this section have been proved in \cite{CheegerSpectral1}, \cite{CheegerSpectral2}, \cite{KarigiannisLotay} and \cite{FoscoloHaskinsNordstorm} .

The starting point is the definition of homogenous forms on $C$.
\begin{definition}
A $p$-form $\gamma=r^{\lambda}(r^{p-1}dr\wedge\alpha+r^p\beta)$ is called homogenous of rate $\lambda$ if $\frac{\partial}{\partial r}\alpha=\frac{\partial}{\partial r}\beta=0$.
\end{definition}

A direct calculation shows the following:
\begin{proposition}
Let $\gamma=r^{\lambda}(r^{p-1}dr\wedge\alpha+r^p\beta)$ be a $p$-form on $C$. Let $d_F$ and $d_F^*$ be the operator on each sphere $F$ using the metric on the unit sphere, then
\[d_C\gamma=r^{\lambda+p-1}dr\wedge((\lambda+p)\beta+r\frac{\partial}{\partial r}\beta-d_F\alpha)+r^{\lambda+p}d_F\beta,\]
\[d_C^*\gamma=r^{\lambda+p-3}dr\wedge(-d_F^*\alpha)+r^{\lambda+p-2}(-(\lambda-p+6)\alpha-r\frac{\partial}{\partial r}\alpha+d_F^*\beta),\]
\[\begin{split}
\Delta_C\gamma=&r^{\lambda+p-3}dr\wedge(\Delta_F\alpha-(\lambda+p-2)(\lambda-p+6)\alpha-(r\frac{\partial}{\partial r})^2\alpha\\&-(2\lambda+4)r\frac{\partial}{\partial r}\alpha-2d_F^*\beta)
+r^{\lambda+p-2}(\Delta_F\beta-(\lambda+p)(\lambda-p+4)\beta\\
&-(r\frac{\partial}{\partial r})^2\beta-(2\lambda+4)r\frac{\partial}{\partial r}\beta-2d_F\alpha).
\end{split}\]
\label{Laplacian-local-coordinate}
\end{proposition}

So the homogenous harmonic forms are closely related to the eigenforms on the link $F$.

\begin{definition}
Using Hilbert-Schmidt theorem, assume that $\phi_{0,j}$ are orthogonal basis of $L^2(\Lambda^0(F))$ with $\Delta_F\phi_{0,j}=\mu_{0,j}\phi_{0,j}$. Then $d_F\phi_{0,j}$ are orthogonal to each other because \[(d_F\phi_{0,j},d_F\phi_{0,j'})=(d_F^*d_F\phi_{0,j},\phi_{0,j'})=\mu_{0,j}(\phi_{0,j},\phi_{0,j'})=0\] if $j\not=j'$. By Hilbert-Schmidt theorem applied to the Laplacian oprator acting on 1-forms, it is possible to assume that $d_F\phi_{0,j}$ for $j=2,3,...$ and $\phi_{1,j}$ for $j=1,2,...$ are orthogonal basis of $L^2(\Lambda^1(F))$ with \[\Delta_Fd_F\phi_{0,j}=\mu_{0,j}d_F\phi_{0,j}\] and \[\Delta_F\phi_{1,j}=\mu_{1,j}\phi_{1,j}.\] It is clear that $d_F^*d_F\phi_{0,j}=\mu_{0,j}\phi_{0,j}$ while $d_F^*\phi_{1,j}=0$.

Inductively, for $p=0,1,2,...5$, $d_F\phi_{p,j}$ are orthogonal to each other, so it is possible to assume that $d\phi_{p,j}$ for $j=h_p+1,h_p+2,...$ and $\phi_{p+1,j}$ for $j=1,2,3,...$ are orthogonal basis of $L^{p+1}(\Lambda^2(F))$ with \[\Delta_Fd_F\phi_{p,j}=\mu_{p,j}d_F\phi_{p,j}\] and \[\Delta_F\phi_{p+1,j}=\mu_{p+1,j}\phi_{p+1,j},\] where $h_p$ is the dimension of the cohomology group $H^p(F,\mathbb{R})$.
\end{definition}

The relationship between homogenous harmonic forms on $C$ and eigenforms on $F$ are given by the following:
\begin{definition}
Choose $\lambda=-2$ in Proposition \ref{Laplacian-local-coordinate}, then
\[\begin{split}
\Delta_C\gamma=&r^{p-5}dr\wedge(\Delta_F\alpha+(p-4)^2\alpha-(r\frac{\partial}{\partial r})^2\alpha-2d_F^*\beta)\\&
+r^{p-4}(\Delta_F\beta+(p-2)^2\beta-(r\frac{\partial}{\partial r})^2\beta-2d_F\alpha).
\end{split}\]
So it is important to study the eigenvalues of the self-adjoint operator
\[(\alpha,\beta)\rightarrow (\Delta_F\alpha+(p-4)^2\alpha-2d_F^*\beta,\Delta_F\beta+(p-2)^2\beta-2d_F\alpha)\]
from a subspace of $L^2(\Lambda^{p-1}F\oplus\Lambda^{p}F)$ to another subspace of $L^2(\Lambda^{p-1}F\oplus\Lambda^{p}F)$. It is easy to see that the eigenforms are, up to linear combinations,

(1) $(d_F\phi_{p-2,j},0)$ with eigenvalue $\mu_{p-2,j}+(p-4)^2$,

(2) $(0,\phi_{p,j})$ with eigenvalue $\mu_{p,j}+(p-2)^2$,

(3) $(\phi_{p-1,j},0)$ with eigenvalue $(p-4)^2$ if $\mu_{p-1,j}=0$,

(4) $(\phi_{p-1,j},\frac{3-p\pm\sqrt{(p-3)^2+\mu_{p-1,j}}}{\mu_{p-1,j}}d_F\phi_{p-1,j})$ with eigenvalue
\[(\sqrt{(p-3)^2+\mu_{p-1,j}}\mp1)^2\] if $\mu_{p-1,j}\not=0$.

Using the identification $\gamma=r^{-2}(r^{p-1}dr\wedge\alpha+r^p\beta)$ between $\gamma\in\Lambda^p(C(F))$ with $(\alpha,\beta)\in\Lambda^{p-1}F\oplus\Lambda^{p}F$, the eigenforms are denoted by $\hat\phi_{p,j}\in\Lambda^p(C(F))$ with eigenvalues $\hat\mu_{p,j}$. By Hilbert-Schmidt theorem, they form an $L^2$ basis. By Proposition \ref{Harmonic-form-decomposition}, any homogenous form is harmonic if and only if it is the linear combination of $r^{\pm\sqrt{\hat\mu_{p,j}}}\hat\phi_{p,j}$.
\label{Definition-hat-phi}
\end{definition}

A more precise decomposition is the following:
\begin{proposition}
A harmonic homogenous $p$-form $\gamma=r^{\lambda}(r^{p-1}dr\wedge\alpha+r^p\beta)$ can be written as a linear combination of

(1) closed but not coclosed harmonic homogenous forms $r^{\lambda}r^{p-1}dr\wedge d_F\phi_{p-2,j}$ with $\mu_{p-2,j}=(\lambda+p-2)(\lambda-p+6)\not=0$,

(2) closed but not coclosed harmonic homogenous forms $r^{\lambda}r^{p-1}dr\wedge\phi_{p-1,j}$ with $\mu_{p-1,j}=\lambda+p-2=0$, and $\lambda\not=-2$.

(3) closed and coclosed homogenous forms $r^{\lambda}r^{p-1}dr\wedge\phi_{p-1,j}$ with \[\mu_{p-1,j}=\lambda-p+6=0,\]

(4) closed and coclosed homogenous forms \[r^{\lambda}((\lambda+p)r^{p-1}dr\wedge\phi_{p-1,j}+r^pd_F\phi_{p-1,j})\] with $\mu_{p-1,j}=(\lambda+p)(\lambda-p+6)\not=0$,

(5) neither closed nor coclosed homogenous harmonic forms \[r^{\lambda}(-(\lambda-p+4)r^{p-1}dr\wedge\phi_{p-1,j}+r^pd_F\phi_{p-1,j})\] with $\mu_{p-1,j}=(\lambda+p-2)(\lambda-p+4)\not=0$ and $\lambda\not=-2$,

(6) closed and coclosed homogenous forms $r^{\lambda}(r^p\phi_{p,j})$ with $\mu_{p,j}=\lambda+p=0$,
 and

(7) coclosed but not closed harmonic homogenous forms $r^{\lambda}(r^p\phi_{p,j})$ with \[\mu_{p,j}=(\lambda+p)(\lambda-p+4)\] and $\lambda+p\not=0$.
\label{Harmonic-form-decomposition}
\end{proposition}
\begin{proof}
Use Hilber-Schimdt theorem to write $\alpha$ and $\beta$ as the generalized Fourier series
\[\alpha=\sum_{j=h_{p-2}+1}^{\infty}\alpha_{p-2,j}d_F\phi_{p-2,j}+\sum_{j=1}^{\infty}\alpha_{p-1,j}\phi_{p-1,j},\]
and
\[\beta=\sum_{j=h_{p-1}+1}^{\infty}\beta_{p-1,j}d_F\phi_{p-1,j}+\sum_{j=1}^{\infty}\beta_{p,j}\phi_{p,j}.\]
Since $d_F\phi_{p-2,j}$ is perpendicular to $d_F^*\beta$, the harmonic assumption implies that $\alpha_{p-2,j}=0$ unless $\mu_{p-2,j}=(\lambda+p-2)(\lambda-p+6)$. When $\mu_{p-1,j}=0$, using the fact that $\phi_{p-1,j}$ is also perpendicular to $d_F^*\beta$, the harmonic assumption implies that $\alpha_{p-1,j}=0$ unless $(\lambda+p-2)(\lambda-p+6)=0$. Similarly, $\beta_{p,j}=0$ unless $\mu_{p,j}=(\lambda+p)(\lambda-p+4)$.

For $j=h_{p-1}+1, h_{p-1}+2, ...$, the equation \[\Delta_F\beta-(\lambda+p)(\lambda-p+4)\beta-2d_F\alpha=0\] implies that $\alpha_{p-1,j}=\frac{1}{2}(\mu_{p-1,j}-(\lambda+p)(\lambda-p+4))\beta_{p-1,j}$. So the equation \[\Delta_F\alpha-(\lambda+p-2)(\lambda-p+6)\alpha-2d_F^*\beta=0\] implies that
\[\frac{1}{4}(\mu_{p-1,j}-(\lambda+p-2)(\lambda-p+6))(\mu_{p-1,j}-(\lambda+p)(\lambda-p+4))=\mu_{p-1,j}\] unless $\alpha_{p-1,j}=\beta_{p-1,j}=0$. Note that \[\frac{1}{4}(\mu_{p-1,j}-(\lambda+p-2)(\lambda-p+6))(\mu_{p-1,j}-(\lambda+p)(\lambda-p+4))=\mu_{p-1,j}\] if and only if $\mu_{p-1,j}=(\lambda+p)(\lambda-p+6)$ or $\mu_{p-1,j}=(\lambda+p-2)(\lambda-p+4)$.
\end{proof}

Recall the definition of $\mathcal{K}_i(\lambda_i)$ in Definition \ref{Definition-critical}. It is easy to prove the following proposition:

\begin{proposition}
Suppose that \[\gamma\in\mathcal{K}_i(\lambda_i)\] for the Hodge Laplacian acting on p-forms. Then up to linear combinations, either

(1) $\gamma=r^{\pm\sqrt{\hat\mu_{p,j}}}\hat\phi_{p,j}$ is homogenous with $\hat\mu_{p,j}\in\mathbb{R}$, or

(2) $\gamma=\log r\hat\phi_{p,j}$ with $\hat\mu_{p,j}=0$.
\label{No-log-term}
\end{proposition}
\begin{proof}
Write $\gamma$ as
\[\gamma=\sum_{j=1}^{\infty}\gamma_j(r)\hat\phi_{p,j},\]
then
\[(\hat\mu_{p,j}-(r\frac{d}{dr})^2)\gamma_j=0.\]
If $\hat\mu_{p,j}\not=0$, $\gamma_j$ is the linear combination of $r^{\pm\sqrt{\hat\mu_{p,j}}}$. When $\hat\mu_{p,j}=0$, $\gamma_j$ is the linear combination of 1 and $\log r$.
\end{proof}

The next goal is the estimate of eigenvalues:
\begin{proposition}
(Obata \cite{Obata}) $\mu_{0,1}=0$ and $\phi_{0,1}=1$. For all $j=2,3,4,...$, $\mu_{0,j}>5$.
\label{Obata}
\end{proposition}
\begin{proof}
This follows from \cite{Obata} because the metric on $F$ is Einstein with scalar curvature 20 and $F$ is not isometric to the sphere.
\end{proof}

\begin{proposition}
$\mu_{1,j}\ge8$ for all $j=1,2,3...$, moreover, when $\mu_{1,j}=8$, then $r^2\phi_{1,j}^\#$ is a Killing vector field on $C$, where $\phi_{1,j}^\#$ means the metric dual using $g_C$.
\label{Killing-metric}
\end{proposition}
\begin{proof}
This is similar to Lemma 3.11 of \cite{KarigiannisLotay}.
\end{proof}

\begin{proposition}
If $\phi_{2,j}$ is a primitive (1,1)-form on $C$, then either $\mu_{2,j}=0$ or $\mu_{2,j}\ge 9$.
\label{Primitive-one-one}
\end{proposition}
\begin{proof}
It is proved in the proof of Proposition 4.9. (iii) of \cite{FoscoloHaskinsNordstorm}.
\end{proof}

The homogenous harmonic forms on $C$ can be studied using the estimates of $\mu_{p,j}$.
\begin{proposition}
Let $\gamma$ be a homogeneous harmonic 1-form on $C$ with rate in $[-3,0]$, then $\gamma=0$.
\label{Harmonic-one-form-below-zero}
\end{proposition}
\begin{proof}
The is an immediate corollary of Proposition \ref{Harmonic-form-decomposition}, Proposition \ref{Obata}, and Proposition \ref{Killing-metric}.
\end{proof}

Recall the following theorem essentially due to Cheeger-Tian:
\begin{proposition}
(Theorem 7.27 of \cite{CheegerTian}, see also Lemma 2.17 of \cite{HeinSun}). Let $\gamma$ be a homogeneous 1-form on $C$ with rate in $(0,1]$. Then $\gamma$ is harmonic if and only if, up to linear combinations, either $\gamma=d(r^{\sqrt{\mu_{0,j}+4}-2}\phi_{0,j})$ with $\mu_{0,j}\in(5,12]$ or $\gamma=r^2\phi_{1,j}$, where $\mu_{1,j}=8$ or $\gamma=rdr$.
\label{Cheeger-Tian}
\end{proposition}

Remark that Proposition \ref{Cheeger-Tian} has been adjusted to the strongly regular Calabi-Yau cone case.

\begin{lemma}
Choose $r\phi_{1,1}=e^1=-Je^0=-Jdr$. Then $\mu_{1,1}=8$. If $\mu_{1,j}=8$, then there exists a constant $k_j$ such that \[L_{r^2(\phi_{1,j}-k_j\phi_{1,1})^\#}\omega
=L_{r^2(\phi_{1,j}-k_j\phi_{1,1})^\#}\mathrm{Re}\Omega
=L_{r^2(\phi_{1,j}-k_j\phi_{1,1})^\#}\mathrm{Im}\Omega=0.\]
\label{Killing-SU(3)}
\end{lemma}
\begin{proof}
By Proposition \ref{Killing-metric}, $r^2\phi_{1,j}^\#$ preserves the metric $g_C$. Let $e^{sr^2\phi_{1,j}^\#}$ be the one-parameter subgroup generated by $r^2\phi_{1,j}^\#$, then $e^{sr^2\phi_{1,j}^\#}$ preserves the metric $g_C$. Since $\omega$, $\Omega$ are parallel unit-length forms on $C$, $(e^{sr^2\phi_{1,j}^\#})^*\omega$ and $(e^{sr^2\phi_{1,j}^\#})^*\Omega$ are also parallel. The holonomy group of $g_C$ equals to $\mathrm{SU}(3)$ instead of a proper subgroup of $\mathrm{SU}(3)$, so the only unit-length parallel 2-forms are $\pm\omega$. By continuity, $(e^{sr^2\phi_{1,j}^\#})^*(\omega)=\omega$ for all $s$. So $L_{r^2\phi_{1,j}^\#}\omega=0$. Any unit-length parallel 3-form must be $e^{i\theta_j(s)}\Omega$. So after differentiating, there exists a constant $k_j$ such that $L_{r^2\phi_{1,j}^\#}\mathrm{Re}\Omega=3k_j\mathrm{Im}\Omega$ and $L_{r^2\phi_{1,j}^\#}\mathrm{Im}\Omega=-3k_j\mathrm{Re}\Omega$.
When $j=1$, $r^2\phi_{1,1}^\#=rJ\frac{\partial}{\partial r}$. So \[L_{r^2\phi_{1,1}^\#}\mathrm{Re}\Omega=d(r^2\phi_{1,1}^\#\lrcorner\mathrm{Re}\Omega)=d(r\frac{\partial}{\partial r}\lrcorner\mathrm{Im}\Omega)=3\mathrm{Im}\Omega.\]
Similarly $L_{r^2\phi_{1,1}^\#}\mathrm{Im}\Omega=-3\mathrm{Re}\Omega$. So \[L_{r^2(\phi_{1,j}-k_j\phi_{1,1})^\#}\omega
=L_{r^2(\phi_{1,j}-k_j\phi_{1,1})^\#}\mathrm{Re}\Omega
=L_{r^2(\phi_{1,j}-k_j\phi_{1,1})^\#}\mathrm{Im}\Omega=0.\]
\end{proof}

\begin{proposition}
A homogenous 1-form $\gamma$ with rate $\lambda\in[-3,1]$ is harmonic if and only if up to linear combinations, either

(1) $\gamma=d_C(r^{\sqrt{\mu_{0,j}+4}-2}\phi_{0,j})$ with $\mu_{0,j}\in(5,12]$,

(2) $\gamma=d_C r^2$,

(3) $\gamma=(d_C r^2)^\#\lrcorner\omega$, $d_C\gamma=4\omega$ and $d_C((d_C r^2)^\#\lrcorner\mathrm{Re}\Omega)=6\mathrm{Re}\Omega$
or

(4) $\gamma=(d_C(r^2\phi_{0,j}))^\#\lrcorner\omega$ with $\mu_{0,j}=12$, $\lambda=1$ and $d_C((d_C r^2\phi_{0,j})^\#\lrcorner\mathrm{Re}\Omega)$ equals to a linear combination of $\mathrm{Re}\Omega$ and $\mathrm{Im}\Omega$.
\label{Harmonic-one-form}
\end{proposition}
\begin{proof}
By Proposition \ref{Harmonic-one-form-below-zero} and Proposition \ref{Cheeger-Tian}, up to linear combinations, either (1) or (2) holds or \[\gamma=r^2\phi_{1,j}
=(Jr^2\phi_{1,j})^\#\lrcorner\omega\] with $\mu_{1,j}=8$. Remark that \[d_C((Jr^2\phi_{1,j})^\#\lrcorner\mathrm{Re}\Omega)=d_C((r^2\phi_{1,j})^\#\lrcorner\mathrm{Im}\Omega)\] is a multiple of $\mathrm{Re}\Omega$ by Proposition \ref{Killing-SU(3)}.

The 1-form $r^2\phi_{1,j}$ is harmonic. By the $\mathrm{SU}(3)$ structure, $Jr^2\phi_{1,j}$ is also a harmonic homogenous 1-form of rate 1. By Lemma \ref{Cheeger-Tian}, $Jr^2\phi_{1,j}$ is a linear combination of $d_C r^2=2rdr$, $d_C(r^2\phi_{0,j'})$ with $\mu_{0,j'}=12$, and $r^2\phi_{1,j''}$ with $\mu_{1,j''}=8$. By Proposition \ref{Killing-SU(3)}, $d_C((r^2\phi_{1,j''})^\#\lrcorner\omega)=0$ and $d_C((r^2\phi_{1,j''})^\#\lrcorner\mathrm{Re}\Omega)$ is a multiple of $\mathrm{Im}\Omega$. Since a linear combination of $r^2\phi_{1,j'''}$ is closed if and only if it is 0, by Lemma \ref{Cheeger-Tian}, $(r^2\phi_{1,j''})^\#\lrcorner\omega$ equals to a linear combination of forms in (1) and (2).
\end{proof}

\begin{proposition}
A homogenous 2-form $\gamma$ with rate $\lambda\in[-3,1]$ is harmonic if and only if up to linear combinations, either

(1) $\gamma=\phi_{2,1}$ with $\lambda=-2$,

(2) $\gamma=d_C(r^{\lambda+2}\phi_{1,j})$ with $\lambda\in[0,1]$ and $\mu_{1,j}=(\lambda+2)(\lambda+4)$,

(3) $\gamma=\omega$ with $\lambda=0$,

(4) $\gamma=rdr^\#\lrcorner\mathrm{Re}\Omega$ with $\lambda=1$,

(5) $\gamma=r^2\phi_{1,j}^\#\lrcorner\mathrm{Re}\Omega$ with $\lambda=1$,
or

(6) \[\gamma=(d_C(r^{\lambda+2}\phi_{0,j}))^\#\lrcorner\mathrm{Re}\Omega\] with $\lambda\in(0,1]$ and $\mu_{0,j}=(\lambda+3)^2-4\in(5,12]$.
\label{Harmonic-two-form}
\end{proposition}
\begin{proof}
There are two ways of decomposing homogenous harmonic 2-forms. The first way is to decompose it as in Proposition \ref{Harmonic-form-decomposition}. The other way is to decompose it into (2,0), (0,2), multiple of $\omega$ and primitive (1,1) components. Assume that $\gamma=\gamma'+\gamma''+\gamma'''$, where $\gamma'$ is a linear combination of (2,0), (0,2), multiple of $\omega$ homogenous harmonic forms, $\gamma''$ is a linear combination of all homogenous harmonic forms in Proposition \ref{Harmonic-form-decomposition} except the type (7), and $\gamma'''$ is primitive (1,1) form of type (7) in Proposition \ref{Harmonic-form-decomposition}. By Proposition \ref{Primitive-one-one}, $\gamma'''=0$. It is easy to see that only the type (4) and type (6) components of $\gamma''$ in Proposition \ref{Harmonic-form-decomposition} may be non-zero. They correspond to $\phi_{2,1}$ and $d_C(r^{\lambda+2}\phi_{1,j})$.

By the $\mathrm{SU}(3)$ structure, the multiple of $\omega$ component of $\gamma'$ equals to a homogenous harmonic function of rate $\lambda$ times $\omega$. By Proposition \ref{Harmonic-form-decomposition}, it equals to a constant multiple of $\omega$ because $\mu_{0,2}>5$. Still by the $\mathrm{SU}(3)$ structure, the (2,0) and (0,2) component must be the contraction of the metric dual of a homogenous harmonic 1-form of rate $\lambda$ with $\mathrm{Re}\Omega$. By Proposition \ref{Harmonic-one-form-below-zero} and Lemma \ref{Cheeger-Tian}, it must be a linear combination of $r^2\phi_{1,j}$ with $\mu_{1,j}=8$, $d_C(r^{\sqrt{\mu_{0,j'}+4}-2}\phi_{0,j'})$ with $\mu_{0,j'}=(\lambda+3)^2-4\in(5,12]$, and $rdr$.
\end{proof}

\begin{corollary}
Suppose that $\gamma_2$ and $\gamma_3$ are homogenous 2-form and 3-form with same rate $\lambda\in(-2,0]$ on $C$. Then $\gamma_2$ and $\gamma_3$ are both closed and coclosed if and only if for the 3-form $\gamma$ defined as $d\theta\wedge\gamma_2+\gamma_3$ on $C\times \mathbb{S}^1$, up to linear combinations, either

(1) $\gamma=\varphi=\mathrm{Re}\Omega+d\theta\wedge\omega$ with $\lambda=0$,

(2) $\gamma=d_C((d_C r^2)^\#\lrcorner\varphi)=6\mathrm{Re}\Omega+4d\theta\wedge\omega$ with $\lambda=0$,

(3) $\gamma=d_C(J(d_C r^2)^\#\lrcorner\varphi)=6\mathrm{Im}\Omega$ with $\lambda=0$, or

(4) \[\gamma=d_C((d_C(r^{\lambda+2}\phi_{0,j}))^\#\lrcorner\varphi)\] with $\lambda\in(-1,0]$ and $\mu_{0,j}=(\lambda+4)^2-4\in(5,12]$.
\label{Two-form-and-three-form-representation}
\end{corollary}
\begin{proof}
Decompose $\gamma_2$ and $\gamma_3$ as in Proposition \ref{Harmonic-form-decomposition}. The closeness implies that the type (5) and type (7) components in Proposition \ref{Harmonic-form-decomposition} vanish. The cocloseness implies that the type (1) and type (2) components also vanish. The type (3) and type (6) components also vanish by the assumption on $p$ and $\lambda$. So $\gamma_2$ and $\gamma_3$ are of type (4) in Proposition \ref{Harmonic-form-decomposition}. This implies that they are in the image of $d_C$ acting on homogenous harmonic 1-forms or 2-forms. The result follows easily from Proposition \ref{Harmonic-one-form} and Proposition \ref{Harmonic-two-form}.
\end{proof}

\section{Doubling construction of Calabi-Yau threefolds}

This section proves Theorem \ref{Doubling-construction-of-Calabi–Yau-threefolds}.

Recall that in the setting of Theorem \ref{Doubling-construction-of-Calabi–Yau-threefolds}, $M$ is glued by $V_\pm$. The first goal is to study the operator $d+d^*$ from odd-degree forms to even-degree forms on $\mathbb{S}^1$ times $C$, $V_{\pm}$ or $M$. Let $\theta$ be the standard variable on $\mathbb{S}^1$. Then any odd-degree form can be expressed as
\[\begin{split}
\gamma&=\gamma^1+\gamma^3+\gamma^5+\gamma^7\\
&=(d\theta\wedge\gamma_0+\gamma_1)+(d\theta\wedge\gamma_2+\gamma_3)
+(d\theta\wedge\gamma_4+\gamma_5)+(d\theta\wedge\gamma_6),
\end{split}\]
where $\gamma_p$ is a degree $p$-form on each slice. A direct calculation shows that
\[\begin{split}
(d_{\mathbb{S}^1\times C}+d_{\mathbb{S}^1\times C}^*)\gamma&=\sum_{p=0,2,4,6}(d_C\gamma_{p-1}+d_C^*\gamma_{p+1}-\frac{\partial\gamma_p}{\partial\theta})\\
&-\sum_{p=1,3,5}d\theta\wedge(d_C\gamma_{p-1}+d_C^*\gamma_{p+1}-\frac{\partial\gamma_p}{\partial\theta}),
\end{split}\]
where $d_C$ and $d_C^*$ mean doing the $d$ and $d^*$ operators on each slice. $C$ may be replaced by $V_{\pm}$ or $M$.

Similarly, any even-degree form can be expressed as
\[\begin{split}
\gamma&=\gamma^0+\gamma^2+\gamma^4+\gamma^6\\
&=(\gamma_0)+(d\theta\wedge\gamma_1+\gamma_2)+(d\theta\wedge\gamma_3+\gamma_4)
+(d\theta\wedge\gamma_5+\gamma_6),
\end{split}\]
Another direct calculation shows that
\[\begin{split}
(d_{\mathbb{S}^1\times C}+d_{\mathbb{S}^1\times C}^*)\gamma
&=-\sum_{p=0,2,4,6}d\theta\wedge(d_C\gamma_{p-1}+d_C^*\gamma_{p+1}-\frac{\partial\gamma_p}{\partial\theta})\\
&+\sum_{p=1,3,5}(d_C\gamma_{p-1}+d_C^*\gamma_{p+1}-\frac{\partial\gamma_p}{\partial\theta}).
\end{split}\]
When $\gamma$ is $\mathbb{S}^1$-invariant, then there is no $\frac{\partial}{\partial\theta}$ part. Therefore, it suffices to study $d+d^*$ on $C$, $M$ or $V_{\pm}$.

\begin{proposition}
Suppose that $\delta>0$ is small enough. \[\gamma\in W^{k,2}_{-3-\delta,...,-3-\delta,\delta}(\Lambda^{\mathrm{even}}(V_{\pm}))\] or \[\gamma\in W^{k,2}_{-3-\delta,...,-3-\delta}(\Lambda^{\mathrm{even}}(M)).\]
If $(d+d^*)\gamma=0$, then $d\gamma=0$.
\label{Closeness-harmonic-form}
\end{proposition}
\begin{proof}
Assume that $\gamma\in W^{k,2}_{-3-\delta,...,-3-\delta}(\Lambda^{\mathrm{even}}(M))$. Using the definition of the Hodge Laplacian $\Delta$, it is easy to see that $\Delta\gamma_0=\Delta\gamma_2=\Delta\gamma_4=\Delta\gamma_6=0$. Near each singular point $x_i\in V^{\mathrm{sing}}_{\pm}$, $\gamma\in W^{k,2}_{-3-\delta}$. By Proposition \ref{Harmonic-form-decomposition}, Proposition \ref{No-log-term} and Propositon \ref{Obata}, there is no 0-form in $\mathcal{K}_i(\lambda_i)$ with rate $\lambda_i\in(-3-\delta,-\delta)$. So by the definition of $\mathcal{P}_i(\lambda_i)$ and Theorem \ref{Polyhomogenous-L-totally-character}, $\gamma_0\in W^{k,2}_{-\delta,...,-\delta}(\Lambda^2 M)$.

Similarly, by Proposition \ref{No-log-term}, Proposition \ref{Harmonic-two-form}, and Theorem \ref{Polyhomogenous-L-totally-character}, $\gamma_2$ can be written as the linear combination of $\phi_{2,1}$, $\log r \phi_{2,1}$ and an element in $W^{k,2}_{-2+\delta}$ near $x_i$. Remark that the difference between $\phi_{2,1}$ or $\log r \phi_{2,1}$ and the corresponding element in $\mathcal{P}_i(\lambda_i)$ lies $W^{k,2}_{-2+\delta}$. By Hodge duality on $M$, $\gamma_4$ can be written as the linear combination of $*\phi_{2,1}$, $*\log r \phi_{2,1}$ and an element in $W^{k,2}_{-2+\delta}$ near $x_i$. However, the equation \[d\gamma_2+d^*\gamma_4=0\] implies that the coefficients of the log terms vanish by Proposition \ref{Laplacian-local-coordinate}. So near $x_i$, $d\gamma_p\in W^{k,2}_{-3+\delta}$ for $p=0,2,4$. Globally, $d\gamma_p\in W^{k,2}_{-3+\delta,...,-3+\delta}$.

It is easy to see that the boundary term in the integral \[\int_{r_i>r_0}(d\gamma_p,d\gamma_p)-(\gamma_p,d^*d\gamma_p)\] goes to 0 when $r_0$ goes to 0. Since $d^*d\gamma_p=-d^*d^*\gamma_{p+2}=0$, it follows that $d\gamma_p=0$. The $V_{\pm}$ case is similar.
\end{proof}

Recall the definitions of $V_{\pm}$ and $\tilde V_{\pm}$ in Section 4. In this section, remark that the definitions of $V_+$ and $\tilde V_+$ remain unchanged but the definitions of $V_-$ and $\tilde V_-$ have been changed to another copy of $V_+$ and $\tilde V_+$. By definition, $\tilde V_{\pm}$ is the small resolution of $V_{\pm}$. Locally, near each point $x\in V^{\mathrm{sing}}_{\pm}$, the neighborhood of $x$ is topologically a cone over $\mathbb{S}^2\times\mathbb{S}^3$. The corresponding set in $\tilde V_{\pm}$ is topologically $\mathbb{S}^2\times\mathbb{B}^4$, where $x$ is replaced by $\mathbb{S}^2\times\{0\}$.

The next goal is to study the Hodge theory on $V_{\pm}$. It was pioneered by Cheeger \cite{Cheeger} using a slightly difference version of weighted analysis and followed by many people including Melrose \cite{Melrose}.

Recall that $F_\infty=\mathbb{S}^1\times S_{\pm}$. As in Section 6, by Hilbert-Schmidt theorem, assume that $d_{F_\infty}\phi_{p-1,j,\infty}$ and $\phi_{p,j,\infty}$ are orthogonal basis for $L^2(\Lambda^p(F_\infty))$ satisfying $\Delta_{F_\infty}\phi_{p,j,\infty}=\mu_{p,j,\infty}\phi_{p,j,\infty}$. Moreover, let $h_{p,\infty}$ be the $p$-th betti number of $F_\infty$. Then

\begin{lemma}
Consider the Hodge Laplacian operator $\Delta$ acting on $p$-forms on $[T,\infty)\times F_\infty$.
\[\mathcal{K}_\infty(0)=\mathrm{Span}\{\phi_{p,j,\infty},t\phi_{p,j,\infty}\}_{j=1}^{h_{p,\infty}} \oplus\mathrm{Span}\{dt\wedge \phi_{p-1,j,\infty},tdt\wedge \phi_{p-1,j,\infty}\}_{j=1}^{h_{p-1,\infty}}.\]
\label{Polyhomogenous-end}
\end{lemma}
\begin{proof}
Any $\gamma\in\mathcal{K}_\infty(0)$ can be written as \[\gamma=dt\wedge\alpha+\beta.\] Then \[\Delta_{[T,\infty)\times\mathbb{S}^1\times S_{\pm}}\gamma=dt\wedge(\Delta_{\mathbb{S}^1\times S_{\pm}}-\frac{\partial^2}{\partial t^2})\alpha+(\Delta_{\mathbb{S}^1\times S_{\pm}}-\frac{\partial^2}{\partial t^2})\beta=0.\]
Consider the self-adjoint operator $\Delta_{\mathbb{S}^1\times S_{\pm}}$.
Write $\beta$ as
\[\beta=\sum_{j=h_{p-1,\infty}+1}^{\infty}\beta_{p-1,j}(t)d_{F_\infty}\phi_{p-1,j,\infty} +\sum_{j=1}^{\infty}\beta_{p,j}(t)\phi_{p,j,\infty},\]
then
\[(\mu_{p,j,\infty}-\frac{d^2}{dt^2})\beta_{p,j}=(\mu_{p-1,j,\infty}-\frac{d^2}{dt^2})\beta_{p-1,j}=0.\]
If $\mu_{p,j,\infty}\not=0$, $\beta_{p,j}$ is a linear combination of $r^{\pm\sqrt{\mu_{p,j,\infty}}}$. When $\mu_{p,j,\infty}=0$, $\beta_{p,j}$ is a linear combination of 1 and $t$. On the other hand, $\beta_{p-1,j}$ is always a linear combination of $r^{\pm\sqrt{\mu_{p-1,j,\infty}}}$. The result for $\alpha$ is similar.
\end{proof}

\begin{lemma}
(Poincar\'e lemma) Suppose that $\gamma\in W^{k,2}_{-\delta}(\Lambda^p([T,\infty)\times\mathbb{S}^1\times S_{\pm}))$ is closed. Then there exist $\tilde\gamma\in W^{k+1,2}_{-\delta}(\Lambda^{p+1}([T,\infty)\times\mathbb{S}^1\times S_{\pm}))$ and unique constants $\gamma_j$ such that
\[\gamma=d\tilde\gamma+\sum_{j=1}^{h_{p,\infty}}\gamma_j\phi_{p,j,\infty}.\]

Suppose that $\gamma$ is a closed form in $W^{k,2}_{\delta}(\Lambda^p([T,\infty)\times\mathbb{S}^1\times S_{\pm}))$ instead. Then there exists $\tilde\gamma\in W^{k+1,2}_{\delta}(\Lambda^{p+1}([T,\infty)\times\mathbb{S}^1\times S_{\pm}))$ such that $\gamma=d\tilde\gamma$.
\label{Poincare-lemma}
\end{lemma}
\begin{proof}
Write $\gamma$ as \[\begin{split}
\gamma=dt\wedge(\sum_{j=h_{p-2,\infty}+1}^{\infty}\alpha_{p-2,j}(t)d_{F_\infty}\phi_{p-2,j,\infty} +\sum_{j=1}^{\infty}\alpha_{p-1,j}(t)\phi_{p-1,j,\infty})\\
+\sum_{j=h_{p-1,\infty}+1}^{\infty}\beta_{p-1,j}(t)d_{F_\infty}\phi_{p-1,j,\infty} +\sum_{j=1}^{\infty}\beta_{p,j}(t)\phi_{p,j,\infty},
\end{split}\] then \[\begin{split}
0=&d_{[T,\infty)\times F_\infty}\gamma\\
=&dt\wedge(-\sum_{j=h_{p-1,\infty}+1}^{\infty}\alpha_{p-1,j}(t)d_{F_\infty}\phi_{p-1,j,\infty}\\ &+\sum_{j=h_{p-1,\infty}+1}^{\infty}\frac{d\beta_{p-1,j}}{dt}(t)d_{F_\infty}\phi_{p-1,j,\infty} +\sum_{j=1}^{\infty}\frac{d\beta_{p,j}}{dt}(t)\phi_{p,j,\infty})\\
+&\beta_{p,j}(t)d_{F_\infty}\phi_{p,j,\infty}.
\end{split}\]

So $\beta_{p,j}$ are constants. Moreover, they vanishes unless $j=1,2,...,h_{p,\infty}$. Define $\tilde\alpha_{p-2,j}(t)$ as
\[-\mu_{p-2,j,\infty}e^{\sqrt{\mu_{p-2,j,\infty}}t} \int_{t}^{\infty}e^{-2\sqrt{\mu_{p-2,j,\infty}}\tau}\int_{T}^{\tau} e^{\sqrt{\mu_{p-2,j,\infty}}s}\alpha_{p-2,j}(s)ds d\tau.\]
Define $\tilde\beta_{p-2,j}(t)$ as $\frac{1}{\mu_{p-2,j,\infty}}\frac{d}{dt}\tilde\alpha_{p-2,j}(t)$. Then it is easy to see that \[\frac{d}{dt}\tilde\alpha_{p-2,j}-\mu_{p-2,j,\infty}\tilde\beta_{p-2,j}=0\] and
\[-\tilde\alpha_{p-2,j}+\frac{d}{dt}\tilde\beta_{p-2,j}=\alpha_{p-2,j}.\]

When $\gamma\in W^{k,2}_{-\delta}(\Lambda^p([T,\infty)\times\mathbb{S}^1\times S_{\pm}))$, define $\tilde\gamma$ as
\[\begin{split}
\tilde\gamma=dt\wedge\sum_{j=h_{p-2,\infty}+1}^{\infty}\tilde\alpha_{p-2,j}(t)\phi_{p-2,j,\infty} +\sum_{j=h_{p-2,\infty}+1}^{\infty}\tilde\beta_{p-2,j}(t)d_{F_\infty}\phi_{p-2,j,\infty}\\ +\sum_{j=1}^{h_{p-1,\infty}}(\int_T^t\alpha_{p-1,j}(\tau)d\tau)\phi_{p-1,j,\infty} +\sum_{j=h_{p-1,\infty}}^{\infty}\beta_{p-1,j}(t)\phi_{p-1,j,\infty},
\end{split}\]
then $\gamma=d\tilde\gamma+\sum_{j=1}^{h_{p,\infty}}\beta_{p,j}(t)\phi_{p,j,\infty}$ and $d^*_{[T,\infty)\times F_\infty}\tilde\gamma=0$.

When $\gamma\in W^{k,2}_{\delta}(\Lambda^p([T,\infty)\times\mathbb{S}^1\times S_{\pm}))$, define $\tilde\gamma$ as
\[\begin{split}
\tilde\gamma=dt\wedge\sum_{j=h_{p-2,\infty}+1}^{\infty}\tilde\alpha_{p-2,j}(t)\phi_{p-2,j,\infty} +\sum_{j=h_{p-2,\infty}+1}^{\infty}\tilde\beta_{p-2,j}(t)d_{F_\infty}\phi_{p-2,j,\infty}\\ +\sum_{j=1}^{h_{p-1,\infty}}(-\int_t^\infty\alpha_{p-1,j}(\tau)d\tau)\phi_{p-1,j,\infty} +\sum_{j=h_{p-1,\infty}}^{\infty}\beta_{p-1,j}(t)\phi_{p-1,j,\infty},
\end{split}\]
then $\gamma=d\tilde\gamma$ and $d^*_{[T,\infty)\times F_\infty}\tilde\gamma=0$.

The estimate on $\tilde\gamma$ is standard.
\end{proof}

There is a natural map $e$ from the relative deRham cohomology group of $\tilde V_{\pm}$ to the absolute deRham cohomology group. By Section 6.4 of \cite{Melrose}, the image $e[\mathcal{H}^2_{\mathrm{dR,rel}}(\tilde V_{\pm})]$ is isomorphic to the space \[\frac{\{\gamma\in C_0^{\infty}(\Lambda^2(\tilde V_{\pm})), d\gamma=0\}}{\{\gamma\in C_0^{\infty}(\Lambda^2(\tilde V_{\pm})), \gamma=d\gamma', \gamma'\in \cap_{k=1}^{\infty}W^{k,2}_{-\delta}(\Lambda^1(\tilde V_{\pm}))\}}\]

The next goal is to show that

\begin{proposition}
Suppose that $\delta>0$ is small enough. Then the space \[e[\mathcal{H}^2_{\mathrm{dR,rel}}(\tilde V_{\pm})]=\frac{\{\gamma\in C_0^{\infty}(\Lambda^2(\tilde V_{\pm})), d\gamma=0\}}{\{\gamma\in C_0^{\infty}(\Lambda^2(\tilde V_{\pm})), \gamma=d\gamma', \gamma'\in \cap_{k=1}^{\infty}W^{k,2}_{-\delta}(\Lambda^1(\tilde V_{\pm}))\}}\]
is isomorphic to the space $\mathcal{H}^2_{b,\mathrm{Ho}}(V_{\pm})$ defined as
\[\{\gamma\in W^{k,2}_{-3-\delta,...,-3-\delta,\delta}(\Lambda^2(V_{\pm})), (d+d^*)\gamma=0\}.\]
\label{Hodge-theory-L2-harmonic-form}
\end{proposition}
\begin{proof}
The method in this proof is the combination of the results in Section 6.4 of Melrose's book \cite{Melrose}. Suppose that $\gamma\in C_0^{\infty}(\Lambda^2(\tilde V_{\pm}))$ is closed. Remark that $H^2(\mathbb{S}^2\times\mathbb{B}^4)=\mathbb{R}$. So $\gamma=d\gamma_x+\gamma_{2,1,x}\phi_{2,1,x}$ on $\mathbb{S}^2\times\mathbb{B}^4$ corresponding to $x\in V^{\mathrm{sing}}_{\pm}$, where $\phi_{2,1,x}$ is the pull back of the generator of $H^2(\mathbb{S}^2)$. Choose a cut-off function $\chi_x$ which is supported near $x$ and is 1 in a smaller neighborhood. Then \[\gamma'=\gamma-\sum_{x\in V^{\mathrm{sing}}_{\pm}}d(\chi_x\gamma_x)+\sum_{x\in V^{\mathrm{sing}}_{\pm}}\gamma_{2,1,x}\phi_{2,1,x}\] is a 2-form on $V_{\pm}$. It equals to $\gamma_{2,1,x}\phi_{2,1,x}$ near $x$.

Define $\mathcal{H}^p_{b,\mathrm{Ho}}(V_{\pm})$ as \[\{\gamma\in W^{k,2}_{-3-\delta,...,-3-\delta,\delta}(\Lambda^p(V_{\pm})), (d+d^*)\gamma=0\}\] for $p=0,2,4,6$ and define $\mathcal{H}^{\mathrm{even}}_{b,\mathrm{Ho}}(V_{\pm})$ as  \[\mathcal{H}^{\mathrm{even}}_{b,\mathrm{Ho}}(V_{\pm}):=\{\gamma\in W^{k,2}_{-3-\delta,...,-3-\delta,\delta}(\Lambda^{\mathrm{even}}(V_{\pm})), (d+d^*)\gamma=0\}.\] By Proposition \ref{Closeness-harmonic-form}, \[\mathcal{H}^{\mathrm{even}}_{b,\mathrm{Ho}}(V_{\pm})=\mathcal{H}^0_{b,\mathrm{Ho}}(V_{\pm}) \oplus\mathcal{H}^2_{b,\mathrm{Ho}}(V_{\pm})\oplus\mathcal{H}^4_{b,\mathrm{Ho}}(V_{\pm}) \oplus\mathcal{H}^6_{b,\mathrm{Ho}}(V_{\pm}).\]

The $L^2$ dual of \[d+d^*:W^{k,2}_{-2+\delta,...-2+\delta,-\delta}(\Lambda^{\mathrm{odd}}(V_{\pm}))\rightarrow W^{k-1,2}_{-3+\delta,...-3+\delta,-\delta}(\Lambda^{\mathrm{even}}(V_{\pm}))\] is \[d+d^*:W^{1-k,2}_{-3-\delta,...-3-\delta,\delta}(\Lambda^{\mathrm{even}}(V_{\pm}))\rightarrow W^{-k,2}_{-4-\delta,...-4-\delta,\delta}(\Lambda^{\mathrm{odd}}(V_{\pm})).\]
The kernel the the dual map is $\mathcal{H}^{\mathrm{even}}_{b,\mathrm{Ho}}(V_{\pm})$ by standard elliptic regularity. By the proof of Proposition \ref{Closeness-harmonic-form}, $\mathcal{H}^{\mathrm{even}}_{b,\mathrm{Ho}}(V_{\pm})\subset W^{k,2}_{-2+\delta,...-2+\delta,\delta}(\Lambda^{\mathrm{even}}(V_{\pm}))$. So \[W^{k-1,2}_{-3+\delta,...-3+\delta,-\delta}(\Lambda^{\mathrm{even}}(V_{\pm}))
=\mathcal{H}^{\mathrm{even}}_{b,\mathrm{Ho}}(V_{\pm})\oplus((d+d^*)(W^{k,2}_{-2+\delta,...-2+\delta,-\delta}))\]
by elliptic regularity.
Moreover,
\[\begin{split}
W^{k-1,2}_{-3+\delta,...-3+\delta,-\delta}&(\Lambda^2(V_{\pm}))
=\mathcal{H}^2_{b,\mathrm{Ho}}(V_{\pm})\\
&\oplus dW^{k,2}_{-2+\delta,...-2+\delta,-\delta}(\Lambda^1(V_\pm))
\oplus d^*W^{k,2}_{-2+\delta,...-2+\delta,-\delta}(\Lambda^3(V_\pm)).
\end{split}\]
In fact, it suffices to show that the intersection of $dW^{k,2}_{-2+\delta,...-2+\delta,-\delta}(\Lambda^1(V_\pm))$
and $d^*W^{k,2}_{-2+\delta,...-2+\delta,-\delta}(\Lambda^3(V_\pm))$ is the empty set. Choose any element $\gamma''$ in the intersection. It is harmonic. So by Theorem \ref{Polyhomogenous-L-totally-character} and Lemma \ref{Polyhomogenous-end}, it can be written as an element in $\mathcal{K}_\infty(0)$ plus an element in $W^{k-1,2}_{\delta}$ near infinity if $\delta$ is small enough. By the closeness and cocloseness, their are no $t\phi_{p,j,\infty}$ and $tdt\wedge \phi_{p-1,j,\infty}$ terms. By the exactness, coexactness and Lemma \ref{Poincare-lemma}, there are no $\phi_{p,j,\infty}$ and $dt\wedge \phi_{p-1,j,\infty}$ terms. So $\gamma''\in W^{k-1,2}_{\delta}$ near infinity. Using integration by parts, $\gamma''=0$.

There is a natural map $(r,(s_2,s_3))\rightarrow (s_2,(r,s_3))$ from $C(\mathbb{S}^2\times\mathbb{S}^3)$ to $\mathbb{S}^2\times\mathbb{B}^4$ outside the singular point. It induces a map from $W^{k-1,2}(\Lambda^1(\mathbb{S}^2\times\mathbb{B}^4))$ to $W^{k-1,2}_{-1-\delta}(\Lambda^1(C(\mathbb{S}^2\times\mathbb{S}^3)))$. Using this map, it is easy to see that the projection of $\gamma'$ to $\mathcal{H}^2_{b,\mathrm{Ho}}(V_{\pm})$ is a well-defined map from $e[\mathcal{H}^2_{\mathrm{dR,rel}}(\tilde V_{\pm})]$ to $\mathcal{H}^2_{b,\mathrm{Ho}}(V_{\pm})$. In order to show the injectivity, assume that $\gamma$ is mapped to 0. Then there exists $\gamma'''\in W^{k,2}_{-2+\delta,...-2+\delta,-\delta}(\Lambda^1(V_\pm))$ and $\gamma''''\in W^{k,2}_{-2+\delta,...-2+\delta,-\delta}(\Lambda^3(V_\pm))$ such that $\gamma'=d\gamma'''+d^*\gamma''''$. So $d\gamma'''$ is both exact and coexact on the end. So $d\gamma'''\in W^{k-1,2}_{\delta}$ near infinity as before. Using integration by parts, $d^*\gamma''''=0$. So $\gamma_{2,1,x}=0$ for all $x\in V_{\pm}^{\mathrm{sing}}$ using the uniqueness part of the analogy of Lemma \ref{Poincare-lemma} near $x$. Therefore $\gamma'=0$ near $x$. Since $H^1(S^2\times S^3)=0$, $\gamma'''=d\gamma'''_x$ near $x$. So \[\gamma'=d(\gamma'''-\sum_{x}d(\chi_x\gamma'''_x)).\]
Therefore, \[[\gamma]=[d(\gamma'''-\sum_{x}d(\chi_x\gamma'''_x))+\sum_{x}d(\chi_x\gamma_x)]=0\in e[\mathcal{H}^2_{\mathrm{dR,rel}}(\tilde V_{\pm})].\]

In order to show the surjectivity, pick any form $\tilde\gamma$ in $\mathcal{H}^2_{b,\mathrm{Ho}}(V_{\pm})$. By Proposition \ref{Closeness-harmonic-form}, $\tilde\gamma$ can be written as $\tilde\gamma=\tilde\gamma_{2,1,x}\phi_{2,1,x}+\tilde\gamma'_x$ near each $x\in V^{\mathrm{sing}}_{\pm}$, where $\tilde\gamma'_x\in W^{k,2}_{-2+\delta}$ near $x$. Since $\tilde\gamma'_x$ is closed and is in $W^{k,2}_{-2+\delta}$ near $x$, by the analogy of Lemma \ref{Poincare-lemma} near $x$, there exists $\tilde\gamma''_x$ such that $\tilde\gamma'_x=d\tilde\gamma''_x$ near $x$. So $\tilde\gamma'''=\tilde\gamma_{2,1,x}\phi_{2,1,x}+\tilde\gamma'_x-d(\chi_x\tilde\gamma''_x)$ is a well defined form on $\tilde V_{\pm}$. On the other hand, by Lemma \ref{Poincare-lemma}, there exists $\tilde\gamma''''$ such that $\tilde\gamma'''=d\tilde\gamma''''$ near infinity. Fix $\chi:\mathbb{R}\rightarrow[0,1]$ as a smooth function satisfying $\chi(s)=1$ for $s\le 1$ and $\chi(s)=0$ for $s\ge 2$, then $\tilde\gamma'''-d((1-\chi(t_{\pm}-T+2))\tilde\gamma'''')\in C^\infty_0$ and its image approaches $\tilde\gamma$ when $T$ goes to infinity. Therefore, the image of $e[\mathcal{H}^2_{\mathrm{dR,rel}}(\tilde V_{\pm})]$ is dense. Since $\mathcal{H}^2_{b,\mathrm{Ho}}(V_{\pm})$ is finite dimensional, the map is in fact surjective.
\end{proof}

\begin{corollary}
Suppose that $\delta>0$ is small enough and \[\gamma\in W^{k,2}_{-3-\delta,...,-3-\delta,\delta}(\Lambda^{\mathrm{even}}(V_{\pm})).\]
If $(d+d^*)\gamma=0$, then $\gamma=0$.
\label{Vanishing-harmonic-even-form}
\end{corollary}
\begin{proof}
By Proposition \ref{Closeness-harmonic-form}, $d\gamma_0=0$. So $\gamma_0$ is a constant. It vanishes because it decays at infinity. By Proposition \ref{Hodge-theory-L2-harmonic-form}, $\gamma_2=0$ because as in Proposition 5.38 of \cite{Kovalev}, the space $e[\mathcal{H}^2_{\mathrm{dR,rel}}(\tilde V_{\pm})]$ vanishes. By Hodge duality, $\gamma_4=\gamma_6=0$.
\end{proof}

Similarly, it is possible to prove the following:
\begin{proposition}
Suppose that $\delta>0$ is small enough. Choose $\chi:\mathbb{R}\rightarrow[0,1]$ as a smooth function satisfying $\chi(s)=1$ for $s\le 1$ and $\chi(s)=0$ for $s\ge 2$. Define $\chi_\infty=(1-\chi(t_{\pm}-T+1))$. Then the space $\mathcal{H}^2_{b-\mathrm{abs}}(V_{\pm})$ defined as
\[\{\gamma\in \mathrm{Span}\{\chi_{\infty}\phi_{2,j,\infty}\}_{j=1}^{h_{2,\infty}}\oplus W^{k,2}_{-3-\delta,...,-3-\delta,\delta}(\Lambda^2(V_{\pm})), (d+d^*)\gamma=0\}\] is isomorphic to $H^2_{\mathrm{dR,abs}}(\tilde V_{\pm})$. Define the space $\mathcal{H}^2_{b-\mathrm{rel}}(V_{\pm})$ as \[\{\gamma\in \mathrm{Span}\{\chi_\infty dt\wedge\phi_{1,j,\infty}\}_{j=1}^{h_{1,\infty}}\oplus W^{k,2}_{-3-\delta,...,-3-\delta,\delta}(\Lambda^2(V_{\pm})), (d+d^*)\gamma=0\}.\] Then the space $\mathcal{H}^2_{eb}(V_{\pm})$ defined as \[\{\gamma\in W^{k,2}_{-3-\delta,...,-3-\delta,-\delta}(\Lambda^2(V_{\pm})), (d+d^*)\gamma=0\}\] can be written as $\mathcal{H}^2_{eb}(V_{\pm})=\mathcal{H}^2_{b-\mathrm{abs}}(V_{\pm})\oplus\mathcal{H}^2_{b-\mathrm{rel}}(V_{\pm})$
\label{Hodge-theory}
\end{proposition}
\begin{proof}
Remark that \[H^2_{\mathrm{dR,abs}}(\tilde V_{\pm})\cong H^2_{\mathrm{dR,abs}}(\tilde V_{\pm}\cap\{t_{\pm}<T+1\}),\] where the isomorphism map is given by restriction. Given any form $\gamma$ in $H^2_{\mathrm{dR,abs}}(\tilde V_{\pm}\cap\{t_{\pm}<T+1\})$, using the fact that
\[H^2_{\mathrm{dR,abs}}(\tilde V_{\pm}\cap\{T<t_{\pm}<T+1\})=H^2([T,T+1]\times \mathbb{S}^1\times S_{\pm})=H^2(S_{\pm}),\]
$\gamma$ can be written as $\gamma=d\gamma_\infty+\phi_{\infty}$ on $t_{\pm}\in(T,T+1)$, where $\phi_\infty\in H^2(S_{\pm})$.
As in the proof of Proposition \ref{Hodge-theory-L2-harmonic-form}, \[\gamma'=\gamma-\sum_{x\in V^{\mathrm{sing}}_{\pm}}d(\chi_x\gamma'_x)+\sum_{x\in V^{\mathrm{sing}}_{\pm}}\gamma_{2,1,x}\phi_{2,1,x}-d(\chi_\infty\gamma_\infty)+\phi_\infty\] is a 2-form on $V_{\pm}$.
Using the decomposition \[\begin{split}
W^{k-1,2}_{-3+\delta,...-3+\delta,-\delta}&(\Lambda^2(V_{\pm}))
=dW^{k,2}_{-2+\delta,...-2+\delta,-\delta}(\Lambda^1(V_\pm))\\
&\oplus d^*W^{k,2}_{-2+\delta,...-2+\delta,-\delta}(\Lambda^3(V_\pm)),
\end{split}\]
as in the proof of Proposition \ref{Hodge-theory-L2-harmonic-form}, the projection of $\gamma''$ to the second component provides a well-defined isomorphic from $H^2_{\mathrm{dR,abs}}(\tilde V_{\pm})$ to $\mathcal{H}^2_{b-\mathrm{abs},\mathrm{Ho}}(V_{\pm})$. Remark that there is no $\mathcal{H}^2_{b,\mathrm{Ho}}(V_{\pm})$ component by Corolloary \ref{Vanishing-harmonic-even-form}.

Finally, given $\gamma\in\mathcal{H}^2_{eb}(V_{\pm})\subset W^{k-1,2}_{-3+\delta,...-3+\delta,-\delta}(\Lambda^2(V_{\pm}))$, its first component is exact and coclosed. So as in the proof of Proposition \ref{Hodge-theory-L2-harmonic-form}, it belongs to $\mathcal{H}^2_{b-\mathrm{rel}}(V_{\pm})$. On the other hand, its second component belongs to $\mathcal{H}^2_{b-\mathrm{abs}}(V_{\pm})$.
\end{proof}

\begin{proposition}
Suppose $\delta>0$ is small enough. Then the space
\[\mathcal{H}^2_{\mathrm{Ho}}(M):=\{\gamma\in W^{k,2}_{-3-\delta,...,-3-\delta}(\Lambda^2(M)), (d+d^*)\gamma=0\}.\] is isomorphic to $H^2_{\mathrm{dR,abs}}(\tilde M)$.
\label{Hodge-theory-compact}
\end{proposition}
\begin{proof}
It is proved similarly as Proposition \ref{Hodge-theory}. Since the manifold is compact, there is no need to do anything near infinity.
\end{proof}

Recall that $\tilde M$ is the gluing of $\tilde V_+$ and $\tilde V_-$ using $t_+=2T+1-t_-$. Define $t$ by $t=t_+-T-\frac{1}{2}=T+\frac{1}{2}-t_-$. Using the fact that $\tilde M=\{t<\frac{1}{2}\}\cup\{t>-\frac{1}{2}\}$, there is a long exact sequence for the cohomogology groups of $\tilde M$, $\{t<\frac{1}{2}\}$, $\{t>-\frac{1}{2}\}$ and $\{|t|<\frac{1}{2}\}$. In particular
\[H^2(\tilde M)\rightarrow H^2(\{t<\frac{1}{2}\})\oplus H^2(\{t>-\frac{1}{2}\})\rightarrow H^2(\{|t|<\frac{1}{2}\})\]
is exact. Remark that $\tilde V_+$ is isomorphic to $\tilde V_-$, so the map from $H^2(\{t>-\frac{1}{2}\})$ to $H^2(\{|t|<\frac{1}{2}\})$ is isomorphic to the map from $H^2(\{t<\frac{1}{2}\})$ to $H^2(\{|t|<\frac{1}{2}\})$. This map is injective by the proof of Proposition 5.38 of \cite{Kovalev}. It follows that the long exact sequence is reduced to
\[0\rightarrow H^1(\{|t|<\frac{1}{2}\})\rightarrow H^2(\tilde M)\rightarrow H^2(\{t<\frac{1}{2}\})\rightarrow 0\]
using the fact that $H^1(\{t<\frac{1}{2}\})\oplus H^1(\{t>-\frac{1}{2}\})=0$. By Proposition \ref{Hodge-theory} and Proposition \ref{Hodge-theory-compact}, it induces a natural map from $\mathcal{H}^2_{\mathrm{Ho}}(M)$ to $\mathcal{H}^2_{b-\mathrm{abs}}(V_+)$. Moreover, $\mathrm{dim}\mathcal{H}^2_{\mathrm{Ho}}(M)=\mathrm{dim}\mathcal{H}^2_{b-\mathrm{abs}}(V_+)+1$ because
\[H^1(\{|t|<\frac{1}{2}\},\mathbb{R})=H^1(\mathbb{S}^1\times S_+)=\mathbb{R}.\]

\begin{proposition}
There exists a map from $\mathcal{H}^{\mathrm{even}}_{\mathrm{Ho}}(M)$ to $\mathcal{H}^{\mathrm{even}}_{eb}(V_+)$. Moreover, suppose that $\gamma\in\mathcal{H}^{\mathrm{even}}_{\mathrm{Ho}}(M)$ is mapped to $\gamma'$, then if $T$ is large enough, \[(\gamma,\gamma'\chi(t_+-\frac{T}{2}))_{W^{k,2}_{-3-\delta,...,-3-\delta,-\delta}}\ge \frac{9}{10}||e^{-\delta(t+T+\frac{1}{2})}\gamma||_{W^{k,2}_{-3-\delta,...,-3-\delta}}^2.\]
\label{Asymptotic-kernel-approximation}
\end{proposition}
\begin{proof}
Map 1 to 1 and the volume form $*1$ to $*1$. The estimate is trivial for such components. Using Hodge star, it suffices to define the map for 2-forms. Assume that $\gamma$ is a 2-form and by normalization, $||e^{-\delta(t+T+\frac{1}{2})}\gamma||_{W^{k,2}_{-3-\delta,...,-3-\delta}}=1$.

Since the difference between the asymptotically cylindrical metric and the product metric on the cylinder is $O(e^{-\nu t_{\pm}})$, it is easy to see that \[||(d+d^*)_\infty\gamma||_{W^{k-1,2}(|t|<\frac{T}{2})}\le Ce^{-\frac{\nu T}{2}+\frac{3\delta T}{2}},\] where $(d+d^*)_\infty$ is the operator $d+d^*$ defined using the product metric on the cylinder.

Using generalized Fourier series, $\gamma$ can be written as $dt\wedge\alpha+\beta+\gamma''$ in $|t|<\frac{T}{2}$, where $\alpha\in \mathcal{H}^1(t=0)\cong\mathbb{R}$, $\beta\in \mathcal{H}^2(t=0)\cong H^2(S_{\pm})$ and $\gamma''$ is an exact form satisfying $||\gamma''||_{W^{k,2}(|t|<\frac{1}{2})}\le Ce^{-\frac{\nu T}{4}}$ if both $\nu$ and $\frac{\delta}{\nu}$ are small enough. Choose $\gamma'''$ such that $||\gamma'''||_{W^{k+1,2}(|t|<\frac{1}{2})}\le Ce^{-\frac{\nu T}{4}}$ and $d\gamma'''=\gamma''$ when $|t|<\frac{1}{2}$. It is clear that $\gamma-d((1-\chi(t+\frac{3}{2}))\gamma''')$ induces a form in $W^{k,2}_{-3-\delta,...,-3-\delta,-\delta}(V_+)$ which equals to $dt\wedge\alpha+\beta$ when $t_+>T+1$. Define $\gamma'_{b-\mathrm{abs}}$ as its $d^*W^{k+1,2}_{-2+\delta,...-2+\delta,-\delta}(\Lambda^3(V_+))$ component. It is easy to see that $\gamma'_{b-\mathrm{abs}}\in\mathcal{H}^2_{b-\mathrm{abs}}(V_+)$ and $\gamma'_{b-\mathrm{abs}}-\beta\in W^{k,2}_{\delta}(V_+)$ near infinity. By Proposition \ref{Hodge-theory} and Proposition \ref{Hodge-theory-compact}, $\gamma'_{b-\mathrm{abs}}$ is also the image of $\gamma$ using the restriction map from $H^2_{\mathrm{dR,abs}}(\tilde M)$ to $H^2_{\mathrm{dR,abs}}(\tilde V_+\cap \{t<\frac{1}{2}\})$.

On the other hand, $*_M\gamma-*_\infty(dt\wedge\alpha+\beta)$ can also be written as $d\gamma''''$ when $|t|<\frac{1}{2}$ for $||\gamma''''||_{W^{k+1,2}(|t|<\frac{1}{2})}\le Ce^{-\frac{\nu T}{4}}$. So
\[*_{V_+}(*_M\gamma-d((1-\chi(t+\frac{3}{2}))\gamma''''))\] also induces a form in $W^{k,2}_{-3-\delta,...,-3-\delta,-\delta}(V_+)$ which equals to $*_{V_+}*_\infty(dt\wedge\alpha+\beta)$ when $t_+>T+1$. Define $\gamma'_{b-\mathrm{rel}}$ as its $dW^{k+1,2}_{-2+\delta,...-2+\delta,-\delta}(\Lambda^1(V_+))$ component. It belongs to $\mathcal{H}^2_{b-\mathrm{rel}}$ and $\gamma'_{b-\mathrm{abs}}-dt\wedge\alpha\in W^{k,2}_{\delta}$ near infinity.

Define $\gamma'=\gamma'_{b,abs}+\gamma'_{b-\mathrm{rel}}$. Then $\gamma-d((1-\chi(t+\frac{3}{2}))\gamma''')$ can be written as $\gamma'+\gamma'''''$ for an element $\gamma'''''\in W^{k,2}_{-3-\delta,...,-3-\delta,\delta}$.

By weighted elliptic estimate, \[||\gamma'''''||_{W^{k,2}_{-3-\delta,...-3-\delta,\delta}}\le C||(d+d^*)\gamma'''''||_{W^{k-1,2}_{-4+\delta,...-4+\delta,\delta}}\le Ce^{-\frac{\nu T}{8}}.\]

So \[\begin{split}
||\gamma'||_{W^{k,2}_{-3-\delta,...,-\delta}(t_+<\frac{T}{2})}&\le ||\gamma||_{W^{k,2}_{-3-\delta,...,-\delta}(t_+<\frac{T}{2})} +||\gamma'''''||_{W^{k,2}_{-3-\delta,...,-\delta}(t_+<\frac{T}{2})}\\
&\le 1+Ce^{-\frac{\nu T}{8}}e^{\delta T}.
\end{split}\]
Using the fact that $\gamma'$ is asymptotic to $dt\wedge\alpha+\beta$,
\[||\gamma'||_{W^{k,2}_{-3-\delta,...,-\delta}(t_+>\frac{T}{2})}\le C e^{-\frac{\delta T}{2}} ||\gamma'||_{W^{k,2}_{-3-\delta,...,-\delta}(t_+<\frac{T}{2})}.\]
In particular, \[||\gamma||_{W^{k,2}(|t|\le\frac{1}{2})}\le C e^{\frac{\delta T}{2}}.\]
Using \[||(1-\chi(t+\frac{3}{2}))\gamma||_{W^{k,2}_{-3-\delta,...,\delta}(V_-)}\le C||(d+d^*)_{V_-}(1-\chi(t+\frac{3}{2}))\gamma||_{W^{k-1,2}_{-4-\delta,...,\delta}(V_-)},\]
it is easy to get the conclusion.
\end{proof}

\begin{proposition}
For large enough $T$, and $p=2,3$, there exists a linear map \[B_T:W^{k,2}_{-2+\delta,...-2+\delta}(\Lambda^pM)\rightarrow W^{k,2}_{-2+\delta,...-2+\delta}(\Lambda^pM)\]
such that $(d+d^*)B_T\gamma=d^*\gamma$ and \[||B_T\gamma||_{W^{k,2}_{-2+\delta,...-2+\delta}(\Lambda^pM)}\le
Ce^{3\delta T}||\gamma||_{W^{k,2}_{-2+\delta,...-2+\delta}(\Lambda^pM)}.\]
\label{Estimate-on-BT}
\end{proposition}
\begin{proof}
When $p=3$, define \[A_\pm: W^{k,2}_{-2+\delta,...-2+\delta,0}(\Lambda^{\mathrm{odd}}(V_{\pm}))\rightarrow W^{k-1,2}_{-3+\delta,...-3+\delta,0}(\Lambda^{\mathrm{even}}(V_{\pm}))\]
and \[A_T: W^{k,2}_{-2+\delta,...-2+\delta}(\Lambda^{\mathrm{odd}}(M))\rightarrow W^{k-1,2}_{-3+\delta,...-3+\delta}(\Lambda^{\mathrm{even}}(M))\]
as $A_{\pm}\gamma=e^{\mp\delta t_\pm}(d+d^*)(e^{\pm\delta t_\pm}\gamma)$ and $A_T\gamma=e^{-\delta t}(d+d^*)(e^{\delta t}\gamma)$. Define the asymptotic kernel as $\chi(t_\pm-\frac{T}{2})\mathrm{Ker} A_\pm$. The $L^2$-dual maps are
\[A^*_\pm: W^{1-k,2}_{-3-\delta,...-3-\delta,0}(\Lambda^{\mathrm{even}}(V_{\pm}))\rightarrow W^{-k,2}_{-4-\delta,...-4-\delta,0}(\Lambda^{\mathrm{odd}}(V_{\pm}))\]
and \[A_T^*: W^{1-k,2}_{-3-\delta,...-3-\delta}(\Lambda^{\mathrm{even}}(M))\rightarrow W^{-k,2}_{-4-\delta,...-4-\delta}(\Lambda^{\mathrm{odd}}(M))\]
defined as $A_{\pm}^*\gamma=e^{\pm\delta t_\pm}(d+d^*)(e^{\mp\delta t_\pm}\gamma)$ and $A_T^*\gamma=e^{\delta t}(d+d^*)(e^{-\delta t}\gamma)$.
$A_T$ induces a map $A_T'$ from the $L^2$ complement to the asymptotic kernels of $A_{\pm}$ to the $L^2$ complement to the asymptotic cokernels. By Proposition 4.2 of \cite{KovalevSinger}, $A_T'$ is bijective.

By Proposition \ref{Asymptotic-kernel-approximation}, there exists an injective map from $\mathcal{H}^2_{\mathrm{Ho}}(M)$ to $\mathcal{H}^2_{eb}(V_+)$. However, by Proposition \ref{Vanishing-harmonic-even-form} and the fact that $H^1(\mathbb{S}^1\times S_+)=\mathbb{R}$, the dimension of $\mathcal{H}^2_{rel}(V_+)$ is at most 1 but the dimension of $\mathcal{H}^2_{\mathrm{Ho}}(M)$ equals to the dimension of $\mathcal{H}^2_{abs}(V_+)$ plus 1. So the map from $\mathcal{H}^2_{\mathrm{Ho}}(M)$ to $\mathcal{H}^2_{eb}(V_+)$ is bijective. By Hodge duality, the map from $\mathcal{H}^{\mathrm{even}}_{\mathrm{Ho}}(M)$ to $\mathcal{H}^{\mathrm{even}}_{eb}(V_+)$ is also bijective. Therefore, using Proposition \ref{Asymptotic-kernel-approximation}, it is easy to see that the map $A_T$ from the $L^2$ complement of the asymptotic kernels of $A_{\pm}$ to the $L^2$ complement of the kernel of $A_T^*$ is also bijective.

For $\gamma\in W^{k,2}_{-2+\delta,...-2+\delta}(\Lambda^3M)$, $e^{-\delta t}d^*\gamma$ is in the $L^2$ complement of the kernel of $A_T^*$ by Proposition \ref{Closeness-harmonic-form}. So there exists an element $\gamma'$ in the $L^2$ complement of the asymptotic kernels of $A_{\pm}$ such that $A_T\gamma'=e^{-\delta t}d^*\gamma$. Let $\gamma''=e^{\delta t}\gamma'$, then it is easy to see that $d\gamma''_3+d^*\gamma''_5=0$ and $d^*\gamma''_3+d\gamma''_1=d^*\gamma$.
Using integration by parts, $d^*\gamma''_5=d\gamma''_1=0$. So $B_T\gamma=\gamma''_3$ in this case. The estimate for $B_T\gamma$ follows from Proposition 4.2 of \cite{KovalevSinger}.

When $p=2$, consider the Laplacian operator
\[\Delta:W^{k+1,2}_{-1+\delta,...-1+\delta}(\Lambda^1M)\rightarrow W^{k-1,2}_{-3+\delta,...-3+\delta}(\Lambda^1M).\]
The $L^2$ dual map is
\[\Delta:W^{1-k,2}_{-3-\delta,...-3-\delta}(\Lambda^1M)\rightarrow W^{-1-k,2}_{-5-\delta,...-3-\delta}(\Lambda^1M).\]
By Proposition \ref{Harmonic-form-decomposition}, Proposition \ref{No-log-term}, Proposition \ref{Obata}, and Proposition \ref{Killing-metric}, there is no element in $\mathcal{P}_i(\lambda_i)$ for $\mathrm{Re}\lambda_i\in(-4,0)$. So by standard elliptic regularity and Theorem \ref{Polyhomogenous-L-totally-character}, any element in the second kernel also lies in $W^{k,2}_{-2+\delta,...-2+\delta}(\Lambda^1M)$. So using integration by parts, it is closed and coclosed. Therefore, for any element $\gamma\in W^{k,2}_{-2+\delta,...-2+\delta}(\Lambda^2M)$, $d^*\gamma$ is $L^2$-perpendicular to the kernel of the second map. So it lies in the image of the first map. Let $\gamma'$ be its inverse. Then $dd^*\gamma'+d^*d\gamma'=d^*\gamma$. Using integration by parts, $dd^*\gamma'=0$. Define $B_T\gamma$ as $d\gamma'$, then $(d+d^*)B_T\gamma=d^*\gamma$. Moreover, integration by parts again implies that $||B_T\gamma||_{L^2}\le||\gamma||_{L^2}\le Ce^{\delta T}||\gamma||_{W^{k,2}_{-2+\delta,...-2+\delta}(\Lambda^2M)}$.

By standard elliptic regularity, \[||B_T\gamma||_{W^{k,2}_{-3,...,-3}}\le C(||B_T\gamma||_{L^2}+||\Delta B_T\gamma||_{W^{k-2,2}_{-5,...,-5}})\le Ce^{\delta T}||\gamma||_{W^{k,2}_{-2+\delta,...-2+\delta}}.\]
So the required estimate is obtained using Theorem \ref{Polyhomogenous-L-totally-character}.
\end{proof}

\begin{remark}
The $p=2$ case of Proposition \ref{Estimate-on-BT} is similar to Theorem A of \cite{Joyce}. The $p=3$ case of Proposition \ref{Estimate-on-BT} is similar to Proposition 5.40 of \cite{Kovalev} in the smooth twisted connected sum case.  However, the author is not able to understand the proof of Proposition 5.40 of \cite{Kovalev}. This does not affect the main result of \cite{Kovalev} because in the smooth case, Proposition 5.40 can be proved using Theorem A of Joyce's paper \cite{Joyce}. In the singular case, it is not possible to find analogy of Theorem A of Joyce's paper \cite{Joyce} in $p=3$ case. That is the reason to make full use of the Hodge theory in this paper.
\end{remark}

Using the identification $\gamma=d\theta\wedge\gamma_2+\gamma_3$, $B_T$ can also be viewed as a map from $\mathbb{S}^1$-invariant 3-form on $\mathbb{S}^1\times M$ to itself satisfying $(d+d^*)B_T\gamma=d^*\gamma$ and the estimate \[||B_T\gamma||_{W^{k,2}_{-2+\delta,...-2+\delta}(\Lambda^3(\mathbb{S}^1\times M))}\le
Ce^{3\delta T}||\gamma||_{W^{k,2}_{-2+\delta,...-2+\delta}(\Lambda^3(\mathbb{S}^1\times M))}.\]

The next goal is the improvement of the growth rate near each singularity.

\begin{proposition}
Given $x_i\in V_+^{\mathrm{sing}}\cup V_-^{\mathrm{sing}}$. For simplicity, denote $r_{x_i}$ by $r$. Then if $\gamma\in W^{k,2}_{-2+\delta}(\Lambda^3(\mathbb{S}^1\times(V_{\pm}\cap\{r\le r_{0,x_i}\})))$ is an $\mathbb{S}^1$-invariant 3-form with $(d+d^*)_{\mathbb{S}^1\times V_{\pm}}\gamma\in W^{k-1,2}_{-1+\delta}$, then $\gamma=\gamma_{\le 0}+\gamma_{>0}$ with
\[\begin{split}\gamma_{\le 0}&=\sum_{\sqrt{\mu_{0,j}+4}-4\in(-2+\delta,\delta)}
c_{j,1}d((d(r^{\sqrt{\mu_{0,j}+4}-2}\chi_i\phi_{0,j}))^\#\lrcorner\varphi)\\
&+\chi_i(c_2\mathrm{Re}\Omega+c_3\mathrm{Im}\Omega+c_4d\theta\wedge\omega),\end{split}\]
and
\[\begin{split}
\sum_{\sqrt{\mu_{0,j}+4}-4\in(-2+\delta,\delta)}|c_{j,1}|+|c_2|+|c_3|+|c_4|+||\gamma_{>0}||_{W^{k,2}_{\delta}}\\
\le C(||(d+d^*)\gamma||_{W^{k-1,2}_{-1+\delta}}+||\gamma||_{W^{k,2}_{-2+\delta}}),
\end{split}\]
where the norm is taking on $\Lambda^*(\mathbb{S}^1\times(V_{\pm}\cap\{r\le r_{0,x_i}\}))$.

Define $\gamma_{\ge 0}$ as \[\gamma_{\ge 0}=\gamma_{>0}+c_2\chi_i\mathrm{Re}\Omega +c_3\chi_i\mathrm{Im}\Omega+c_4\chi_id\theta\wedge\omega,\]
and $\gamma_{<0}$ as \[\gamma_{<0}=\sum_{\sqrt{\mu_{0,j}+4}-4\in(-2+\delta,\delta)}
c_{j,1}d((d(r^{\sqrt{\mu_{0,j}+4}-2}\chi_i\phi_{0,j}))^\#\lrcorner\varphi).\]
\label{Polyhomogenous-doubling}
\end{proposition}
\begin{proof}
Write $\gamma$ as $\gamma=d\theta\wedge\gamma_2+\gamma_3$. Then $(d+d^*)_{V_\pm}\gamma_p=0$ for $p=2,3$. Consider the Laplacian operator acting on 2-forms or 3-forms on $V_{\pm}$. By Theorem \ref{Polyhomogenous-L-totally-character}, $\gamma=\gamma_{>0}+\gamma_{\le0}$, where $\gamma_{>0}\in W^{k,2}_\delta$ and $\gamma_{\le0}$ is a linear combination of $\chi_i\mathcal{P}_i(\lambda_i)$ for $\mathrm{Re}\lambda_i\in(-2+\delta,\delta)$. By Proposition \ref{No-log-term}, any element in such $\mathcal{K}_i(\lambda_i)$ is a homogenous harmonic form on $C$ with real rate $\lambda_i$. Consider the lowest critical rate $\lambda_i\in(-2+\delta,\delta)$. Using the fact that $(d+d^*)\gamma\in W^{k-1,2}_{-1+\delta}$, the element in $\mathcal{K}_i(\lambda_i)$ corresponding to the $\chi_i\mathcal{P}_i(\lambda_i)$ component of $\gamma_{\le0}$ must be closed and coclosed. By Corollary \ref{Two-form-and-three-form-representation}, it must be a linear combination of $d_C((d_C(r^{\lambda_i+2}\phi_{0,j}))^{\#_C}\lrcorner\varphi_C)$ with $\lambda_i\in(-1,0)$ and $\mu_{0,j}=(\lambda_i+4)^2-4\in(5,12)$. Using the fact that $\phi_{0,j}$ is pluriharmonic, it is a harmonic function both on $C$ and $V_{\pm}$ near $x_i$. So $d_{V_{\pm}}((d_{V_{\pm}}(r^{\lambda_i+2}\phi_{0,j}))^{\#_{V_{\pm}}}\lrcorner\varphi_{V_{\pm}})$ is also closed and coclosed on $V_{\pm}$ near $x_i$. So it can be redefined as an element in $\mathcal{P}_i(\lambda_i)$. Then the problem for the second lowest critical $\lambda_i\in(-2+\delta,\delta)$ is similar. By induction, the problem is reduced to the $\lambda_i=0$ case. In this case, $\mathrm{Re}\Omega$, $\mathrm{Im}\Omega$ and $d\theta\wedge\omega$ are in $\mathcal{K}_i(0)$. Moreover, $\phi_{0,j}$ may not be pluriharmonic. However, by choosing $\delta<\nu$, the difference between $\mathcal{P}_i(0)$ and $\mathcal{K}_i(0)$ can be absorbed into $\gamma_{>0}$.
\end{proof}

Recall that $\Theta(\varphi_T)$ is the Hodge dual of $\varphi_T$ using the metric defined by $\varphi_T$. Define $\tilde\Theta(\varphi_T)$ as \[\tilde\Theta(\varphi_T)=\Theta(\varphi_T)-\omega_{T,\pm}\wedge\omega_{T,\pm} -d\theta_{\pm}\wedge\mathrm{Im}\Omega_{T,\pm}.\]
It is easy to see that $d\Theta(\varphi_T)=d\tilde\Theta(\varphi_T)$ and $\tilde\Theta(\varphi_T)$ is a form supported in the regions $t_{\pm}\in[T-1,T]$. The $W^{k,2}$ norm of $\tilde\Theta(\varphi_T)$ is $O(e^{-\nu T})$.

The following lemma is similar to Proposition 10.3.4 of \cite{JoyceBook}. The proof is omitted.
\begin{lemma}
Suppose that $\xi\in W^{k+2,2}_{1+\delta,...1+\delta}$ is a function supported in the $t_\pm<T-1$ region. It defines a 3-form $\gamma_{<0}$ by \[\gamma_{<0}=d((d\xi)^{\#_\pm}\lrcorner\varphi_\pm).\]
Remark that when $t_\pm<T-1$, $\varphi_T=\varphi_\pm$.
Suppose that $\gamma_{\ge0}\in W^{k+2,2}_{-\delta,...-\delta}$ is a 3-form on $M$. Define $\gamma$ as $\gamma_{<0}+\gamma_{\ge 0}$. Then as long as the norms of $\xi$ and $\gamma_{\ge0}$ in the corresponding spaces are small enough, $\delta$ is small enough and $T$ is large enough, the equation
\[(d+d^*_{\varphi_T})\gamma+*_{\varphi_T}d((1+\frac{1}{3}<\gamma_{\ge 0},\varphi_T>)\tilde\Theta(\varphi_T))
-*_{\varphi_T}dQ_{\varphi_T}(\gamma_{\ge 0})=0\]
implies that \[d\Theta(\varphi_T+\gamma_{\ge 0})=0.\]
Remark that $Q_{\varphi_T}$ in the equation means the non-linear term of $\Theta$ defined in Proposition \ref{Theta-function}.
\end{lemma}

Using the norm on $W^{k,2}_{\delta,...,\delta}\oplus(\oplus_{i=1}^{2N_2}(\mathbb{R}\chi_i\mathrm{Re}\Omega
\oplus\mathbb{R}\chi_i\mathrm{Im}\Omega\oplus\mathbb{R}\chi_id\theta\wedge\omega))$, it is easy to see that
$||*_{\varphi_T}Q_{\varphi_T}(\gamma)||\le ||\gamma||^2$ if $||\gamma||$ is small enough. By implicit function theorem, it is possible to find a solution of the equation \[(d+d^*_{\varphi_T})\gamma+*_{\varphi_T}d((1+\frac{1}{3}<\gamma_{\ge 0},\varphi_T>)\tilde\Theta(\varphi_T))
-*_{\varphi_T}dQ(\gamma_{\ge 0})=0\] with $\gamma_{\ge0}\in W^{k,2}_{\delta,...,\delta}\oplus(\oplus_{i=1}^{2N_2}(\mathbb{R}\chi_i\mathrm{Re}\Omega
\oplus\mathbb{R}\chi_i\mathrm{Im}\Omega\oplus\mathbb{R}\chi_id\theta\wedge\omega))$. So $\varphi_T+\gamma_{\ge 0}$ provides the required $\mathbb{S}^1$-invariant torsion-free G$_2$ structure on $\mathbb{S}^1\times M$, or equivalently, the Calabi-Yau threefold structure on $M$.

\section{The obstruction of the singular twisted connected sum construction}

In this section, the first goal is to prove the analogy of Theorem \ref{Polyhomogenous-L-totally-character} as a refined version of Theorem 7.14 of \cite{Mazzeo}.

Let $C$ be the nodal cone as in Example \ref{Nodal-singularity}. Consider the Laplacian operator $\Delta$ acting on $p$-forms on $\mathbb{S}^1\times C$. Let $\theta$ be the standard variable on $\mathbb{S}^1$. Then any $p$-form can be expressed as
\[\gamma=d\theta\wedge\alpha+\beta,\]
where $\alpha$ is a $(p-1)$-form on $C$ and $\beta$ is a $p$-form on $C$. A direct calculation shows that
\[\begin{split}
\Delta_{\mathbb{S}^1\times C}\gamma=d\theta\wedge(\Delta_C\alpha-\frac{\partial^2}{\partial\theta^2}\alpha)
+(\Delta_C\beta-\frac{\partial^2}{\partial\theta^2}\beta),
\end{split}\]
where $\Delta_C$ means $\Delta$ operator on each slice.

By Mazzeo \cite{Mazzeo}, up to a sign, the corresponding operator $\Delta_0$ on $(p-1)$-forms $\alpha$ and $p$-forms $\beta$ on $C$ is given by
\[\Delta_0(\alpha,\beta)=(\Delta_C+1)(\alpha,\beta),\]
where roughly speaking, $\frac{\partial}{\partial\theta}$ is replaced by $i=\sqrt{-1}$. Still by \cite{Mazzeo}, up to a sign, the operator $I(\Delta)$ on forms $(\alpha,\beta)$ on $C$ is given by
\[I(\Delta)(\alpha,\beta)=(\Delta_C\alpha,\Delta_C\beta),\]
where roughly speaking, $\frac{\partial}{\partial\theta}$ is deleted. $\delta$ is called critical if it is critical for $I(\Delta)$.

The solution of $L_0\gamma=0$ is related to the Bessel function.
\begin{definition}
The Bessel I-function is defined by \[I_{\mu}(r)=\sum_{m=0}^{\infty}\frac{1}{m! \Gamma(m+\mu+1)}(\frac{r}{2})^{2m+\mu}.\]
The Bessel K-function is defined by \[K_{\mu}(r)=\frac{\pi}{2}\frac{I_{-\mu}(r)-I_{\mu}(r)}{\sin\mu\pi}\] if $\mu\not\in\mathbb{Z}$. When $\mu\in\mathbb{Z}$, the limit $\lim_{\hat\mu\rightarrow\mu}K_{\hat\mu}(r)$ exists and is defined as $K_\mu(r)$. In either cases, $I_{\mu}(r)$ and $K_{\mu}(r)$ are two independent solutions to the modified Bessel equation
\[((r{\frac{d}{dr}})^2-(r^2+\mu^2))y=r^{2}{\frac {d^2y}{dr^2}}+r\frac{dy}{dr}-(r^2+\mu^2)y=0.\]
\end{definition}

The following proposition is well known.

\begin{proposition}
(1) When $r$ goes to infinity, \[I_{\mu}(r)=\frac{1}{\sqrt{2\pi r}}e^r(1+O(\frac{1}{r})),\]
and \[K_{\mu}(r)=\sqrt{\frac{\pi}{2r}}e^{-r}(1+O(\frac{1}{r})).\]

(2) When $r$ goes to 0, \[\lim_{r\rightarrow 0}r^{-\mu}I_{\mu}(r)=\frac{1}{\Gamma(\mu+1)}(\frac{1}{2})^{\mu}.\]
On the other hand, if $\mu>0$, then \[\lim_{r\rightarrow 0}r^{\mu}K_{\mu}(r) =\lim_{\tilde\mu\rightarrow\mu}\frac{\pi}{2\Gamma(-\tilde\mu+1)\sin\tilde\mu\pi}(\frac{1}{2})^{-\mu}.\]
If $\mu=0$, then \[\lim_{r\rightarrow 0}K_0(r)(\log r)^{-1}=-1.\]

(3) $K_0'(r)=-K_1(r)$.
\end{proposition}

\begin{proposition}
Suppose that $\gamma\in W^{k,2}_\delta(\Lambda^p(C))$ for $\delta>-2$. If $(\Delta+1)\gamma=0$, then $\gamma=0$.
\label{Underline-delta}
\end{proposition}
\begin{proof}
Write $\gamma$ as the generalized Fourier series \[\gamma=\sum_{j=1}^{\infty}\gamma_j(r)\hat\phi_{p,j}.\]
The equation is reduced to the ordinary differential equations
\[((r{\frac{d}{dr}})^2-(r^2+\hat\mu_{p,j}))\gamma_j=0.\] However, any linear combination of $I_{\sqrt{\hat\mu_{p,j}}}(r)$ and $K_{\sqrt{\hat\mu_{p,j}}}(r)$ does not lie in $W^{k,2}_{\delta+2}$. Thus $\gamma=0$.
\end{proof}

Using the terminology of \cite{Mazzeo},  Proposition \ref{Underline-delta} implies that $\underline{\delta}\le -2$ for $\Delta$ acting on $p$-forms. An immediate corollary is the following:

\begin{corollary}
Suppose that $\gamma\in W^{k,2}_{\delta}(\Lambda^*(\mathbb{S}^1\times C))$ for a non-critical $\delta>-2$, then
\[||\gamma||_{W^{k,2}_{\delta}(\Lambda^*(\mathbb{S}^1\times C))}\le C||\Delta_{\mathbb{S}^1\times C}\gamma||_{W^{k-2,2}_{\delta-2}(\Lambda^*(\mathbb{S}^1\times C))}.\]
\label{Estimate-on-C-times-S1}
\end{corollary}
\begin{proof}
This corollary is essentially due to \cite{Mazzeo}. As in the proof of Theorem 5.16 of \cite{Mazzeo}, this estimate is obtained from Fourier transform, rescaling, applying the inverse of $\Delta+1$ and then the inverse Fourier transform.
\end{proof}

The next proposition is the key estimate for the ordinary differential equation involving Bessel functions.

\begin{proposition}
Assume that $\mu>0$. Suppose that $y(r)\in W^{k,2}_{\delta}((0,1))$ and $z\in W^{k-2,2}_{\delta'}((0,1))$ are functions on the interval $(0,1)$. They vanish in a neighborhood of $1$ and they satisfy the equation \[((r{\frac{d}{dr}})^2-(n^2r^2+\mu^2))y=z(r).\] Then

(1) If $-\mu<\delta<\delta'<\mu$, then $y\in L^2_{\delta'}$.

(2) If $\mu<\delta<\delta'$, then $y\in L^2_{\delta'}$.

(3) If $-\mu<\delta<\mu<\delta'$ and $n\not=0$, define $y_{\le\mu}(r)$ as \[-I_\mu(|n|r)\chi(2|n|r)\lim_{\tilde\mu\rightarrow\mu}\frac{\Gamma(\tilde\mu+1)\Gamma(-\tilde\mu+1)\sin \tilde\mu\pi}{\tilde\mu\pi}\int_0^1 K_\mu(|n|s)z(s)\frac{ds}{s}.\] Define $y_{>\mu}(r)$ as $y(r)-y_{\le\mu}(r)$, then $y_{>\mu}\in L^2_{\delta'}$.

(4) If $-\mu<\delta<\mu<\delta'$ and $n=0$, define $y_{\le\mu}(r)$ by \[y_{\le\mu}(r)=\chi(2r)r^{\mu}\int_0^1(\int_0^s t^{\mu}z(t)\frac{dt}{t})s^{-2\mu}\frac{ds}{s}.\] Define $y_{>\mu}(r)$ as $y(r)-y_{\le\mu}(r)$, then $y_{>\mu}\in L^2_{\delta'}$.
\label{Bessel-equation}
\end{proposition}
\begin{proof}
When $n\not=0$, \[\begin{split}
y(r)=&\lim_{\tilde\mu\rightarrow\mu}\frac{\Gamma(\tilde\mu+1)\Gamma(-\tilde\mu+1)\sin \tilde\mu\pi}{\tilde\mu\pi}(-I_\mu(|n|r)\int_r^1 K_\mu(|n|s)z(s)\frac{ds}{s}\\
&-K_\mu(|n|r)\int_0^r I_\mu(|n|s)z(s)\frac{ds}{s})+C_1I_\mu(|n|r)+C_2 K_{\mu}(|n|r)
\end{split}\] using the fact that \[I_\mu'(r)K_\mu(r)-K_\mu'(r)I_\mu(r)=
\lim_{\tilde\mu\rightarrow\mu}\frac{\tilde\mu\pi}{\Gamma(\tilde\mu+1)\Gamma(-\tilde\mu+1)\sin \tilde\mu\pi}\frac{1}{r}.\]
When $n=0$, \[y(r)=r^{\mu}\int_0^r(\int_0^s t^{\mu}z(t)\frac{dt}{t})s^{-2\mu}\frac{ds}{s}+C_1r^{\mu}+C_2r^{-\mu}.\]
\end{proof}

Assume that $\delta$ is small enough. Pick $x_i\in V_+^{\mathrm{sing}}$. Denote $r_{x_i}$ by $r$. Denote $r_{0,x_i}$ by $r_0$. Then the following theorem is an analogy of Proposition \ref{Polyhomogenous-doubling}:
\begin{theorem}

(1) Suppose that $-2<\delta'<\delta''<-1+\delta$ satisfy $\delta''-\delta'<\nu$. If \[\gamma\in W^{k,2}_{\delta'}(\Lambda^*(\mathbb{S}^1\times(V_+\cap\{r\le r_0\})))\]
and \[(d+d^*)_{\mathbb{S}^1\times V_+}\gamma\in W^{k-1,2}_{\delta''-1}(\Lambda^*(\mathbb{S}^1\times(V_+\cap\{r\le r_0\}))),\] then $\gamma\in W^{k,2}_{\delta''}(\Lambda^*(\mathbb{S}^1\times(V_+\cap\{r\le r_0\})))$ and
\[||\gamma||_{W^{k,2}_{\delta''}}\le C(||(d+d^*)\gamma||_{W^{k-1,2}_{\delta''-1}}
+||\gamma||_{W^{k,2}_{\delta'}}),\]
where the norm is taking on $\Lambda^*(\mathbb{S}^1\times(V_+\cap\{r\le r_0\}))$.

(2) Suppose that $-1+\delta\le \delta'<\delta''<0$. If \[\gamma\in W^{k,2}_{\delta'}(\Lambda^3(\mathbb{S}^1\times(C\cap\{r\le r_0\})))\] and \[(d+d^*)_{\mathbb{S}^1\times C}\gamma\in W^{k-1,2}_{\delta''-1}(\Lambda^*(\mathbb{S}^1\times(C\cap\{r\le r_0\}))),\] then $\gamma=\gamma_{<\delta''}+\gamma_{\ge \delta''}$,
with
\[\gamma_{<\delta''}=\sum_{n=-\infty}^{\infty}\sum_{\sqrt{\mu_{0,j}+4}-4\in(\delta',\delta'')}
c_{n,j}d((d(r^{\sqrt{\mu_{0,j}+4}-2}\chi(\frac{4|n+\frac{1}{2}|r}{r_0})\phi_{0,j}e^{in\theta}))^\#\lrcorner\varphi),\]
and
\[\begin{split}
\sqrt{\sum_{n=-\infty}^{\infty}\sum_{\sqrt{\mu_{0,j}+4}-4\in(\delta',\delta'')}
c_{n,j}^2|n+\frac{1}{2}|^{8-2\sqrt{\mu_{0,j}+4}-2\delta''}}+||\gamma_{\ge\delta''}||_{W^{k,2}_{\delta''}}\\
\le C(||(d+d^*)\gamma||_{W^{k-1,2}_{\delta''-1}}+||\gamma||_{W^{k,2}_{\delta'}}),
\end{split}\]
where the norm is taking on $\Lambda^*(\mathbb{S}^1\times(C\cap\{r\le r_0\}))$.

(3) Suppose that $-1+\delta\le \delta'<0<\delta''<\frac{\delta}{2}$. If \[\gamma\in W^{k,2}_{\delta'}(\Lambda^3(\mathbb{S}^1\times(C\cap\{r\le r_0\})))\] and \[(d+d^*)_{\mathbb{S}^1\times C}\gamma\in W^{k-1,2}_{\delta''-1}(\Lambda^*(\mathbb{S}^1\times(C\cap\{r\le r_0\}))),\] then $\gamma=\gamma_{\le 0}+\gamma_{>0}$ with
\[\begin{split}
\gamma_{\le 0}&=\sum_{n=-\infty}^{\infty}\sum_{\sqrt{\mu_{0,j}+4}-4\in(\delta',\delta'')}c_{n,j,1} d((d(r^{\sqrt{\mu_{0,j}+4}-2}\chi(\frac{4|n+\frac{1}{2}|r}{r_0})\phi_{0,j}e^{in\theta}))^\#\lrcorner\varphi)\\
&+\sum_{n=-\infty}^{\infty}\chi(\frac{4|n+\frac{1}{2}|r}{r_0})e^{in\theta}
(c_{n,2}\mathrm{Re}\Omega+c_{n,3}\mathrm{Im}\Omega+c_{n,4}d\theta\wedge\omega),\end{split}\]
and
\[\begin{split}
\sqrt{\sum_{n=-\infty}^{\infty}\sum_{\sqrt{\mu_{0,j}+4}-4\in(\delta',\delta'')}
c_{n,j,1}^2|n+\frac{1}{2}|^{8-2\sqrt{\mu_{0,j}+4}-2\delta''}}\\
+\sqrt{\sum_{n=-\infty}^{\infty}(c_{n,2}^2+c_{n,3}^2+c_{n,4}^2)|n+\frac{1}{2}|^{-2\delta''}}
+||\gamma_{>0}||_{W^{k,2}_{\delta''}}\\
\le C(||(d+d^*)\gamma||_{W^{k-1,2}_{\delta''-1}}+||\gamma||_{W^{k,2}_{\delta'}}),
\end{split}\]
where the norm is taking on $\Lambda^*(\mathbb{S}^1\times(C\cap\{r\le r_0\}))$.

(4) Part (2) and (3) are also true if $C$ is replaced by $V_+$.
\label{Polyhomogenous-L}
\end{theorem}
\begin{proof}
(1) Using the cut-off function, assume that $\gamma$ is supported in $r<r_0$ and use the metric $g_C$ to define $d+d^*$ and $\Delta$ instead of $g_{V_+}$. Write $\gamma$ and $\Delta\gamma$ as \[\gamma=\sum_{n=-\infty}^{\infty}\sum_{j=1}^{\infty}\gamma_{n,j}(r)\hat\phi_{3,j}e^{in\theta},\]
and \[\Delta\gamma=\sum_{n=-\infty}^{\infty}\sum_{j=1}^{\infty}\gamma'_{n,j}(r)\hat\phi_{3,j}e^{in\theta},\] then the equation is reduced to the ordinary differential equations
\[((r{\frac{d}{dr}})^2-(\hat\mu_{3,j}+n^2r^2)\gamma_{n,j}=r^2\gamma'_{n,j}.\] Remark that the rate of $\hat\phi_{3,j}$ is -2, so $\gamma_{n,j}\in L^2_{\delta'+2}$ and $r^2\gamma'_{n,j}\in L^2_{\delta''+2}$. By Proposition \ref{Bessel-equation}, if $\delta'+2<\sqrt{\hat\mu_{3,j}}<\delta''+2$, then $\gamma_{n,j}=\gamma_{n,j,\le\delta''+2}+\gamma_{n,j,>\delta''+2}$ with $\gamma_{n,j,>\delta''+2}\in L^2_{\delta''+2}$ and \[\begin{split}
\gamma_{n,j,\le\delta''+2} =-I_{\sqrt{\hat\mu_{3,j}}}(|n|r)\chi(\frac{|n|r}{r_0})\lim_{\tilde\mu\rightarrow\sqrt{\hat\mu_{3,j}}}
\frac{\Gamma(\tilde\mu+1)\Gamma(-\tilde\mu+1)\sin\tilde\mu\pi}{\tilde\mu\pi}\\
\int_0^1 K_{\sqrt{\hat\mu_{3,j}}}(|n|s)s^2\gamma'_{n,j}(s)\frac{ds}{s}
\end{split}\] if $n\not=0$, while \[\gamma_{n,j,\le\delta''+2}=\chi(\frac{r}{r_0})r^{\sqrt{\hat\mu_{3,j}}}\int_0^1(\int_0^s t^{\sqrt{\hat\mu_{3,j}}+2}\gamma'_{n,j}(t)\frac{dt}{t})s^{-2\sqrt{\hat\mu_{3,j}}}\frac{ds}{s}\] if $n=0$. Define \[c_{n,j}=\lim_{\tilde\mu\rightarrow\sqrt{\hat\mu_{3,j}}}
\frac{\Gamma(\tilde\mu+1)\Gamma(-\tilde\mu+1)\sin \tilde\mu\pi}{\tilde\mu\pi|n|}\int_0^1 K_{\sqrt{\hat\mu_{3,j}}}(|n|s)s^2\gamma'_{n,j}(s)\frac{ds}{s}\] for $n\not=0$, then
\[\begin{split}
|c_{n,j}|&\le C_{j,1}\sqrt{\int_0^1K_{\sqrt{\hat\mu_{3,j}}}^2(|n|s)s^{2\delta''+4}\frac{ds}{s}}\sqrt{\int_0^1 (s^2\gamma'_{n,j}(s))^2s^{-2\delta''-4}\frac{ds}{s}}\\
&\le C_{j,2}|n|^{-\delta''-2}||s^2\gamma'_{n,j}(s)||_{L^2_{\delta''+2}},
\end{split}\]
where \[C_{j,2}=\lim_{\tilde\mu\rightarrow\sqrt{\hat\mu_{3,j}}}
\frac{\Gamma(\tilde\mu+1)\Gamma(-\tilde\mu+1)\sin \tilde\mu\pi}{\tilde\mu\pi}\sqrt{\int_0^\infty K_{\sqrt{\hat\mu_{3,j}}}^2(s)s^{2\delta''+4}\frac{ds}{s}}.\]
Remark that replacing $I_{\sqrt{\hat\mu_{3,j}}}(|n|r)$ by its leading term $\frac{1}{\Gamma(\sqrt{\hat\mu_{3,j}}+1)}(\frac{|n|r}{2})^{\sqrt{\hat\mu_{3,j}}}$ does not affect the conclusion. By multiplying $(\frac{|n+\frac{1}{2}|}{|n|})^{\sqrt{\hat\mu_{3,j}}}$, it can be replaced by $\frac{1}{\Gamma(\sqrt{\hat\mu_{3,j}}+1)}(\frac{|n+\frac{1}{2}|r}{2})^{\sqrt{\hat\mu_{3,j}}}$ so that the $n=0$ case can be absorbed into the estimate.

Using the fact that \[(d+d^*)_{\mathbb{S}^1\times V_+}\gamma\in W^{k-1,2}_{\delta''-1}(\Lambda^*(\mathbb{S}^1\times(V_+\cap\{r\le r_0\}))),\] it suffices to consider $\hat\phi_{3,j}$ such that $\delta'+2<\sqrt{\hat\mu_{3,j}}<\delta''+2$ and $\hat\phi_{3,j}$ is both closed and coclosed. By Corollary \ref{Two-form-and-three-form-representation}, such form does not exist. Therefore, $\gamma_{n,j}=\gamma_{n,j,>\delta''+2}\in L^2_{\delta''+2}$ if $\delta'+2<\sqrt{\hat\mu_{3,j}}<\delta''+2$.

On the other hand, if $\delta'+2<\sqrt{\hat\mu_{3,j}}<\delta''+2$ is not true, by Proposition \ref{Bessel-equation}, it is also true that $\gamma_{n,j}\in L^2_{\delta''+2}$.

Apply Corollary \ref{Estimate-on-C-times-S1} for each term $\gamma_{n,j}(r)\hat\phi_{3,j}e^{in\theta}$. Using the fact that the Laplacian of each term are perpendicular to each other in any weighted $L^2$-norm, it is easy to see that \[||\gamma||_{L^2_{\delta''}}\le C||\Delta_C\gamma||_{L^2_{\delta''-2}}.\] By standard elliptic estimate
\[||\gamma||_{W^{k,2}_{\delta''}}\le C||\Delta_C\gamma||_{W^{k-2,2}_{\delta''-2}}.\]

(2) The proof is similar to (1).

(3) The proof is similar to (1).

(4) Choose non-critical $\delta_l$ so that $\delta'=\delta_1<...<\delta_{N_5}=\delta''$ and $\delta_l-\delta_{l-1}<\nu$. Assume that $\delta_{N_5-1}<0$. The statement is proved by induction. When $N_5=2$, the statement follows from (2) and (3) because $\delta''-\delta'<\nu$ in this case. Suppose that the statement has been proved for all $N_5<N_6$. The goal is to prove the statement for $N_5=N_6$.

By assumption, $\gamma=\gamma_{<\delta_{N_5-1}}+\gamma_{\ge\delta_{N_5-1}}$ with $\gamma_{<\delta_{N_5-1}}$ defined by
\[\sum_{n=-\infty}^{\infty}\sum_{\sqrt{\mu_{0,j}+4}-4\in(\delta',\delta_{N_5-1})}
c_{n,j}d((d(r^{\sqrt{\mu_{0,j}+4}-2}\chi(\frac{4|n+\frac{1}{2}|r}{r_0})\phi_{0,j}e^{in\theta}))^\#\lrcorner\varphi)\]
using the G$_2$ structure on $\mathbb{S}^1\times V_+$,
and
\[\begin{split}
\sqrt{\sum_{n=-\infty}^{\infty}\sum_{\sqrt{\mu_{0,j}+4}-4\in(\delta',\delta_{N_5-1})}
c_{n,j}^2|n+\frac{1}{2}|^{8-2\sqrt{\mu_{0,j}+4}-2\delta_{N_5-1}}}+||\gamma_{\ge \delta_{N_5-1}}||_{W^{k,2}_{\delta_{N_5-1}}}\\
\le C(||(d+d^*)_{\mathbb{S}^1\times V_+}\gamma||_{W^{k-1,2}_{\delta_{N_5-1}-1}}+||\gamma||_{W^{k,2}_{\delta'}}),
\end{split}\]
where the norm is taking on $\Lambda^*(\mathbb{S}^1\times(V_+\cap\{r\le r_0\}))$.

Consider $\gamma'_{<\delta_{N_5-1}}$ defined by \[\sum_{n=-\infty}^{\infty}\sum_{\sqrt{\mu_{0,j}+4}-4\in(\delta',\delta_{N_5-1})}
c_{n,j}d((d(r^{\sqrt{\mu_{0,j}+4}-2}\chi(\frac{4|n+\frac{1}{2}|r}{r_0})\phi_{0,j}e^{in\theta}))^\#\lrcorner\varphi)\]
using the G$_2$ structure on $\mathbb{S}^1\times C$ instead, then
\[\begin{split}
&(d+d^*)_{\mathbb{S}^1\times C}(\gamma-\gamma_{<\delta_{N_5-1}}+\gamma'_{<\delta_{N_5-1}})\\
&=((d+d^*)_{\mathbb{S}^1\times C}-(d+d^*)_{\mathbb{S}^1\times V_+})(\gamma-\gamma_{<\delta_{N_5-1}})\\
&+(d+d^*)_{\mathbb{S}^1\times V_+}\gamma+((d+d^*)_{\mathbb{S}^1\times C}\gamma'_{<\delta_{N_5-1}} -(d+d^*)_{\mathbb{S}^1\times V_+}\gamma_{<\delta_{N_5-1}}).
\end{split}\]
By Proposition \ref{Calabi-Yau-cone}, $r^{\sqrt{\mu_{0,j}+4}-2}\phi_{0,j}$ is pluriharmonic and therefore harmonic on both $C$ and $V_+$. So
\[\begin{split}
&(d+d^*)d((d(r^{\sqrt{\mu_{0,j}+4}-2}\phi_{0,j}e^{in\theta}))^\#\lrcorner\varphi)\\
&=\Delta((d(r^{\sqrt{\mu_{0,j}+4}-2}\phi_{0,j}e^{in\theta}))^\#\lrcorner\varphi)
-dd^*((d(r^{\sqrt{\mu_{0,j}+4}-2}\phi_{0,j}e^{in\theta}))^\#\lrcorner\varphi)\\
&=((d\Delta(r^{\sqrt{\mu_{0,j}+4}-2}\phi_{0,j}e^{in\theta}))^\#\lrcorner\varphi)
-d*d((d(r^{\sqrt{\mu_{0,j}+4}-2}\phi_{0,j}e^{in\theta}))\wedge*\varphi)\\
&=n^2((d(r^{\sqrt{\mu_{0,j}+4}-2}\phi_{0,j}e^{in\theta}))^\#\lrcorner\varphi)
\end{split}\]
on both $\mathbb{S}^1\times V_+$ and $\mathbb{S}^1\times C$. By estimates on each terms, it is easy to see that
\[\begin{split}
&||(d+d^*)_{\mathbb{S}^1\times C}(\gamma-\gamma_{<\delta_{N_5-1}} +\gamma'_{<\delta_{N_5-1}})||_{W^{k-1,2}_{\delta''-1}}\\
&\le C(||(d+d^*)_{\mathbb{S}^1\times V_+}\gamma||_{W^{k-1,2}_{\delta''-1}}+||\gamma||_{W^{k,2}_{\delta'}}).
\end{split}\]
The induction statement follows from (2) or (3) applied to $\gamma-\gamma_{<\delta_{N_5-1}}+\gamma'_{<\delta_{N_5-1}}$.
\end{proof}

The analogy of Proposition \ref{Polyhomogenous-doubling} has been proved. The next goal is to obtain the analogy of Proposition \ref{Closeness-harmonic-form}. The key point in the proof of Proposition \ref{Closeness-harmonic-form} is the fact that there is no $\log r\phi_{2,1}$ term in the $\mathbb{S}^1$-invariant case. However, in the singular twisted connected sum case, it is easy to see that for $n\not=0$, \[(d+d^*)_{\mathbb{S}^1\times C}(e^{in\theta}d\theta\wedge K_0(|n|r)\phi_{2,1}+\frac{i|n|}{n}e^{in\theta}K_1(|n|r)dr\wedge\phi_{2,1})=0.\] Remark that $K_0(|n|r)$ is asymptotic to $-\log(|n|r)$ when $|n|r$ is small. By comparison to the proof of Proposition \ref{Closeness-harmonic-form} and \ref{Polyhomogenous-L}, they provide the infinite dimensional obstruction space for the singular twisted connected sum construction.

The leading obstruction terms are \[e^{in\theta}d\theta\wedge -\log(|n|r)\phi_{2,1}+\frac{i|n|}{n}e^{in\theta}(|n|r)^{-1}dr\wedge\phi_{2,1}.\] Remark that the decay rates of $e^{in\theta}d\theta\wedge -\log(|n|r)\phi_{2,1}$ are $r^{-2}\log r$ while the decay rates of $\frac{i|n|}{n}e^{in\theta}(|n|r)^{-1}dr\wedge\phi_{2,1}$ are $r^{-3}$. So the leading obstruction term is an infinite dimensional linear combination of $e^{in\theta}(|n|r)^{-1}dr\wedge\phi_{2,1}$.

The obstruction comes from inverting the error term in some sense. Using Fourier series with respect to the first $\mathbb{S}^1$ factor, $M_+$ is $\mathbb{S}^1$-invariant. On the other hands, even though $M_-$ is $\mathbb{S}^1$-invariant with respect to the second $\mathbb{S}^1$ factor, the deviation of $M_-$ from being $\mathbb{S}^1$-invariant with respect to the first $\mathbb{S}^1$ factor decays exponentially. Moreover, the decay rate for the $e^{in\theta}$ factor is $O(e^{-|n|t_-})$. The leading term is the $n=\pm1$ case. So the leading obstruction term is a linear combination of $\cos\theta r^{-1}dr\wedge\phi_{2,1}$ and $\sin{\theta}r^{-1}dr\wedge\phi_{2,1}$. By rescaling and changing $\theta$ by a constant, the leading obstruction term is $\sin\theta r^{-1}dr\wedge\phi_{2,1}$.

Remark that the Calabi-Yau metric on the deformation $C_\epsilon$ of $C$ is asymptotically conically with leading error term $\epsilon r^{-1}dr\wedge\phi_{2,1}$ \cite{ConlonHein}, so roughly speaking, the obstruction will be resolved if each slice $\{\theta\}\times C$ of $\mathbb{S}^1\times C$ is deformed to $C_{\sin\theta}$. It is an analogy to the construction of Li \cite{Li}. There are singularities near $\theta=0$ and $\theta=\pi$. Topologically, when $\theta=0$ or $\theta=\pi$, the slice is $C(\mathbb{S}^2\times \mathbb{S}^3)$. The other slices are $\mathbb{B}^3\times \mathbb{S}^3$. The total space is $C(\mathbb{S}^3\times \mathbb{S}^3)$ near $\theta=0$ or $\theta=\pi$. In other words, there are strong evidences that the nodal singularity along $\mathbb{S}^1$ should be replaced by two isolated conical singularities with model $C(\mathbb{S}^3\times \mathbb{S}^3)$. Such problem will be left for future studies.

Remark that the main tool of Li's construction \cite{Li} is Yau's solution \cite{Yau} of the Calabi conjecture and its non-compact generalizations by Tian-Yau \cite{TianYau} and Hein \cite{Hein}. In the G$_2$ case, the analogy of Yau's theorem is not available, so one has to instead find extra structures to reduce the dimension.

\end{document}